\documentclass{article}

\usepackage{xspace} % To manipulate strings

\usepackage{amsmath,amsfonts,amssymb,mathrsfs}
\usepackage{dsfont} % For \mathds{1}
\numberwithin{equation}{section} % For the equation numbering
\allowdisplaybreaks % To allow page breaks inside equations

\usepackage{amsthm} % For more control over the theorem environment
\usepackage[dvipsnames]{xcolor} % For nice colors

\usepackage{hyperref}
\hypersetup{
    colorlinks,
    linkcolor={NavyBlue},
    citecolor={NavyBlue},
    urlcolor={black}}

\usepackage{graphicx}
\usepackage{caption}
\usepackage{subcaption}
\usepackage{tikz}
\usetikzlibrary{math}

\makeatletter
\@addtoreset{figure}{section}
\makeatother
\renewcommand{\thefigure}{\ifnum\value{section}>0
	\thesection.\fi\arabic{figure}}

\newlength{\bibitemsep}
\newlength{\bibparskip}
\let\oldthebibliography\thebibliography
\renewcommand\thebibliography[1]{\oldthebibliography{#1}
	\setlength{\parskip}{\bibitemsep}
	\setlength{\itemsep}{\bibparskip}}

\usepackage{geometry}
\usepackage{enumitem}
\setlist[enumerate,1]{label=(\roman*), font = \normalfont} 

\makeatletter
\newcommand\footnoteref[1]{\protected@xdef\@thefnmark{\ref{#1}}\@footnotemark}
\makeatother

\newtheorem{theorem}{Theorem}[section]

\newtheorem{lemma}[theorem]{Lemma}

\newtheorem{remark}[theorem]{Remark}

\newcommand{\C}{\mathbb{C}}
\newcommand{\E}{\mathbb{E}}
\newcommand{\N}{\mathbb{N}}
\renewcommand{\P}{\mathbb{P}}
\newcommand{\Q}{\mathbb{Q}}
\newcommand{\R}{\mathbb{R}}
\newcommand{\Z}{\mathbb{Z}}
\newcommand{\cD}{\mathcal{D}}
\newcommand{\cE}{\mathcal{E}}
\newcommand{\cF}{\mathcal{F}}
\newcommand{\cG}{\mathcal{G}}
\newcommand{\cN}{\mathcal{N}}
\newcommand{\cP}{\mathcal{P}}
\newcommand{\cT}{\mathcal{T}}
\newcommand{\ie}{i.e.\@\xspace}
\newcommand{\eg}{e.g.\@\xspace}
\newcommand{\as}{a.s.\@\xspace}
\newcommand{\iid}{i.i.d.\@\xspace}
\renewcommand{\tilde}{\widetilde}

\renewcommand{\d}[1]{\mathop{}\!\mathrm{d}#1}
\newcommand{\e}{\mathrm{e}}
\newcommand{\iu}{\mathrm{i}}
\newcommand{\Expec}[1]{\mathbb{E}\left[#1\right]}

\newcommand{\abs}[1]{\left\lvert#1\right\rvert}
\DeclareMathOperator{\Log}{Log}
\DeclareMathOperator{\re}{Re}

\let\originalleft\left
\let\originalright\right
\renewcommand{\left}{\mathopen{}\mathclose\bgroup\originalleft}
\renewcommand{\right}{\aftergroup\egroup\originalright}

\title{Asymptotics of the overlap distribution of branching Brownian motion at high temperature}

\author{Louis \textsc{Chataignier}%
\footnote{Institut de Mathématiques de Toulouse, UMR 5219, Université de Toulouse, CNRS, F-31062 Toulouse Cedex 9, France.}~
and Michel \textsc{Pain}$^*$}

\date{\today}

\begin{document}

\maketitle

\begin{abstract}
    At high temperature, the overlap of two particles chosen independently according to the Gibbs measure of the branching Brownian motion converges to zero as time goes to infinity.
    We investigate the precise decay rate of the probability to obtain an overlap greater than $a$, for some $a>0$, in the whole subcritical phase of inverse temperatures $\beta \in [0,\beta_c)$.
    Moreover, we study this probability both conditionally on the branching Brownian motion and non-conditionally.
    Two sub-phases of inverse temperatures appear, but surprisingly the threshold is not the same in both cases.
\end{abstract}

{
	\hypersetup{linkcolor=black} % avoid to have the table of contents in blue
	\tableofcontents
}

\section{Introduction}

The (binary) branching Brownian motion (BBM) is a continuous-time branching Markov process constructed as follows.
At time $t = 0$, a single particle starts a standard Brownian motion in $\R$ from the origin.
After a random time following an exponential distribution with parameter $1$, it splits in two, or equivalently, it dies while giving birth to two children.
These new particles then repeats the same process, independently of each other.
More precisely, the children perform independent Brownian motions from the position of their parent at its death and, after independent exponential times, they in turn split in two.
For a formal construction, see \cite{Chauvin1991}.

Branching Brownian motion has been widely studied over the last fifty years, initially for its link with reaction-diffusion equations \cite{McKean1975} which motivated a precise study of its extremes \cite{Bramson1978,Bramson1983,LalleySellke1987,AidekonBerestyckiBrunetShi2013,ArguinBovierKistler2013}.
Other important motivations for this model come from physics.
In spin glass theory, the study of disordered systems with a hierarchical structure of correlation for the energies dates back to the generalized random energy model introduced by Derrida and Gardner \cite{DerridaGardner1986} and includes the BBM (see \eg Bovier's book \cite{Bovier2016}).
The BBM can also be seen as an infinite dimensional limit of directed polymers \cite{DerridaSpohn1988} and is also related to diffractive scattering in high-energy particle physics \cite{Munier2009}.
See Section \ref{subsec:comments} for more details.
In all these frameworks coming from physics, understanding the overlap distribution of particles chosen according to the Gibbs measure plays an important role, and this paper is set in this context, by investigating the probability to have an atypical overlap in the high temperature phase.

\subsection{Definitions and previous results}

From a statistical physics perspective, for some fixed time $t \geq 0$, one can identify the positions of particles alive at time $t$ with the energy levels of a physical system.
Each particle can then be thought of as a configuration of the system.
Let $\cN(t)$ be the set of particles alive at time $t$ and $X_u(t)$ the position of particle $u$ at time $t$.
The \emph{Gibbs measure} of the BBM at time $t$ and at inverse temperature $\beta \geq 0$ is the (random) probability measure $\cG_{\beta, t}$ on $\cN(t)$ that assigns to each particle $u$ a weight proportional to $\e^{\beta X_u(t)}$.
By rescaling the associated partition function, we obtain the so-called \emph{additive martingale} of the BBM at inverse temperature $\beta$, introduced by McKean \cite{McKean1975},
\begin{equation}\label{eq:definition_additive_martingale}
	W_t(\beta) = \e^{-\psi(\beta)t} \sum_{u \in \cN(t)} \e^{\beta X_u(t)},
\end{equation}
where $\psi(\beta) = 1 + \beta^2/2$.
Since it is a positive martingale, it converges \as to a random variable $W_\infty(\beta)$.
This limit plays a key role in the study of the long-term behavior of the BBM, as illustrated in Theorem~\ref{thm:asymptotics_of_the_overlap_distribution} and Theorem~\ref{thm:asymptotics_in_mean} below.
A phase transition occurs: by \cite[Proposition~1]{Neveu1988}, if $\beta \geq \sqrt{2}$ then $W_\infty(\beta) = 0$ \as, if $\beta < \sqrt{2}$ then $W_\infty(\beta) > 0$ a.s.
In the latter case, the additive martingale is uniformly integrable, or equivalently, it converges in $L^1$.
In the critical case $\beta = \sqrt{2}$, it is more suitable to consider the \emph{derivative martingale}, introduced by Lalley and Sellke \cite{LalleySellke1987},
\begin{equation}\label{eq:definition_derivative_martingale}
    Z_t = \sum_{u \in \cN(t)} (\sqrt{2}t - X_u(t)) \e^{\sqrt{2}X_u(t) - 2t}.
\end{equation}
The authors showed that it converges \as to a random variable $Z_\infty$ and that $Z_\infty > 0$ \as
The critical additive martingale is related to the derivative martingale through the following convergence
\begin{equation}\label{eq:convergence_critical_additive_martingale}
    \sqrt{t} W_t(\sqrt{2}) \xrightarrow[t \to \infty]{} \sqrt{2\smash{/}\pi} Z_\infty, \quad \text{in probability},
\end{equation}
proved for the branching random walk (BRW) by A\"{i}dékon and Shi \cite{AidekonShi2014}.
For the BBM, this is a direct consequence of \cite[Proposition~2.2]{MaillardPain2019}.

In this paper, we are interested in the behavior of the overlap between particles chosen according to the Gibbs measure.
The \emph{overlap} between two particles $u, v \in \cN(t)$, denoted $q_t(u, v) \in [0,1]$, is defined as 
the last time where $u$ and $v$ have had a common ancestor alive, rescaled by a factor of $t$.
In other words, it is the proportion of time where $u$ and $v$ have shared a common trajectory.
Then, the \emph{overlap distribution} at time $t \geq 0$ and inverse temperature $\beta$ is the following random probability measure on $[0,1]$:
\begin{equation} \label{eq:def_overlap_dist}
    \nu_{\beta, t} = \cG_{\beta, t}^{\otimes 2}(q_t(u, v) \in \cdot) = \frac{1}{W_t(\beta)^2} \sum_{u, v \in \cN(t)} \e^{\beta(X_u(t) + X_v(t)) - 2 \psi(\beta) t} \delta_{q_t(u, v)},
\end{equation}
so that, for $a \in [0,1]$, $\nu_{\beta, t}([a,1])$ is the conditional probability, given the BBM, that two particles chosen independently according to $\cG_{\beta,t}$ have an overlap at least $a$.
This quantity was introduced by Derrida and Spohn \cite{DerridaSpohn1988} by analogy with the overlap between two configurations of spin glasses.
We are also interested in the \textit{mean overlap distribution} $\E[\nu_{\beta, t}]$, which is a deterministic probability measure on $[0,1]$ such that $\E[\nu_{\beta, t}([a,1])]$ is now the (non-conditional) probability that two particles chosen independently according to $\cG_{\beta,t}$ have an overlap at least $a$.

The overlap distribution satisfies the following phase transition:
\begin{equation} \label{eq:cv_overlap_distrib}
   	\nu_{\beta, t}
	\xrightarrow[t\to\infty]{} 
	\begin{cases}
		\delta_0,  
		& \text{if } \beta \leq \sqrt{2}, \\
		(1 - \pi_\beta) \delta_0 + \pi_\beta \delta_1, 
		&  \text{if } \beta > \sqrt{2},
	\end{cases} 
\end{equation}
in distribution for the weak topology, where $\pi_\beta$ is a random variable with values in $(0,1)$, mean $1-\sqrt{2}/\beta$ and a law which depends explicitely on $\beta$.
This result was conjectured by Derrida and Spohn \cite{DerridaSpohn1988}.
In the subcritical case, it was first proved by Chauvin and Rouault \cite{ChauvinRouault1997} for the BRW; more precisely, they proved the \as weak convergence of the rescaled measure $\nu_{\beta,t}(\cdot/t)$ on $[0,\infty)$ towards a non-trivial random measure.
In the critical case, this is a direct consequence of Bovier and Kurkova \cite{BovierKurkova2004} in the case of the binary Gaussian BRW, and it has been established for general BRW by Madaule \cite{Madaule2016}.
The supercritical case was treated by Mallein~\cite{Mallein2018} for the BRW (the convergence of the mean overlap distribution was proved earlier in \cite{BovierKurkova2004}).
Finally, Bonnefont~\cite{Bonnefont2022} proved the convergence~\eqref{eq:cv_overlap_distrib} for the BBM.
In this paper, we focus on the subcritical phase $\beta < \sqrt{2}$ and obtain the precise decay rate of the overlap distribution on $(0,1]$, as well as its limit, once rescaled by this rate.
We investigate separately the typical overlap distribution in Theorem~\ref{thm:asymptotics_of_the_overlap_distribution} and the mean overlap distribution in Theorem~\ref{thm:asymptotics_in_mean}.
Each theorem reveals two main regimes, though the threshold differs between them.

Theorem~\ref{thm:asymptotics_in_mean} involves a new probability measure $\Q_{\beta}$, introduced by Chauvin and Rouault \cite{ChauvinRouault1988}.
If we denote by $(\cF_t)_{t \geq 0}$ the natural filtration of the BBM, we can define it through the Radon--Nikodym derivatives
\begin{equation}\label{eq:definition_of_Q_b_without_spine}
    \frac{\d{\Q_{\beta}|_{\cF_t}}}{\d{\P|_{\cF_t}}} = W_t(\beta),
\end{equation}
for $t \geq 0$.
Under $\Q_{\beta}$, one of the particles, called \emph{spine}, performs a Brownian motion with drift $\beta$ and splits at rate $2$ while the behavior of the other particles remains unchanged, see \ref{subsec:change_of_measure} for details.
We denote by $\E_{\Q_\beta}$ the associated expectation.

Throughout this paper, the letters $t$ and $s$ denote some points in time $[0, \infty)$.
The letters $C$ and $c$ denote constants, \ie deterministic and time-invariant quantities, which may vary from one line to another.
We set $\N = \{0,1,2,\dots\}$ and $\N^* = \N \setminus \{0\}$.
We use the Bachmann--Landau notations $o$, $O$, $\sim$, and Hardy's notation $\asymp$ with their usual meaning.
More precisely, for $f \colon \R_+ \to \R$ and $g \colon \R_+ \to \R_+^*$, we say 
that $f(t) = o(g(t))$ as $t \to \infty$ if $\lim_{t\to\infty} f(t)/g(t) = 0$, 
that $f(t) = O(g(t))$ as $t \to \infty$ if $\limsup_{t\to\infty} \abs{f(t)}/g(t) < \infty$, 
that $f(t) \sim g(t)$ as $t \to \infty$ if $\lim_{t\to\infty} f(t)/g(t) = 1$, 
and that $f(t) \asymp g(t)$ as $t \to \infty$ if $0 < \liminf_{t\to\infty} f(t)/g(t) \leq \limsup_{t\to\infty} f(t)/g(t) < \infty$.

\subsection{Main results}

In the following theorems, we describe the asymptotic behavior at high temperature of the overlap distribution and its expectation.

\begin{theorem}\label{thm:asymptotics_of_the_overlap_distribution}
    Let $0 \leq \beta < \sqrt{2}$ and $0 < a < 1$.
    \begin{enumerate}
        \item\label{it:asymptotics_of_the_overlap_distribution_1} If $0 \leq \beta < \sqrt{2}/2$, then
        \begin{equation*}
            \e^{(1 - \beta^2)at} \nu_{\beta, t}([a, 1]) \xrightarrow[t \to \infty]{} \frac{W_\infty(2\beta)}{W_\infty(\beta)^2} \Expec{W_\infty(\beta)^2}, \quad \text{\as}
        \end{equation*}
        \item\label{it:asymptotics_of_the_overlap_distribution_2} If $\beta = \sqrt{2}/2$, then
        \begin{equation*}
            \sqrt{at} \e^{at/2} \nu_{\beta, t}([a, 1]) \xrightarrow[t \to \infty]{} \sqrt{\frac{2}{\pi}} \frac{Z_\infty}{W_\infty(\beta)^2} \Expec{W_\infty(\beta)^2}, \quad \text{in probability}.
        \end{equation*}
        \item\label{it:asymptotics_of_the_overlap_distribution_3} If $\sqrt{2}/2 < \beta < \sqrt{2}$, then
        \begin{equation*}
            (at)^{3 \beta/\sqrt{2}} \e^{(\sqrt{2} - \beta)^2 at} \nu_{\beta, t}([a, 1]) \xrightarrow[t \to \infty]{} \frac{C_\beta Z_\infty^{\sqrt{2}\beta}}{W_\infty(\beta)^2} S_{2\beta}, \quad \text{in distribution},
        \end{equation*}
        where $S_{2\beta}$ is a non-degenerate $(\sqrt{2}/2\beta)$-stable random variable independent of the BBM and $C_\beta > 0$ is a constant that depends only on $\beta$.
    \end{enumerate}
\end{theorem}

\begin{theorem}\label{thm:asymptotics_in_mean}
    Let $0 \leq \beta < \sqrt{2}$ and $0 < a < 1$.
    As $t \to \infty$,
    \begin{enumerate}
        \item\label{it:asymptotics_in_mean_1} if $\beta = 0$, then
        \begin{equation*}
        \Expec{\nu_{\beta, t}([a, 1])} \sim 2at \e^{-at},
        \end{equation*}
        \item\label{it:asymptotics_in_mean_2} if $0 < \beta < \sqrt{2\smash{/}3}$, then $W_t(\beta)$ converges $\P$-\as and $\Q_{2\beta}$-\as to some positive and finite variable $W_\infty(\beta)$, and
        \begin{equation*}
            \Expec{\nu_{\beta, t}([a, 1])} \sim \E_{\Q_{2\beta}} \left[ \frac{1}{W_\infty(\beta)^2} \right] \Expec{W_\infty(\beta)^2} \e^{-(1 - \beta^2)at},
        \end{equation*}
        \item\label{it:asymptotics_in_mean_3} if $\beta = \sqrt{2\smash{/}3}$, then
        \begin{equation*}
            \E \left[ \nu_{\beta, t}([a, 1]) \right] \asymp t^{-1/2} \e^{-at/3},
        \end{equation*}
        \item\label{it:asymptotics_in_mean_4} if $\sqrt{2\smash{/}3} < \beta < \sqrt{2}$, then
        \begin{equation*}
            \E \left[ \nu_{\beta, t}([a, 1]) \right] \asymp t^{-3/2} \e^{-(2 - \beta^2)^2at/8\beta^2}.
        \end{equation*}
    \end{enumerate}
\end{theorem}

The exponents appearing in the exponential decay of the overlap distribution in Theorems \ref{thm:asymptotics_of_the_overlap_distribution} and \ref{thm:asymptotics_in_mean} are represented in Figure \ref{fig:exponents}.
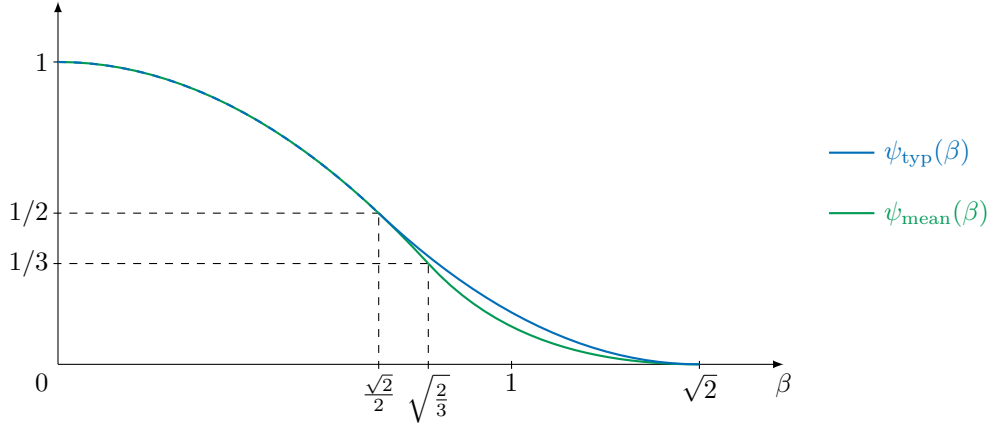
\begin{figure}
	\centering
	\begin{tikzpicture}[xscale=6,yscale=4]
		% Axes
		\draw[->,>=latex] (0,0) -- (1.6,0) node[below]{$\beta$};
		\draw[->,>=latex] (0,0) -- (0,1.2);
        \draw (0,0) node[below left]{$0$};
		% Traits axe des abscisses
		\draw[dashed] ({sqrt(2)/2},-.015) -- ({sqrt(2)/2},1/2);
		\draw ({sqrt(2)/2},0) node[below]{$\frac{\sqrt{2}}{2}$};
		\draw[dashed] ({sqrt(2/3)},-.015) -- ({sqrt(2/3)},1/3);
		\draw ({sqrt(2/3)},0) node[below]{$\sqrt{\frac{2}{3}}$};
        \draw[dashed] (1,-.015) -- (1,.015);
		\draw (1,0) node[below]{$1$};
		\draw[dashed] ({sqrt(2)},-.015) -- ({sqrt(2)},.015);
		\draw ({sqrt(2)},0) node[below]{$\sqrt{2}$};
        % Traits axe des ordonnées
        \draw[dashed] (-.01,1/3) -- ({sqrt(2/3)},1/3);
		\draw (0,1/3) node[left]{$1/3$};
        \draw[dashed] (-.01,1/2) -- ({sqrt(2)/2},1/2);
		\draw (0,1/2) node[left]{$1/2$};
		\draw (-.01,1) -- (.01,1);
		\draw (0,1) node[left]{$1$};
		% Mean overlap distribution
		\draw[domain=0:sqrt(2/3),ForestGreen,thick] plot (\x, 1-\x*\x);
		\draw[domain=sqrt(2/3):sqrt(2),ForestGreen,thick] plot (\x, {(2-\x*\x)*(2-\x*\x)/(8*\x*\x)});
		% Mean overlap distribution
		\draw[domain=0:sqrt(2)/2,NavyBlue,thick,dashed] plot (\x, 1-\x*\x);
		\draw[domain=sqrt(2)/2:sqrt(2),NavyBlue,thick] plot (\x, {(sqrt(2)-\x)*(sqrt(2)-\x)});
        % Légende
        \draw[NavyBlue,thick] (1.7,.7) -- (1.8,.7) node[right]{$\psi_{\text{typ}}(\beta) $};
        \draw[ForestGreen,thick] (1.7,.5) -- (1.8,.5) node[right]{$\psi_{\text{mean}}(\beta) $};
	\end{tikzpicture}
    \caption{Graph of the functions $\psi_{\text{typ}}$ and $\psi_{\text{mean}}$ such that, for any $a\in(0,1)$ and up to polynomial factors, 
    $\nu_{\beta, t}([a, 1])$ decays like $\e^{-\psi_{\text{typ}}(\beta) at}$ 
    and $\Expec{\nu_{\beta, t}([a, 1])}$ decays like $\e^{-\psi_{\text{mean}}(\beta) at}$.
    \label{fig:exponents}}
\end{figure}

\begin{remark} \label{rem:rewriting}
    If $0 \leq \beta < \sqrt{2}/2$, the additive martingale $W_t(2\beta)$ is uniformly integrable and therefore the measure $\Q_{2\beta}$ is uniformly continuous w.r.t.\@ $\P$ with Radon--Nikodym derivative $W_\infty(2\beta)$.    
    In particular, if $0 < \beta < \sqrt{2}/2$, we can rewrite Theorem~\ref{thm:asymptotics_in_mean}.\ref{it:asymptotics_in_mean_2} as
    \begin{equation}\label{eq:asymptotics_in_mean_2_reformulation}
        \e^{(1 - \beta^2)at} \Expec{\nu_{\beta, t}([a, 1])} \xrightarrow[t \to \infty]{} \Expec{\frac{W_\infty(2\beta)}{W_\infty(\beta)^2}} \Expec{W_\infty(\beta)^2}.
    \end{equation}
    Note that this limit is the expectation of the \as limit in Theorem~\ref{thm:asymptotics_of_the_overlap_distribution}.\ref{it:asymptotics_of_the_overlap_distribution_1}.
    This no longer holds if $\sqrt{2}/2 \leq \beta < \sqrt{2\smash{/}3}$ since the right-hand side of~\eqref{eq:asymptotics_in_mean_2_reformulation} becomes zero.
    The probability measure $\Q_{2\beta}$ then appears as the appropriate framework to extend this formula.
\end{remark}

\subsection{Comments, related literature and heuristics}
\label{subsec:comments}

\paragraph{Motivations.}
As mentioned earlier, our motivation for studying the overlap comes from the physics literature.
In mean-field spin glass models such as the Sherrington--Kirkpatrick model, the overlap between two configurations is defined as an explicit function of the correlation between the energies of these configurations, and the overlap distribution plays a key role in the interpretation of Parisi's formula for the free energy \cite{Parisi1983} and of the ultrametric structure of the Gibbs measure in the infinite-volume limit \cite{MezardParisiSourlasToulouseVirasoro1984}.
To understand the free energy and the overlap distribution on more tractable examples, Derrida and Gardner \cite{DerridaGardner1986} introduced the Generalized Random Energy Model (GREM), where the ultrametric structure is directly implemented in the definition of the Hamiltonian at finite volume: more precisely, the energies have a hierarchical structure of correlations with a fixed number of levels, whereas the BBM corresponds to a number of levels growing with the number of particles.
As indicated above, the overlap distribution for the BBM has been first studied in physics by Derrida and Spohn \cite{DerridaSpohn1988}, motivated by both spin glass and polymer theories. 
Indeed, the BBM can also be seen as an infinite limit of the directed polymer model and, in this framework, the overlap distribution describes the proportion of time where two polymers chosen independently according to the Gibbs measure are overlapping.

More recently, Derrida and Mottishaw \cite{DerridaMottishaw2016} investigated the finite-size corrections for the mean overlap distribution in the critical and supercritical phases $\beta \geq \sqrt{2}$: to be more precise than the known convergence $\E[\nu_{\beta,t}] \to \sqrt{2}/\beta \delta_0 + (1-\sqrt{2}/\beta) \delta_1$ as $t \to\infty$, they describe how the mass of $\E[\nu_{\beta,t}]$ between 0 and 1 vanishes. 
More precisely, they conjecture that, for any $0<a<b<1$,
\begin{equation} \label{eq:Derrida-Mottishaw}
    \E[\nu_{\beta,t}([a,b])] 
    \sim \frac{1}{\beta \sqrt{2 \pi t}} \int_a^b \frac{\d{x}}{x^{3/2} (1-x)^{(1+2\cdot \mathds{1}_{\beta>\sqrt{2}})/2}}.
\end{equation}
This prediction is based both on numerical simulations and on an approximation by the GREM, for which Derrida and Mottishaw \cite{DerridaMottishaw2018} calculated the finite-size corrections of the overlap as well.
In the case $\beta>\sqrt{2}$, the asymptotic~\eqref{eq:Derrida-Mottishaw} also has an interpretation in high-energy particle physics, uncovered by Mueller and Munier \cite{MuellerMunier2018a,MuellerMunier2018b}: using a link between the BBM and the scattering of high-energy hadrons dating back to \cite{IancuMuellerMunier2005}, they show a link between the statistics of rapidity gaps in diffractive electron-nucleus scattering and statistics of the extremal particles of the BBM, including their overlap distribution, which is well-predicted by~\eqref{eq:Derrida-Mottishaw} because supercritical Gibbs measures are supported by extremal particles.

In this paper, we study this question in the subcritical phase $\beta < \sqrt{2}$ instead, and investigate moreover the behavior of the typical overlap distribution.
Observe that the exponent appearing in the exponential decay of the mean overlap distribution is vanishing as $\beta$ approaches $\sqrt{2}$ (see Figure \ref{fig:exponents}), which is consistent with the fact that the decay is expected to be polynomial only for $\beta \geq \sqrt{2}$.
This is also the case for the typical overlap distribution, so we also expect a polynomial decay for $\beta \geq \sqrt{2}$ in that case.
However, we expect the typical overlap distribution to decay faster than the mean overlap distribution when $\beta \geq \sqrt{2}$, as it is the case for $\beta \in [\sqrt{2}/2,\sqrt{2})$, which means that $\E[\nu_{\beta,t}([a,b])]$ is governed by a rare event (called ``fluctuation'' in \cite{MuellerMunier2018a}).

\paragraph{Typical overlap distribution.}
Most heuristics can be deduced from the following rewriting of the overlap distribution (see~\eqref{eq:overlap_rewriting} for details)
\begin{equation}\label{eq:overlap_rewriting_intro}
    \nu_{\beta, t}([a, 1]) = \frac{\e^{(\beta^2 - 1) at}}{W_t(\beta)^2} \sum_{w \in \cN(at)} \e^{2\beta X_w(at) - \psi(2\beta)at} W_{t - at}^{(w, at)}(\beta)^2,
\end{equation}
where $(W_s^{(w, at)}(\beta))_{s \geq 0}$ is the additive martingale of the BBM starting from particle $w$ at time $at$.
Note that, in the sum of~\eqref{eq:overlap_rewriting_intro}, these additive martingales are tilted by Gibbs weights at inverse temperature $2\beta$ and time $at$.
Consequently, the particles $u,v\in\cN(t)$ that mainly contribute to the overlap distribution (see the expression in~\eqref{eq:def_overlap_dist}) are those contributing to $W_{t-at}^{(w, at)}(\beta)$, with $w$ itself being among the particles contributing to the Gibbs measure $\cG_{2\beta, at}$, or equivalently, to the additive martingale $W_{at}(2\beta)$.
The description of these particles' behaviors follows from this.

There is a first regime $\beta \in [0, \sqrt{2}/2)$ for which the Gibbs measure $\cG_{2\beta, at}$ is subcritical, in the sense that it is supported by an exponential number of particles.
These ones perform a Brownian motion with drift $2\beta$ (see \cite[Theorem~4.2]{Chataignier2024} for a weaker version of this property and \cite[Eq.\@ (1.14)]{Pain2018} for the case of the BRW).
So do the particles that contribute to the overlap distribution, up to time $at$, after which their trajectories relax to a drift $\beta$ (see Figure~\ref{fig:contributing_particles_1}).
The large amount of such particles induces a law of large numbers that allows to approximate~\eqref{eq:overlap_rewriting_intro} with
\begin{equation}\label{eq:overlap_approximation}
    \frac{\e^{(\beta^2 - 1) at}}{W_t(\beta)^2} \sum_{w \in \cN(at)} \e^{2\beta X_w(at) - \psi(2\beta)at} \Expec{W_\infty(\beta)^2} = \frac{\e^{(\beta^2 - 1) at}}{W_t(\beta)^2} W_{at}(2\beta) \Expec{W_\infty(\beta)^2}.
\end{equation}
Theorem~\ref{thm:asymptotics_of_the_overlap_distribution}.\ref{it:asymptotics_of_the_overlap_distribution_1} then follows from the convergence of the additive martingales.

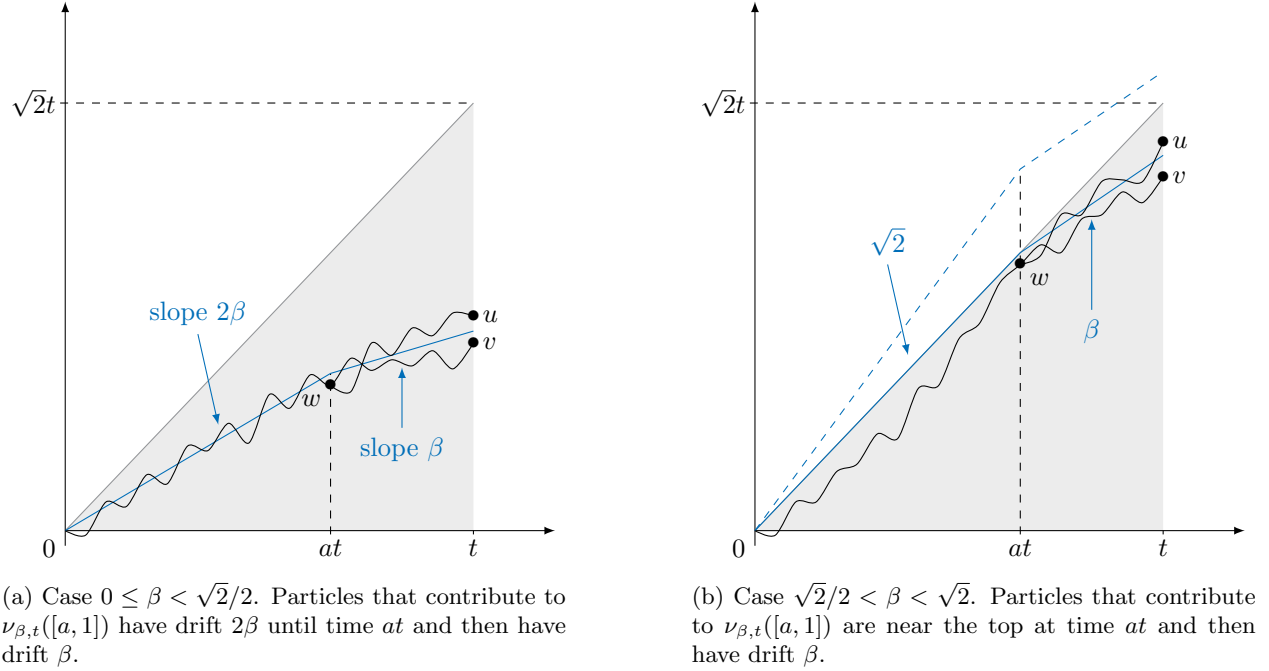
\begin{figure}
    \centering
    % Variables
    \tikzmath{\t=2;\a=.65;\bc=sqrt(2);\b1=.4;\b2=.92;\b3=1;}
    \begin{subfigure}[b]{.45\textwidth}
        \begin{tikzpicture}[xscale=2.7,yscale=2]
            % BBM
            \fill[fill=lightgray!30] (0,0) -- (\t,{\bc*\t}) -- (\t,0) -- (0,0);
            \draw[Gray] (0,0) -- (\t,{\bc*\t});
            % Axis
            \draw[->,>=latex] (0,0) -- (2.4,0);
            \draw[->,>=latex] (0,-.1) -- (0,3.5);
            \draw (0,0) node[below left]{$0$};
            % Lines from abscissa
            \draw[dashed] ({\a*\t},-.025) -- ({\a*\t},{2*\b1*\a*\t});
            \draw ({\a*\t},0) node[below]{$at$};
            \draw[dashed] (\t,-.025) -- (\t,0);
            \draw (\t,0) node[below]{$t$};
            % Lines from ordinate
            \draw[dashed] (-.025,{\bc*\t}) -- (\t,{\bc*\t});
            \draw (0,{\bc*\t}) node[left]{$\sqrt{2}t$};
            % Drift
            \draw[NavyBlue] (0,0) -- ({\a*\t},{2*\b1*\a*\t}) -- (\t,{\b1*\t*(1+\a)});
            \draw[<-,>=latex,NavyBlue] ({\a+.1},{2*\b1*(\a+.1)+.1}) -- (\a,{2*\b1*(\a+.1)+.7});
            \draw[NavyBlue] (\a,{2*\b1*(\a+.1)+.7}) node[above]{slope $2\beta$};
            \draw[<-,>=latex,NavyBlue] ({\t*(1+\a)/2},{(2*\b1*\a*\t+\b1*(\t-\a*\t)/2)-.1}) -- ({\t*(1+\a)/2},{(2*\b1*\a*\t+\b1*(\t-\a*\t)/2)-.5});
            \draw[NavyBlue] ({\t*(1+\a)/2},{(2*\b1*\a*\t+\b1*(\t-\a*\t)/2)-.5}) node[below]{slope $\beta$};
            % Common trajectory
            \draw[smooth] plot coordinates
            {(0,0) (.1,-.03) (.2,.19) (.3,.16) (.4,.37) (.5,.31) (.6,.56) (.7,.53) (.8,.71) (.9,.58) (1,.9) (1.1,.81) (1.2,1.03)
            % Trajectory 1
            (1.3,.96) (1.4,1.14) (1.5,1.06) (1.6,1.13) (1.7,1.09) (1.8,1.19) (1.9,1.07) (2,1.24)};
            % Trajectory 2
            \draw[smooth] plot coordinates
            {(1.3,.96) (1.4,.92) (1.5,1.24) (1.6,1.16) (1.7,1.34) (1.8,1.29) (1.9,1.44) (2,1.42)};
            % Particles
            \draw (2,1.42) node{$\bullet$};
            \draw (2,1.42) node[right]{$u$};
            \draw (2,1.24) node{$\bullet$};
            \draw (2,1.24) node[right]{$v$};
            \draw (1.3,.96) node{$\bullet$};
            \draw (1.3,.96) node[below left]{$w$};
        \end{tikzpicture}
        \caption{Case $0 \leq \beta < \sqrt{2}/2$. Particles that contribute to $\nu_{\beta, t}([a, 1])$ have drift~$2\beta$ until time $at$ and then have drift~$\beta$.\label{fig:contributing_particles_1}}
    \end{subfigure}
    \hfill
    \begin{subfigure}[b]{.45\textwidth}
        \begin{tikzpicture}[xscale=2.7,yscale=2]
            % BBM
            \fill[fill=lightgray!30] (0,0) -- (\t,{\bc*\t}) -- (\t,0) -- (0,0);
            \draw[Gray] (0,0) -- (\t,{\bc*\t});
            % Axis
            \draw[->,>=latex] (0,0) -- (2.4,0);
            \draw[->,>=latex] (0,-.1) -- (0,3.5);
            \draw (0,0) node[below left]{$0$};
            % Lines from abscissa
            \draw[dashed] ({\a*\t},-.025) -- ({\a*\t},{2*\b2*\a*\t});
            \draw ({\a*\t},0) node[below]{$at$};
            \draw[dashed] (\t,-.025) -- (\t,0);
            \draw (\t,0) node[below]{$t$};
            % Lines from ordinate
            \draw[dashed] (-.025,{\bc*\t}) -- (\t,{\bc*\t});
            \draw (0,{\bc*\t}) node[left]{$\sqrt{2}t$};
            % Drift
            \draw[dashed,NavyBlue] (0,0) -- ({\a*\t},{2*\b2*\a*\t}) -- (\t,{\b2*\t*(1+\a)});
            \draw[NavyBlue] (0,0) -- ({\a*\t},{\bc*\a*\t}) -- (\t,{(\bc*\a+\b2*(1-\a))*\t});
            \draw[<-,>=latex,NavyBlue] ({\a+.1},{\bc*(\a+.1)+.1}) -- (\a,{\bc*(\a+.1)+.7});
            \draw[NavyBlue] (\a,{\bc*(\a+.1)+.7}) node[above]{$\sqrt{2}$};
            \draw[<-,>=latex,NavyBlue] ({\t*(1+\a)/2},{(\bc*\a*\t+\b2*(\t-\a*\t)/2)-.1}) -- ({\t*(1+\a)/2},{(\bc*\a*\t+\b2*(\t-\a*\t)/2)-.7});
            \draw[NavyBlue] ({\t*(1+\a)/2},{(\bc*\a*\t+\b2*(\t-\a*\t)/2)-.7}) node[below]{$\beta$};
            % Common trajectory
            \draw[smooth] plot coordinates
            {(0,0) (.1,-.03) (.2,.19) (.3,.19) (.4,.39) (.5,.44) (.6,.64) (.7,.61) (.8,.94) (.9,.96) (1,1.27) (1.1,1.37) (1.2,1.64)
            % Trajectory 1
            (1.3,1.76) (1.4,1.89) (1.5,1.84) (1.6,2.06) (1.7,2.09) (1.8,2.24) (1.9,2.17) (2,2.34)};
            % Trajectory 2
            \draw[smooth] plot coordinates
            {(1.3,1.76) (1.4,1.82) (1.5,2.09) (1.6,2.09) (1.7,2.31) (1.8,2.32) (1.9,2.31) (2,2.57)};
            % Particles
            \draw (2,2.57) node{$\bullet$};
            \draw (2,2.57) node[right]{$u$};
            \draw (2,2.34) node{$\bullet$};
            \draw (2,2.34) node[right]{$v$};
            \draw (1.3,1.76) node{$\bullet$};
            \draw (1.3,1.76) node[below right]{$w$};
        \end{tikzpicture}
        \caption{Case $\sqrt{2}/2 < \beta < \sqrt{2}$. Particles that contribute to $\nu_{\beta, t}([a, 1])$ are near the top at time $at$ and then have drift~$\beta$.\label{fig:contributing_particles_2}}
    \end{subfigure}
    \caption{Behaviors of the particles that contribute typically to the overlap distribution.}
\end{figure}

At $\beta = \sqrt{2}/2$, the Gibbs measure $\cG_{2\beta, at}$ is critical.
It is supported by a number of order $\e^{C\sqrt{t}}$ of particles, having a drift $\sqrt{2} = 2\beta$ altered by an upper barrier of slope $\sqrt{2}$ (see \cite{Madaule2016} for the BRW).
The same holds up to time $at$ for the overlap distribution.
Theorem~\ref{thm:asymptotics_of_the_overlap_distribution}.\ref{it:asymptotics_of_the_overlap_distribution_2} can be seen as a consequence of the approximation~\eqref{eq:overlap_approximation} and the convergence~\eqref{eq:convergence_critical_additive_martingale}.

In the remaining regime $\beta \in (\sqrt{2}/2, \sqrt{2})$, the Gibbs measure $\cG_{2\beta, at}$ is supported by the extremal process, \ie particles at a bounded distance from the top of the BBM.
Up to logarithmic corrections, their trajectories are Brownian bridges from $(0,0)$ to $(at,\sqrt{2}at)$ conditioned to stay under a barrier of slope $\sqrt{2}$ (see \cite{ChenMadauleMallein2019} for the BRW).
This again induces the behavior of the particles that contribute to the overlap distribution (see Figure~\ref{fig:contributing_particles_2}).
The $\alpha$-stable limit with $\alpha = \sqrt{2}/2\beta \in (1/2, 1)$ is reminiscent of the limit appearing for the renormalized supercritical additive martingale, obtained by Barral, Rhodes and Vargas \cite{BarralRhodesVargas2018} for the BRW (see \cite[Remark~2.7]{Bonnefont2022} for the BBM case): for $\beta > \sqrt{2}$,
\begin{equation}\label{eq:supercritical_W}
    t^{3\beta/2\sqrt{2}} \e^{(\beta/\sqrt{2}-1)^2t} W_t(\beta)
    \xrightarrow[t\to\infty]{} Z_\infty^{\beta/\sqrt{2}} S_\beta, \quad \text{in distribution},
\end{equation}
where $S_{\beta}$ is a non-degenerate $\alpha$-stable random variable independent of $Z_\infty$, with $\alpha = \sqrt{2}/\beta \in (0,1)$.
Here, it is not exactly $W_{at}(2\beta)$ which appears in~\eqref{eq:overlap_rewriting_intro} because of the additional factors $W_{t - at}^{(u, at)}(\beta)^2$, but we show that they do not affect the stable nature of the limit.

The transition at $\beta = \sqrt{2}/2$ described above is similar to the one appearing in the fluctuations of the additive martingales, obtained by Iksanov, Kolesko and Meiners \cite{IksanovKoleskoMeiners2020} for the BRW (see~\cite{Chataignier2024} for the adaptation to the BBM case, and also \cite{IksanovKabluchko2016,HartungKlimovsky2018} for previous partial result).
This fluctuations are obtained using the following decomposition
\begin{equation} \label{eq:fluctuations}
    W_\infty(\beta) - W_t(\beta) 
    = \sum_{u\in\cN(t)} \e^{\beta X_u(t) - \psi(\beta) t} (W_\infty^{(u)}(\beta) - 1),
\end{equation}
where the main difference with~\eqref{eq:overlap_rewriting_intro} is that this is a sum of centered variables, whereas the terms in~\eqref{eq:overlap_rewriting_intro} are positive.
However, similarly to the case of the overlap, for $\beta \in (0,\sqrt{2}/2)$, the sum in~\eqref{eq:fluctuations} is dominated by particles contributing to $\cG_{2\beta,t}$ (\ie particles $u\in\cN(t)$ such that $X_u(t) = 2\beta t + O(\sqrt{t})$) and the limit, after rescaling, is normally distributed with variance $c W_\infty(2\beta)$.
Furthermore, for $\beta \in (\sqrt{2}/2,\sqrt{2})$, the sum is dominated by extremal particles and, as noted in \cite{Chataignier2024}, the limit is $\alpha$-stable with $\alpha = \sqrt{2}/\beta \in (1, 2)$.
In the latter regime, we apply methods similar to those used in \cite{IksanovKoleskoMeiners2020}.

We finally mention other results on fluctuation of additive martingales, which are not directly related to our results on the overlap, but which completes the picture presented above.
In the critical case $\beta=\sqrt{2}$, fluctuations of the derivative martingale (defined in~\eqref{eq:definition_derivative_martingale}) around its limit have been obtained in \cite{MaillardPain2019} for the BBM and in \cite{HouRenSong2024,BuraczewskiIksanovMallein2021} for the BRW, with a spectrally positive 1-stable limit.
In \cite{MaillardPain2021} for the BBM, the fluctuations in the convergence~\eqref{eq:convergence_critical_additive_martingale} for $W_t(\sqrt{2})$ are showed to be Cauchy distributed.
The case of a complex valued $\beta$ has also been investigated.
In the whole region where $W_t(\beta)$ converges in $L^1$, the complex-stable fluctuations around its limit have been obtained by \cite{IksanovKoleskoMeiners2020} for the BRW (see also the previous work \cite{HartungKlimovsky2018} which covered the case $\re\beta \in (0,\sqrt{2}/2)$ where the limit is Gaussian).
Another region, called the glassy phase, extends the convergence~\eqref{eq:supercritical_W} and has been treated for the BBM in
\cite{MadauleRhodesVargas2015,HartungKlimovsky2015}.
Finally, the last region, where $\abs{\beta} \geq 1$ and $\re\beta \leq \sqrt{2}/2$, has been covered in \cite{HartungKlimovsky2018}: in that case, oscillations due to the imaginary part of $\beta$ are so strong that $W_t(\beta)$ does not converge, but, with a proper rescaling, it has a complex Gaussian limit with variance $c W_\infty(2\re\beta)$, which is similar to the case $\beta \in (0,\sqrt{2}/2)$ mentioned above.

\paragraph{Mean overlap distribution.}
It is surprising that, in terms of the mean overlap distribution, the threshold at $\beta = \sqrt{2}/2$ disappears, and this new threshold at $\beta = \sqrt{2\smash{/}3}$ appears. We do not know other phenomena for the BBM where this threshold plays a role.
Moreover, let us emphasize here that the proof of Theorem~\ref{thm:asymptotics_in_mean} is more involved than the one of Theorem~\ref{thm:asymptotics_of_the_overlap_distribution}, even in the case $\beta \in (0,\sqrt{2}/2)$, partly due to the fact that $1/W_\infty(\beta)$ is not integrable for any $\beta \in [0,\sqrt{2})$, so the denominator $W_t(\beta)^2$ in~\eqref{eq:overlap_rewriting_intro} has to be handled carefully.
Indeed, an event which could make the mean overlap much larger than its typical value is the event where $W_t(\beta)$ is atypically small: this is exactly what happens in the case $\beta=0$ and makes it special. However, for $\beta>0$, it is implicitly justified by our proof that this event is irrelevant in the behavior of the mean overlap, because in that case the numerator in~\eqref{eq:overlap_rewriting_intro} becomes small enough to compensate the denominator.

We now give some crude heuristics for the behavior of the mean overlap when $\beta \in (0,\sqrt{2})$. 
Since we have already excluded events where the denominator in~\eqref{eq:overlap_rewriting_intro} is small, we now wonder if the mean of the numerator is governed by a rare event or not.
First note that a particle $w \in \cN(at)$ contributes to the numerator in~\eqref{eq:overlap_rewriting_intro} proportionally to 
\begin{equation} \label{eq:contrib_w}
\e^{2\beta X_w(at)} W_{t-at}^{(w,at)}(\beta)^2, 
\end{equation}
where the two factors are independent.
By Girsanov theorem, the event dominating in $\E[\e^{2\beta X_w(at)}]$ is the one where $(X_w(s))_{s\in[0,at]}$ behaves like a Brownian motion with drift $2\beta$, and there are typically such particles $w$ in the BBM if and only if $\beta < \sqrt{2}/2$.
Moreover, in that case, $W_{t-at}^{(w,at)}(\beta)^2$ is bounded in $L^1$.
This means that, in the case $\beta \in (0,\sqrt{2}/2)$, the expectation of the numerator in~\eqref{eq:overlap_rewriting_intro} is not governed by a rare event and therefore the behavior of $\Expec{\nu_{\beta, t}([a, 1])}$ is simply obtained by taking the expectation of the typical behavior (see Remark~\ref{rem:rewriting}).

As soon as $\beta \in [\sqrt{2}/2,\sqrt{2})$, $\Expec{\nu_{\beta, t}([a, 1])}$ starts being dominated by a rare event where some particle $w\in\cN(at)$ has an atypically large contribution~\eqref{eq:contrib_w}.
We now argue that, in order to have a large contribution~\eqref{eq:contrib_w}, the best strategy is to have a high position $X_w(at)$ and not a large value of $W_{t-at}^{(w,at)}(\beta)$.
The subtlety is that this fact is not always true, it depends on how large the contribution should be, so we first restrict the range of this contribution.
Note that a particle $w \in \cN(at)$ also contributes to the denominator in~\eqref{eq:overlap_rewriting_intro}: its contribution to $W_t(\beta)$ equals $\e^{\beta X_w(at) - \psi(\beta)at} W_{t - at}^{(w, at)}(\beta)$.
Since $W_t(\beta)$ is typically of order 1, this contribution dominates as soon as it becomes much larger than 1, and then the numerator and the denominator compensate each other: therefore, it is not useful to force $\e^{\beta X_w(at) - \psi(\beta)at} W_{t - at}^{(w, at)}(\beta)$ to be much larger than one, because $\nu_{\beta, t}([a, 1])$ saturates at 1.
So we want $\e^{\beta X_w(at)} W_{t-at}^{(w, at)}(\beta)$ to be large but not larger than $\e^{\psi(\beta)at}$. 
Now note that, for some fixed $\lambda > 0$, we have
\begin{equation*}
    \P \left( \e^{\beta X_w(at)} \geq \e^{\lambda at} \right)
    \simeq \e^{-\lambda^2 at/2\beta^2} 
    \qquad \text{and} \qquad 
    \P \left( W_{t - at}^{(w, at)}(\beta) \geq \e^{\lambda at} \right) \lesssim \e^{-2\lambda at/\beta^2},
\end{equation*}
where first claim is a Gaussian tail estimate and the second one is formally stated in Lemma~\ref{lem:additive_martingale_convergence_in_Lp}.
It follows that, on the event $\e^{\beta X_w(at)} W_{t-at}^{(w, at)}(\beta) \simeq \e^{\lambda at}$, we typically have $\e^{\beta X_w(at)} \simeq \e^{\lambda at}$ as soon as $\lambda<4$, which is the case in the range that we are considering ($\psi(\beta) < \psi(\sqrt{2}) = 2$). This justifies the claim presented at the beginning of this paragraph.

We can now give heuristics to explain the appearance of the threshold at $\beta = \sqrt{2\smash{/}3}$. 
As mentioned earlier, by Girsanov theorem, the event dominating in $\E[\e^{2\beta X_w(at)}]$ is the one where $X_w(at) = 2\beta at + O(\sqrt{t})$, however we also have the constraint that we do not want to make $\e^{\beta X_w(at)}$ larger than $\e^{\psi(\beta)at}$ otherwise the denominator in~\eqref{eq:overlap_rewriting_intro} starts to play a role and $\nu_{\beta, t}([a, 1])$ saturates at 1.
Therefore, particles $w\in \cN(at)$ contributing mostly to $\Expec{\nu_{\beta, t}([a, 1])}$ are at a position roughly $v(\beta) at$, where $v(\beta)$ is given by
\begin{equation}
    v(\beta) = 2\beta \wedge \frac{\psi(\beta)}{\beta} 
    = \begin{cases}
        2 \beta & \text{if } \beta \in (0,\sqrt{2\smash{/}3}], \\
        \frac{1}{\beta} + \frac{\beta}{2} & \text{if } \beta \in [\sqrt{2\smash{/}3}, \sqrt{2}).
    \end{cases}
\end{equation}
See Figure \ref{fig:v(beta)} for the graph of function $v$.
Hence, the threshold at $\beta = \sqrt{2\smash{/}3}$ corresponds to the inverse temperature above which, in mean, particles contributing mostly to the numerator in~\eqref{eq:overlap_rewriting_intro} also contributes significantly to the denominator.

\begin{figure}
    \centering
    \begin{tikzpicture}[xscale=5,yscale=3]
        % Axis
        \draw[->,>=latex] (0,0) -- (1.6,0) node[below]{$\beta$};
        \draw[->,>=latex] (0,0) -- (0,1.7);
        \draw (0,0) node[below left]{$0$};
        % Lines from abscissa
        \draw[dashed] ({sqrt(2)/2},-.015) -- ({sqrt(2)/2},{sqrt(2)});
        \draw ({sqrt(2)/2},0) node[below]{$\frac{\sqrt{2}}{2}$};
        \draw[dashed] ({sqrt(2/3)},-.015) -- ({sqrt(2/3)},{2*sqrt(2/3)});
        \draw ({sqrt(2/3)},0) node[below]{$\frac{\sqrt{2}}{\sqrt{3}}$};
        \draw[dashed] (1,-.015) -- (1,.015);
        \draw (1,0) node[below]{$1$};
        \draw[dashed] ({sqrt(2)},-.015) -- ({sqrt(2)},.015);
        \draw ({sqrt(2)},0) node[below]{$\sqrt{2}$};
        % Lines from ordinate
        \draw[dashed] (-.01,{sqrt(2)}) -- ({sqrt(2)},{sqrt(2)});
        \draw (0,{sqrt(2)}) node[left]{$\sqrt{2}$};
        % Contributing trajectories
        \draw[domain=0:sqrt(2/3),NavyBlue,thick] plot (\x, 2*\x);
        \draw[domain=sqrt(2/3):sqrt(2),BrickRed,thick] plot (\x, 1/\x+\x/2);
        % Legend
        \draw[NavyBlue,thick] (1.7,.7) -- (1.8,.7) node[right]{$v(\beta) = 2\beta$};
        \draw[BrickRed,thick] (1.7,.5) -- (1.8,.5) node[right]{$v(\beta) = 1/\beta + \beta/2$};
    \end{tikzpicture}
    \caption{Graph of the function $v$ such that particles that contribute to $\Expec{\nu_{\beta, t}([a, 1])}$ are ``near'' $v(\beta)at$ at time $at$.}
    \label{fig:v(beta)}
\end{figure}
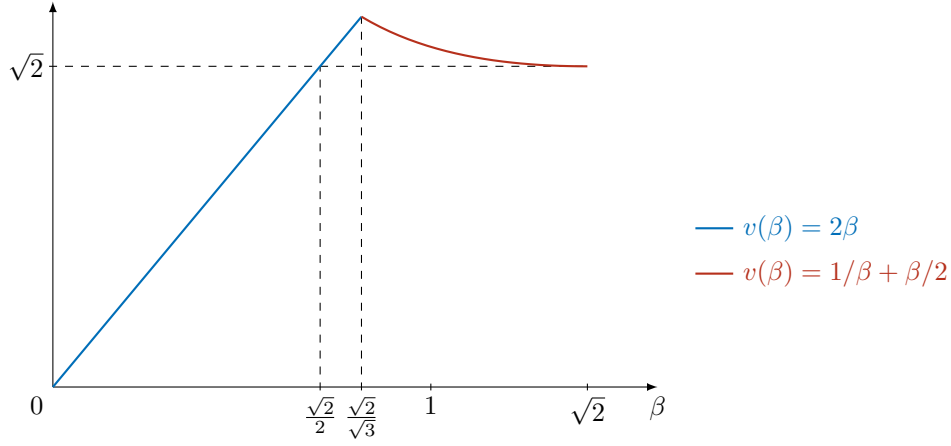

As a consequence of this discussion, one can describe the behavior of particles contributing mostly to the mean overlap distribution.
If $\beta \in (0, \sqrt{2\smash{/}3})$, then contributing particles have a drift $v(\beta) = 2\beta$ until time $at$ and then a drift $\beta$ (see Figure~\ref{fig:contributing_particles_3}).
If $\beta \in [\sqrt{2\smash{/}3},\sqrt{2})$, we explained that contributing particles have a drift $v(\beta) = 1/\beta + \beta/2$ until time $at$, but one can be more precise than that. 
If the contributing particle has, at some time $s \in [0,at]$, an ancestor at $v(\beta) s + x$, then one can check that the descendants at time $t$ of this ancestor have a contribution of order $\e^{\beta x}$ in $W_t(\beta)$, so any such event with a large $x$ results in a small value of $\nu_{\beta, t}([a, 1])$.
Hence, contributing particles are roughly constrained below the straight barrier $s \mapsto v(\beta)s$.
Therefore, if $\beta = \sqrt{2\smash{/}3}$, contributing particles move like a Brownian motion with drift $v(\beta)$ conditioned to stay below the barrier until time $at$, and then like a Brownian motion with drift $\beta$.
If $\beta \in (\sqrt{2\smash{/}3}, \sqrt{2})$, the factor $\e^{2\beta X_w(at)}$ in the numerator of~\eqref{eq:overlap_rewriting_intro} additionally forces particles to end up close to the barrier at time $at$, hence contributing particles move like a Brownian motion with drift $v(\beta)$ conditioned to stay below the barrier until time $at$ while reaching a position $v(\beta) at + O(1)$ at time $at$, and then like a Brownian motion with drift $\beta$ (see Figure~\ref{fig:contributing_particles_4}).
Note that these constraints explain the polynomial corrections $t^{-1/2}$ and $t^{-3/2}$ appearing in Theorem~\ref{thm:asymptotics_in_mean}.\ref{it:asymptotics_in_mean_3} and \ref{it:asymptotics_in_mean_4}.

\begin{figure}
    \centering
    % Variables
    \tikzmath{\t=2;\a=.65;\bc=sqrt(2);\b1=.4;\b2=.92;\b3=1;}    
    \begin{subfigure}[b]{.45\textwidth}
        \begin{tikzpicture}[xscale=2.7,yscale=2]
            % BBM
            \fill[fill=lightgray!30] (0,0) -- (\t,{\bc*\t}) -- (\t,0) -- (0,0);
            \draw[Gray] (0,0) -- (\t,{\bc*\t});
            % Axis
            \draw[->,>=latex] (0,0) -- (2.4,0);
            \draw[->,>=latex] (0,-.1) -- (0,3.5);
            \draw (0,0) node[below left]{$0$};
            % Lines from abscissa
            \draw[dashed] ({\a*\t},-.025) -- ({\a*\t},{2*\b3*\a*\t});
            \draw ({\a*\t},0) node[below]{$at$};
            \draw[dashed] (\t,-.025) -- (\t,0);
            \draw (\t,0) node[below]{$t$};
            % Lines from ordinate
            \draw[dashed] (-.025,{\bc*\t}) -- (\t,{\bc*\t});
            \draw (0,{\bc*\t}) node[left]{$\sqrt{2}t$};
            % Drift
            \draw[NavyBlue] (0,0) -- ({\a*\t},{2*\b3*\a*\t}) -- (\t,{\b3*\t*(1+\a)});
            \draw[<-,>=latex,NavyBlue] ({\a+.1},{2*\b3*(\a+.1)+.1}) -- (\a,{2*\b3*(\a+.1)+.7});
            \draw[NavyBlue] (\a,{2*\b3*(\a+.1)+.7}) node[above]{$2\beta$};
            \draw[<-,>=latex,NavyBlue] ({\t*(1+\a)/2},{(2*\b3*\a*\t+\b3*(\t-\a*\t)/2)-.2}) -- ({\t*(1+\a)/2},{(2*\b3*\a*\t+\b3*(\t-\a*\t)/2)-.9});
            \draw[NavyBlue] ({\t*(1+\a)/2},{(2*\b3*\a*\t+\b3*(\t-\a*\t)/2)-.9}) node[below]{$\beta$};
            % Common trajectory
            \draw[smooth] plot coordinates
            {(0,0) (.1,.03) (.2,.47) (.3,.46) (.4,1.07) (.5,.91) (.6,1.31) (.7,1.23) (.8,1.61) (.9,1.68) (1,2.2) (1.1,2.11) (1.2,2.48)
            % Trajectory 1
            (1.3,2.54) (1.4,2.79) (1.5,2.73) (1.6,2.88) (1.7,2.77) (1.8,3.08) (1.9,2.99) (2,3.18)};
            % Trajectory 2
            \draw[smooth] plot coordinates
            {(1.3,2.54) (1.4,2.62) (1.5,2.86) (1.6,3.01) (1.7,2.91) (1.8,3.26) (1.9,3.19) (2,3.41)};
            % Particles
            \draw (2,3.41) node{$\bullet$};
            \draw (2,3.41) node[right]{$u$};
            \draw (2,3.18) node{$\bullet$};
            \draw (2,3.18) node[right]{$v$};
            \draw (1.3,2.54) node{$\bullet$};
            \draw (1.3,2.54) node[below right]{$w$};
        \end{tikzpicture}
        \caption{Case $0 \leq \beta < \sqrt{2\smash{/}3}$. Particles that contribute to $\Expec{\nu_{\beta, t}([a, 1])}$ have drift~$2\beta$ until time $at$ and then have drift~$\beta$.\label{fig:contributing_particles_3}}
    \end{subfigure}
    \hfill
    \begin{subfigure}[b]{.45\textwidth}
        \begin{tikzpicture}[xscale=2.7,yscale=2]
            % BBM
            \fill[fill=lightgray!30] (0,0) -- (\t,{\bc*\t}) -- (\t,0) -- (0,0);
            \draw[Gray] (0,0) -- (\t,{\bc*\t});
            % Axis
            \draw[->,>=latex] (0,0) -- (2.4,0);
            \draw[->,>=latex] (0,-.1) -- (0,3.5);
            \draw (0,0) node[below left]{$0$};
            % Lines from abscissa
            \draw[dashed] ({\a*\t},-.025) -- ({\a*\t},{2*\b3*\a*\t});
            \draw ({\a*\t},0) node[below]{$at$};
            \draw[dashed] (\t,-.025) -- (\t,0);
            \draw (\t,0) node[below]{$t$};
            % Lines from ordinate
            \draw[dashed] (-.025,{\bc*\t}) -- (\t,{\bc*\t});
            \draw (0,{\bc*\t}) node[left]{$\sqrt{2}t$};
            % Drift
            \draw[NavyBlue] (0,0) -- ({\a*\t},{2*\b3*\a*\t}) -- (\t,{\b3*\t*(1+\a)});
            \draw[<-,>=latex,NavyBlue] ({\a+.1},{2*\b3*(\a+.1)+.1}) -- (\a,{2*\b3*(\a+.1)+.7});
            \draw[NavyBlue] (\a,{2*\b3*(\a+.1)+.7}) node[above]{$1/\beta + \beta/2$};
            \draw[<-,>=latex,NavyBlue] ({\t*(1+\a)/2},{(2*\b3*\a*\t+\b3*(\t-\a*\t)/2)-.2}) -- ({\t*(1+\a)/2},{(2*\b3*\a*\t+\b3*(\t-\a*\t)/2)-.9});
            \draw[NavyBlue] ({\t*(1+\a)/2},{(2*\b3*\a*\t+\b3*(\t-\a*\t)/2)-.9}) node[below]{$\beta$};
            % Common trajectory
            \draw[smooth] plot coordinates
            {(0,0) (.1,.03) (.2,.37) (.3,.46) (.4,.69) (.5,.74) (.6,1.04) (.7,1.11) (.8,1.44) (.9,1.46) (1,1.92) (1.1,1.97) (1.2,2.44)
            % Trajectory 1
            (1.3,2.54) (1.4,2.79) (1.5,2.73) (1.6,2.88) (1.7,2.77) (1.8,3.08) (1.9,2.99) (2,3.18)};
            % Trajectory 2
            \draw[smooth] plot coordinates
            {(1.3,2.54) (1.4,2.62) (1.5,2.86) (1.6,3.01) (1.7,2.91) (1.8,3.26) (1.9,3.19) (2,3.41)};
            % Particles
            \draw (2,3.41) node{$\bullet$};
            \draw (2,3.41) node[right]{$u$};
            \draw (2,3.18) node{$\bullet$};
            \draw (2,3.18) node[right]{$v$};
            \draw (1.3,2.54) node{$\bullet$};
            \draw (1.3,2.54) node[below right]{$w$};
        \end{tikzpicture}
        \caption{Case $\sqrt{2\smash{/}3} < \beta < \sqrt{2}$. Particles that contribute to $\Expec{\nu_{\beta, t}([a, 1])}$ are near $(1/\beta + \beta/2)at$ at time $at$ and then have drift~$\beta$.\label{fig:contributing_particles_4}}
    \end{subfigure}
    \caption{Behaviors of the particles that contribute to the mean overlap.}
\end{figure}
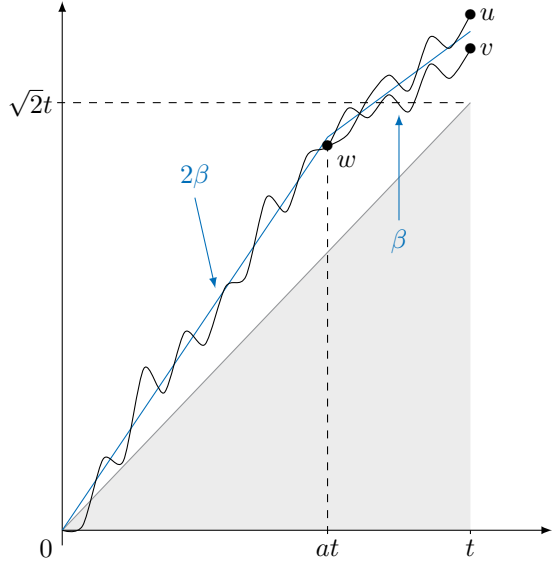
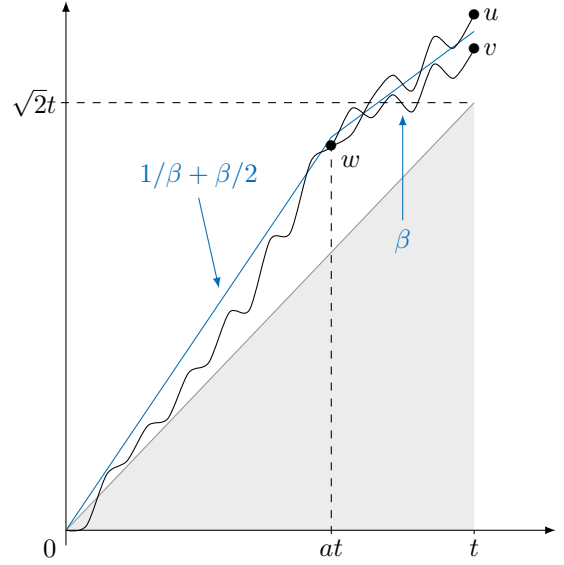

\subsection{Organization of the paper}

Section \ref{sec:preliminary} contains a formal definition of the model, which is then used to introduce properly some change of measures, and results concerning the additive martingales, the extremal process of the BBM and ballot theorems for Brownian motions.
In Section \ref{sec:typical}, we prove Theorem~\ref{thm:asymptotics_of_the_overlap_distribution}, with one subsection dedicated to each part of the statement (in the same order).
Finally, Section \ref{sec:mean} contains the proof of Theorem~\ref{thm:asymptotics_in_mean}, again with one subsection dedicated to each part of the statement (but this time in the order \ref{it:asymptotics_in_mean_2}, \ref{it:asymptotics_in_mean_4}, \ref{it:asymptotics_in_mean_3} and \ref{it:asymptotics_in_mean_1}).

\section{Preliminary results}
\label{sec:preliminary}

\subsection{The model}

Following the formalism of Chauvin and Rouault \cite{ChauvinRouault1988}, we consider the BBM to be a binary tree endowed with random lifetimes and trajectory data.
In addition, it will be useful to distinguish one line of descent, called the \emph{spine}.
Let us specify.

We consider the set of Ulam--Harris--Neveu labels $\cT = \bigcup_{i \geq 0} \{1, 2\}^i$, with the convention $\{1, 2\}^0 = \{\varnothing\}$.
For a label $u \in \cT$, we denote by $|u|$ its length, with the convention $|\varnothing| = 0$.
For two labels $u, v \in \cT$, we denote by $uv$ their concatenation, with the convention $u \varnothing = \varnothing u = u$.
We say that $u$ is an \emph{ancestor} of $v$ and that $v$ is a \emph{descendant} of $u$ if there exists $w \in \cT$ such that $v = uw$.
In this case, we write $u \leq v$.
If in addition $u \neq v$, then we write $u < v$.
We denote by $u \wedge v$ the most recent common ancestor of two particles $u$ and $v$, \ie $u \wedge v = u_1 \ldots u_{i_0}$ with $i_0 = \sup\{i \geq 1 : u_1 \ldots u_i = v_1 \ldots v_i\}$ if $u_1 = v_1$, and $u \wedge v = \varnothing$ otherwise.

The probability space on which we work is $\Omega$ the set of tuples $((\sigma_u)_{u \in \cT}, (Y_u)_{u \in \cT}, \xi)$, where, for each $u \in \cT$, $\sigma_u \geq 0$ and $Y_u : [0, \infty) \to \R$ is a continuous function, and $\xi = \{ \varnothing, \xi_1, \xi_1\xi_2, \xi_1\xi_2\xi_3, \ldots \}$, with $(\xi_k)_{k\geq 1}\in \{1,2\}^{\N^*}$, is an infinite line of descent of $\cT$.
We denote by $b_u = \sum_{v < u} \sigma_v$ the birthtime of a particle $u$, and $d_u = \sum_{v \leq u} \sigma_v$ its deathtime.
For $t \geq 0$, we define
\begin{equation*}
    \cN(t) = \{u \in \cT : b_u \leq t < d_u\} \quad \text{and} \quad \cN([0, t]) = \{u \in \cT : b_u \leq t\}.
\end{equation*}
We denote by $\xi(t)$ the only particle in $\cN(t) \cap \xi$, called \emph{spine}.
We define inductively the position at time $t$ of a particle $u \in \cN(t)$
\begin{equation*}
    X_u(t) = \sum_{v < u} X_v(d_v) + Y_u(t - b_u).
\end{equation*}
We extend the notion of position for a particle $u \in \cN(t)$ to the whole interval $[0, t]$: for every $s \in [0, t]$, we set $X_u(s) = X_v(s)$, where $v$ is the ancestor of $u$ alive at time $s$.
To avoid redundancy, we denote the spinal position $X_\xi(t) = X_{\xi(t)}(t)$.

For $t \geq 0$, let $\cF_t$ be the $\sigma$-algebra containing all the information until time $t$ except the spine,
\begin{equation*}
    \cF_t = \sigma \left( \left(u, d_u \wedge t, X_u|_{[0, d_u \wedge t]} \right) : u \in \cN([0, t]) \right) \quad \text{and} \quad \cF_\infty = \sigma \left( \bigcup_{t \geq 0} \cF_t \right),
\end{equation*}
where $s \wedge t$ denotes the minimum between $s$ and $t$.
Let us define a finer filtration by adding the information about the spine,
\begin{equation*}
    \tilde{\cF}_t = \sigma \left( \cF_t, \xi(t) \right) \quad \text{and} \quad \tilde{\cF}_\infty = \sigma \left( \bigcup_{t \geq 0} \tilde{\cF}_t \right).
\end{equation*}
We define the BBM law as the unique probability measure $\P$ on $(\Omega, \tilde{\cF}_\infty)$ for which
\begin{itemize}
    \item the trajectories $Y_u$ are \iid standard Brownian motions,
    \item the lifetimes $\sigma_u$ are \iid exponential variables with rate $1$,
    \item the spinal coordinates $\xi_1, \xi_2, \ldots$ are \iid with $\P(\xi_1 = 1) = \P(\xi_1 = 2) = 1/2$,
    \item $(Y_u)_{u \in \cT}$, $(\sigma_u)_{u \in \cT}$, $\xi$ are independent.
\end{itemize}
We denote by $\E$ the associated expectation.

\begin{remark}\label{rem:probability_to_be_the_spine}
    For all $t \geq 0$ and $u \in \cN(t)$, $\P \left( \xi(t) = u \middle| \cF_t \right) = 2^{-|u|}$.
\end{remark}

In order to state a branching property, we define, for $t \geq 0$, a shift operator $\Theta_{\xi, t} : \Omega \to \Omega$ by $\Theta_{\xi, t}(\sigma, Y, \xi) = (\sigma', Y', \xi')$, where, if $u = \xi(t)$,
\begin{enumerate}
    \item\label{pt:shifted_lifetime} $\sigma'_v = \begin{cases}
        d_u - t & \text{if } v = \varnothing, \\
        \sigma_{uv} & \text{otherwise},
    \end{cases}$
    \item\label{pt:shifted_trajectory} $Y'_v(s) = \begin{cases}
        X_u(t + s) - X_u(t) & \text{if } v = \varnothing, \\
        Y_{uv}(s) & \text{otherwise},
    \end{cases}$
    \item $\xi' = \{\varnothing, \xi_{|u| + 1}, \xi_{|u| + 1}\xi_{|u| + 2}, \ldots \}$.
\end{enumerate}
We also define, for $u \in \cT$ and $t \geq 0$, $\Omega_{u, t} = \{(\sigma, Y, \xi) \in \Omega : u \in \cN(t)\}$ and $\Theta_{u, t} : \Omega_{u, t} \to \Omega$ by $\Theta_{u, t}(\sigma, Y, \xi) = (\sigma', Y', \xi')$ satisfying \ref{pt:shifted_lifetime}, \ref{pt:shifted_trajectory} and $\xi' = \{\varnothing, 1, 11, \ldots \}$.
A consequence of \cite[Proposition~2.1]{Chauvin1991} is the following branching property.
If $Y$ and $Y_u$, $u \in \cT$, are non-negative variables on $\Omega$ such that $Y$ is $\tilde{\cF}_\infty$-measurable and $Y_u$, $u \in \cT$, are $\cF_\infty$-measurable, then for every $t \geq 0$,
\begin{equation}
    \E \left[ Y \circ \Theta_{\xi, t} \times \prod_{u \in \cN(t), u \neq \xi(t)} Y_u \circ \Theta_{u, t} \middle| \tilde{\cF}_t \right] = \E[Y] \prod_{u \in \cN(t), u \neq \xi(t)} \E[Y_u]. \label{eq:branching_property_with_spine}
\end{equation}
This means that, at time $t$, all the alive particles start independent BBMs shifted in time, space, label, and that the information about the spine is kept in the BBM started from $\xi(t)$.

\subsection{Change of measure}
\label{subsec:change_of_measure}

In this section, we introduce some changes of measure that allow a useful interpretation of the process in term of spinal decomposition.
This approach dates back to Kahane and Peyrière \cite{KahanePeyriere1976} for the BRW and to Chauvin and Rouault \cite{ChauvinRouault1988} for the BBM, and it has been widely used since, see \eg \cite{Lyons1997} and \cite{Kyprianou2004} for applications to the study of the additive martingales. 
The framework of these techniques has been extended in \cite{HardyHarris2006} and \cite{HarrisRoberts2017} to more general branching processes and the presentation below follows roughly the lines of \cite{HardyHarris2006}.

Let $\beta \geq 0$ and $\psi(\beta) = 1 + \beta^2/2$.
The process $(2^{|\xi(t)|} \e^{\beta X_\xi(t) - \psi(\beta)t})_{t \geq 0}$ is a positive $(\tilde{\cF}_t)_{t \geq 0}$-martingale with mean $1$.
By Kolmogorov's extension theorem, we can define a probability measure $\Q_\beta$ on $\tilde{\cF}_\infty$ by
\begin{equation}\label{eq:definition_of_Q_b}
    \frac{\d{\Q_\beta|_{\tilde{\cF}_t}}}{\d{\P|_{\tilde{\cF}_t}}} = 2^{|\xi(t)|} \e^{\beta X_\xi(t) - \psi(\beta)t}.
\end{equation}

\begin{remark}
    The additive martingale defined in~\eqref{eq:definition_additive_martingale} is the projection of the above martingale onto $\cF_t$,
    \begin{equation*}
        W_t(\beta) = \E \left[ 2^{|\xi(t)|} \e^{\beta X_\xi(t) - \psi(\beta)t} \middle| \cF_t \right].
    \end{equation*}
    Therefore, the change of measure~\eqref{eq:definition_of_Q_b} is consistent with~\eqref{eq:definition_of_Q_b_without_spine}.
\end{remark}

In \cite[Theorem~5]{ChauvinRouault1988}, Chauvin and Rouault describe the effect of this change of measure on the BBM.
Under $\Q_\beta$,
\begin{itemize}
    \item the spine's motion is a Brownian motion with drift $\beta$,
    \item the spine particle splits at rate $2$,
    \item at each splitting event of the spine particle, the label of the new spine particle is chosen uniformly among the two children,
    \item at this splitting event, the other child starts an independent standard BBM without spine.
\end{itemize}
By \cite[Theorem~8.1]{HardyHarris2006}, a remarkable and useful property is that
\begin{equation}\label{eq:gibbs_measure_weights}
    \Q_{\beta} \left( \xi(t) = u \middle| \cF_t \right) = \mathds{1}_{u \in \cN(t)} \frac{\e^{\beta X_u(t) - \psi(\beta)t}}{W_t(\beta)}.
\end{equation}
As mentioned in Remark~\ref{rem:rewriting}, if $\beta < \sqrt{2}$, then the uniform integrability of $(W_t(\beta))_{t\geq 0}$ implies that the measure $\Q_\beta|_{\cF_\infty}$ is uniformly continuous w.r.t.\@ $\P|_{\cF_\infty}$ with Radon--Nikodym derivative $W_\infty(\beta)$ (note however that $\Q_\beta$ and $\P$ on the full $\sigma$-field $\tilde{\cF}_\infty$ are singular).
On the other hand, if $\beta \geq \sqrt{2}$, then the convergence of $W_t(\beta)$ towards zero implies that the measures $\Q_\beta|_{\cF_\infty}$ and $\P_{\cF_\infty}$ are singular.

In this paper, we will also use a slightly different change of measure, that turns out to be more suited to the study of the overlap distribution.
For $\beta \geq 0$ and $t \geq 0$, we define $\Q_{\beta, t}$ on $\tilde{\cF}_\infty$ by
\begin{equation}\label{eq:definition_of_Q_bt}
    \frac{\d{\Q_{\beta, t}}}{\d{\P}} = 2^{|\xi(t)|} \e^{\beta X_\xi(t) - \psi(\beta) t}.
\end{equation}

\begin{remark}
    Let $Y_t$, $Y_\xi$ and $Y_u$, $u \in \cT$, be non-negative variables on $\Omega$ such that $Y_t$ is $\tilde{\cF}_t$-measurable, $Y_\xi$ is $\tilde{\cF}_\infty$-measurable, and $Y_u$, $u \in \cT$, are $\cF_\infty$-measurable.
    Then,
    \begin{align}
        &\E_{\Q_{\beta, t}} \left[ Y_t \times Y_\xi \circ \Theta_{\xi, t} \times \prod_{u \in \cN(t), u \neq \xi(t)} Y_u \circ \Theta_{u, t} \right] \nonumber \\
        &\quad = \E \left[ 2^{|\xi(t)|} \e^{\beta X_\xi(t) - \psi(\beta)t} \times Y_t \times \E \left[ Y_\xi \circ \Theta_{\xi, t} \times \prod_{u \in \cN(t), u \neq \xi(t)} Y_u \circ \Theta_{u, t} \middle| \tilde{\cF}_t \right] \right] \nonumber \\
        &\quad = \E_{\Q_\beta} \left[ Y_t \times \E \left[ Y_{\xi} \right] \times \prod_{u \in \cN(t), u \neq \xi(t)} \E \left[ Y_u \right] \right], \nonumber
    \end{align}
    by the branching property~\eqref{eq:branching_property_with_spine} and the definition of $\Q_\beta$.
    This means that, under $\Q_{\beta, t}$,
    \begin{itemize}
        \item between $0$ and $t$, the BBM has the same behavior as under $\Q_\beta$,
        \item at time $t$, the spine particle starts a standard BBM, shifted in time, space and label,
        \item at time $t$, the non-spine particles start standard BBMs without spine, shifted in time, space and label, independent of each other and of the spine-BBM.
    \end{itemize}
\end{remark}

Similarly to~\eqref{eq:gibbs_measure_weights}, we have the following property.

\begin{lemma}\label{lem:gibbs_measure_weights}
    For all $u \in \cT$ and $s, t \geq 0$,
    \begin{equation*}
        \Q_{\beta, t}\left(\xi(t) = u \middle| \cF_{t + s}\right) = \mathds{1}_{u \in \cN(t)} \frac{\e^{\beta X_u(t) - \psi(\beta)t}}{W_t(\beta)}.
    \end{equation*}
\end{lemma}

\begin{proof}
    Let $A \in \cF_{t+s}$.
    By definition~\eqref{eq:definition_of_Q_bt},
    \begin{equation}\label{eq:gibbs_measure_weights_1}
        \E_{\Q_{\beta, t}} \left[ \mathds{1}_A \mathds{1}_{\xi(t) = u} \right] = \E \left[ \mathds{1}_A \mathds{1}_{\xi(t) = u} 2^{|\xi(t)|} \e^{\beta X_\xi(t) - \psi(\beta) t} \right] = \E \left[ \mathds{1}_A \mathds{1}_{\xi(t) = u} 2^{|u|} \e^{\beta X_u(t) - \psi(\beta) t} \right].
    \end{equation}
    But $u$ is the spinal particle at time $t$ if and only if $u$ is alive at time $t$ and admits the spinal particle among its descendants at time $t + s$.
    This consideration together with Remark~\ref{rem:probability_to_be_the_spine} allow us to rewrite~\eqref{eq:gibbs_measure_weights_1} as
    \begin{align}
        &\E \left[ \mathds{1}_A \mathds{1}_{u \in \cN(t)} \left(\sum_{v \in \cN(t+s), v \geq u} \mathds{1}_{\xi(t+s) = v}\right) 2^{|u|} \e^{\beta X_u(t) - \psi(\beta) t} \right] \nonumber \\
        &\quad = \E \left[ \mathds{1}_A \mathds{1}_{u \in \cN(t)} \left( \sum_{v \in \cN(t+s), v \geq u} 2^{-|v|} \right) 2^{|u|} \e^{\beta X_u(t) - \psi(\beta) t} \right]. 
        \label{eq:gibbs_measure_weights_2}
    \end{align}
    Note that $\{v \in \cN(t+s) : v \geq u\}$ is a binary tree and that, for its particles $v$, the number $|v| - |u|$ counts the fission times from $u$ to $v$.
    Therefore,
    \begin{equation*}
        \sum_{v \in \cN(t+s), v \geq u} 2^{-(|v| - |u|)} = 1.
    \end{equation*}
    We can then rewrite~\eqref{eq:gibbs_measure_weights_2} as
    \begin{equation*}
        \E \left[ \mathds{1}_A \mathds{1}_{u \in \cN(t)} \e^{\beta X_u(t) - \psi(\beta) t} \right] = \E_{\Q_{\beta, t}} \left[ \mathds{1}_A \mathds{1}_{u \in \cN(t)} \frac{\e^{\beta X_u(t) - \psi(\beta) t}}{W_t(\beta)} \right],
    \end{equation*}
    using here that $\Q_{\beta,t}|_{\cF_\infty}$ has density $W_t(\beta)$ w.r.t.\@ $\P|_{\cF_\infty}$.
    This concludes the proof.
\end{proof}

\subsection{Moments and tails of the additive martingale}

In what follows, it will be convenient to use the following rewriting of the additive martingale~\eqref{eq:definition_additive_martingale}.
For $t \geq s \geq 0$,
\begin{equation}\label{eq:rewriting_additive_martingale}
    W_t(\beta) = \sum_{u \in \cN(s)} \e^{\beta X_u(s) - \psi(\beta)s} W_{t - s}^{(u, s)}(\beta),
\end{equation}
where, for $r \geq 0$,
\begin{equation*}
    W_r^{(u, s)}(\beta) = W_r \circ \Theta_{u, s} = \sum_{v \in \cN(s + r), v \geq u} \e^{\beta (X_v(s + r) - X_u(s)) - \psi(\beta)r}.
\end{equation*}
The above process is the additive martingale of a BBM shifted in time, space and label.
It converges \as to some limit $W_\infty^{(u, s)}(\beta)$ and hence we can extend~\eqref{eq:rewriting_additive_martingale} to the case $t = \infty$.
By the branching property~\eqref{eq:branching_property_with_spine}, conditionally on $\cF_s$, the processes $(W_r^{(u, s)})_{r \in[0, \infty]}$ for $u \in \cN(s)$ are \iid copies of $(W_r)_{r \in [0, \infty]}$.

\begin{lemma}\label{lem:additive_martingale_convergence_in_Lp}
    Let $\beta \in [0, \sqrt{2})$.
    For any $p \in (1, 2/\beta^2)$, the additive martingale at inverse temperature $\beta$ is bounded in $L^p$ and hence converges in $L^p$.
    Moreover, if $p \in (1, 2/\beta^2) \cap (1, 2]$, then there exists $C>0$ such that, for any $0 \leq s \leq t \leq \infty$,
    \begin{equation*}
        \E \left| W_t(\beta) - W_s(\beta) \right|^p \leq C \e^{-(p - 1)(1 - p\beta^2/2)s}.
    \end{equation*}
\end{lemma}

The proof of this lemma relies partially on the following inequality, which is due to Von Bahr and Esseen \cite[Theorem~2]{VonBahrEsseen1965} and is used repetitively throughout the paper. The application of this type of inequalities to branching processes dates back to Neveu \cite{Neveu1988} and Biggins \cite{Biggins1992}.

\begin{lemma}\label{lem:von_Bahr}
Let $p \in [1,2]$ and $X_1,\dots,X_n$ be centered independent real random variables in $L^p$. Then,
\[
    \Expec{\abs{\sum_{k=1}^n X_k}^p}
    \leq 2 \sum_{k=1}^n \Expec{\abs{X_k}^p}.
\]
\end{lemma}

\begin{proof}[Proof of Lemma~\ref{lem:additive_martingale_convergence_in_Lp}]
    Let $p \in (1, 2/\beta^2)$.
    Equation~\eqref{eq:rewriting_additive_martingale} holds for $t = \infty$ and $s = 1$.
    Therefore, by \cite[Proposition~4.11]{Liu2002}, $W_\infty(\beta)$ is in $L^p$ if and only if
    \begin{equation*}
        1 > \E \left[ \# \cN(1) \right] \E \left[ \left( \e^{\beta X_\xi(1) - \psi(\beta)} \right)^p \right] = \e^{1 + p^2 \beta^2/2 - p\psi(\beta)}.
    \end{equation*}
    This inequality holds since $\psi(\beta) = 1 + \beta^2/2$ and $p \in (1, 2/\beta^2)$.
    Hence $W_\infty(\beta) \in L^p$.
    But, recalling that $(W_t(\beta))_{t\geq 0}$ is uniformly integrable for $\beta < \sqrt{2}$, we have $W_t(\beta) = \E \left[ W_\infty(\beta) \middle| \cF_t \right]$.
    Jensen's inequality for conditional expectations implies that the additive martingale is bounded in $L^p$.

    Now let $p \in (1, 2/\beta^2) \cap [1, 2]$ and $0 \leq s < t \leq \infty$.
    We rewrite
    \begin{equation*}
        W_t(\beta) - W_s(\beta) = \sum_{u \in \cN(s)} \e^{\beta X_u(s) - \psi(\beta)s} (W_{t - s}^{(u, s)} - 1).
    \end{equation*}
    Conditionally on $\cF_s$, the random variables $(W_{t - s}^{(u, s)} - 1)$ for $u \in \cN(s)$ are centered and independent.
    Therefore, by Lemma~\ref{lem:von_Bahr} and the branching property~\eqref{eq:branching_property_with_spine},
    \begin{align}
        &\E \left[ \left|W_t(\beta) - W_s(\beta)\right|^p \middle| \cF_s \right] \nonumber \\
        &\quad \leq \sum_{u \in \cN(s)} \e^{p \beta X_u(s) - p \psi(\beta)s} \E|W_{t-s}(\beta) - 1|^p = W_s(p\beta) \e^{-(p\psi(\beta) - \psi(p\beta))s} \E|W_{t-s}(\beta) - 1|^p. \nonumber
    \end{align}
    Note that
    \begin{equation}\label{eq:pc-cp}
        p\psi(\beta) - \psi(p\beta) = (p - 1)(1 - p\beta^2/2),
    \end{equation}
    which is positive, by choice of $p$.
    This concludes since $W_s(p\beta)$ has mean $1$ and the additive martingale at $\beta$ is bounded in $L^p$.
\end{proof}

\begin{lemma}\label{lem:left_tail_additive_martingale}
    Let $\beta \in [0, \sqrt{2})$ and fix a number $b$ such that
    \begin{equation}\label{eq:b_values}
        0 < b < \frac{\sqrt{\psi(\beta)^2 + 2\beta^2} - \psi(\beta)}{\beta^2}.
    \end{equation}
    There exists $K > 0$ such that, for every $t \in [0, \infty]$ and every $x \geq 0$,
    \begin{equation}\label{eq:left_tail_additive_martingale}
        \P(W_t(\beta) \leq x) \leq K x^b.
    \end{equation}
\end{lemma}

\begin{proof}
    Equation~\eqref{eq:rewriting_additive_martingale} holds for $t = \infty$ and $s = 1$.
    Therefore, by \cite[Theorem~2.4]{Liu2001}, $\P(W_\infty(\beta) \leq x) = O(x^b)$ as $x \to 0$ for any $b>0$ such that
    \begin{equation*}
        \E \left[ \left( \e^{\beta X_\xi(1) - \psi(\beta)} \right)^{-b} \right] < \infty \quad \text{and} \quad \E \left[ \left( \e^{\beta X_\xi(1) - \psi(\beta)} \right)^{-b} \mathds{1}_{\# \cN(1) = 1} \right] = \e^{b^2 \beta^2/2 + b \psi(\beta) - 1} < 1,
    \end{equation*}
    which is equivalent to~\eqref{eq:b_values}.
    As a consequence, $\Expec{W_\infty(\beta)^{-b}} < \infty$ for any $b$ satisfying~\eqref{eq:b_values}.
    Then, for any $t \in [0, \infty]$, by Markov's inequality and Jensen's inequality (using that $W_t(\beta) = \E \left[ W_\infty(\beta) \middle| \cF_t \right]$),    
    \begin{equation*}
        \P(W_t(\beta) \leq x) \leq x^b \Expec{W_t(\beta)^{-b}} \leq x^b \Expec{W_\infty(\beta)^{-b}},
    \end{equation*}
    which concludes the proof.
\end{proof}

\subsection{The extremal process}

We first recall some important results concerning the extremes of the BBM. For $t>0$, let
\begin{equation} \label{eq:def_m(t)}
    m(t) = \sqrt{2} t - \frac{3}{2 \sqrt{2}} \log t.
\end{equation}
The convergence of the maximal position has been proved by Bramson \cite{Bramson1983} and the limit has been identified by Lalley and Sellke \cite{LalleySellke1987}, who introduced the derivative martingale~\eqref{eq:definition_derivative_martingale} and showed that
\begin{equation}\label{eq:convergence_of_the_recentered_maximum}
    \max_{u \in \cN(t)} X_u(t) - m(t) \xrightarrow[t \to \infty]{} \frac{1}{\sqrt{2}} (G + \log(CZ_\infty)), \quad \text{in distribution},
\end{equation}
where $G$ is a standard Gumbel variable.
As a consequence, for any $\delta > 0$, the probability of the following event converges to $1$ as $t \to \infty$,
\begin{equation}\label{eq:bramson_estimate}
    \Lambda_{\delta, t} = \left\{ \max_{u \in \cN(t)} X_u(t) \leq m(t) + \delta \log t \right\}.
\end{equation}

Later, Arguin, Bovier and Kistler
\cite{ArguinBovierKistler2013} and A\"{i}dékon, Berestycki, Brunet and Shi \cite{AidekonShi2014} independently established the convergence in distribution
of the \emph{extremal process},
\begin{equation*}
    \cE_t = \sum_{u \in \cN(t)} \delta_{X_u(t) - m(t)}.
\end{equation*}
This convergence holds in the space of Radon measures on $\R$ endowed with the vague topology.
They obtained the following limit,
\begin{equation*}
    \cE_\infty = \sum_{i, j \geq 1} \delta_{p_i + \Delta_{ij} + \log(C Z_\infty)/\sqrt{2}},
\end{equation*}
where
\begin{itemize}
    \item $\cP = \sum_{i} \delta_{p_i}$ is a Poisson point process on $\R$ with intensity measure $\sqrt{2} \e^{-\sqrt{2}x} \d{x}$, independent of $Z_\infty$,
    \item $\cD_i = \sum_j \delta_{\Delta_{ij}}$ are \iid copies of a point process $\cD = \sum_j \delta_{\Delta_{j}}$ such that $\max \cD = 0$ and are independent of $\cP$ and $Z_\infty$.
\end{itemize}
In other words, the extremal process converges in distribution to a randomly shifted decorated Poisson point process, \ie a Poisson point process $\cP$ with a shift depending on $Z_\infty$, where each point is replaced by an independent copy of a decoration $\cD$.
One can deduce (see \cite[Proposition~3.16]{Chataignier2024} for details) that for any continuous function with compact support $f : \R \to [0, \infty)$,
\begin{equation}\label{eq:convergence_extremal_process_conditionally}
    \lim_{s \to \infty} \lim_{t \to \infty} \E \left[ \e^{-\cE_t(f)} \middle| \cF_s \right] = \E \left[ \e^{-\cE_\infty(f)} \middle| \cF_\infty \right], \quad \text{\as},
\end{equation}
where we emphasize that in $\cE_\infty$, only $Z_\infty$ is $\cF_\infty$-measurable, whereas $\cP$ and $(\cD_i)_{i\geq1}$ are independent of $\cF_\infty$.
Cortines, Hartung and Louidor \cite{CortinesHartungLouidor2019} studied the level sets of the decoration: in their Proposition~1.5, they obtained the following equivalent as $x \to \infty$,
\begin{equation}\label{eq:decoration_estimate}
    \E\cD([-x, 0]) \sim C \e^{\sqrt{2} x}.
\end{equation}
In what follows, we write $X \overset{d}{=} Y$ to indicate that $X$ and $Y$ have same distribution.

\begin{lemma}\label{lem:sum_gibbs_weights_ppp}
    If $\beta > \sqrt{2}$ and $(\xi_k)_{k\geq1}$ denotes the points of $\cE_\infty$, then $\Expec{\sum_j \e^{\beta \Delta_j}} < \infty$, $\sum_k \e^{\beta \xi_k} < \infty$ \as and
    \begin{equation}\label{eq:rewriting_sum_ppp}
        \sum_{k \geq 1} \e^{\beta \xi_k} \overset{d}{=} C \left(Z_\infty\right)^{\beta/\sqrt{2}} S_\beta,
    \end{equation}
    where $S_\beta = \sum_i \e^{\beta p_i}$ is a non-degenerate $\sqrt{2}/\beta$-stable random variable independent of $Z_\infty$.
\end{lemma}

\begin{proof}
    As done in \cite[Lemma~3.1]{Bonnefont2022}, we rewrite
    \begin{equation*}
        \Expec{\sum_{j \geq 1} \e^{\beta \Delta_j}} = \E \int_{-\infty}^0 \e^{\beta x} \cD(\d{x}) = \E \int_0^\infty \beta \e^{-\beta y} \cD([-y, 0]) \d{y}.
    \end{equation*}
    Applying the Fubini-Tonelli theorem and~\eqref{eq:decoration_estimate}, we see that the above quantity is finite.
    Now,
    \begin{equation}\label{eq:rewriting_sum_exp}
        \sum_{k \geq 1} \e^{\beta \xi_k} = (C Z_\infty)^{\beta/\sqrt{2}} \sum_{i \geq 1} \e^{\beta(p_i + X_i)},
    \end{equation}
    where $X_i = \frac{1}{\beta} \log \sum_j \e^{\beta \Delta_{ij}}$.
    The random variables $X_i$ for $i \geq 1$ are \iid, independent of $\cP$ and satisfy
    \begin{equation*}
        \Expec{\e^{\sqrt{2} X_1}} = \Expec{\Biggl(\sum_{j \geq 1} \e^{\beta \Delta_j}\Biggr)^{\sqrt{2}/\beta}} \leq \Expec{\sum_{j \geq 1} \e^{\beta \Delta_j}} < \infty.
    \end{equation*}
    By \cite[Proposition~8.7.a]{BolthausenSznitman2002},
    \begin{equation}\label{eq:rewriting_poisson_distribution}
        \sum_{i \geq 1} \delta_{p_i + X_i} \overset{d}{=} \sum_{i \geq 1} \delta_{p_i + \frac{1}{\sqrt{2}} \log \Expec{\e^{\sqrt{2} X_1}}}.
    \end{equation}
    By~\eqref{eq:rewriting_sum_exp} and~\eqref{eq:rewriting_poisson_distribution}, we have $\sum_k \e^{\beta \xi_k} < \infty$ \as if and only if $\sum_i \e^{\beta p_i} < \infty$ \as
    Computing the expectations, we see that, with probability $1$, we have $\cP([0, \infty)) < \infty$ and $\sum_i \e^{\beta p_i} \mathds{1}_{p_i < 0} < \infty$, so $\sum_i \e^{\beta p_i} < \infty$.
    Besides, since $Z_\infty$ is independent of $\cP$ and $(\cD_i)_{i \geq 1}$, \eqref{eq:rewriting_sum_exp} and~\eqref{eq:rewriting_poisson_distribution} yield
    \begin{equation*}
        \sum_{k \geq 1} \e^{\beta \xi_k} \overset{d}{=} \left(C Z_\infty \Expec{\e^{\sqrt{2} X_1}}\right)^{\beta/\sqrt{2}} \sum_{i \geq 1} \e^{\beta p_i}.
    \end{equation*}
    It remains to show that $\sum_i \e^{\beta p_i}$ is stable.
    Let $\lambda > 0$.
    Applying successively the exponential formula for Poisson point processes (see \eg \cite[Section~0.5]{Bertoin1996}) and the change of variable $y = \e^{\beta x}$, we obtain
    \begin{equation*}
        \Expec{\exp\left(-\lambda \sum_{i \geq 1} \e^{\beta p_i}\right)} = \exp\left(-\int_\R (1 - \e^{-\lambda \e^{\beta x}}) \sqrt{2} \e^{-\sqrt{2} x} \d{x}\right) = \exp\left(-\sqrt{2}/\beta \int_0^\infty (1 - \e^{-\lambda y}) y^{-1-\sqrt{2}/\beta} \d{y}\right).
    \end{equation*}
    We recognize the Laplace transform of a non-degenerate $\sqrt{2}/\beta$-stable law (see \eg \cite{Sato1999}).
\end{proof}

We now define two point processes on $\R^2$ by
\begin{equation}
    \cE_t^* = \sum_{u \in \cN(t)} \delta_{(X_u(t) - m(t), W_\infty^{(u, t)}(\beta))} \quad \text{and} \quad \cE_\infty^* = \sum_{k \geq 1} \delta_{\left(\xi_k, W_k\right)},
\end{equation}
where $W_k$ for $k \geq 1$ are \iid copies of $W_\infty(\beta)$, independent of $\cE_\infty = \sum_{k} \delta_{\xi_k}$ and of the BBM.
Let $f : \R^2 \to \R$ be a bounded continuous function.
Assume that there exists $a \in \R$ such that $f(x, y) = 0$ as soon as $x \leq a$.
Under these conditions and for the BRW, Iksanov, Kolesko and Meiners showed in \cite[Lemma~5.2]{IksanovKoleskoMeiners2020} that $\cE_t^*(f)$ converges in distribution to $\cE_\infty^*(f)$ as $t \to \infty$.
Their result easily extends to BBM (see \cite[Lemma~5.9]{Chataignier2024} for details).
Furthermore, thanks to~\eqref{eq:convergence_extremal_process_conditionally}, one can adapt their proof to obtain the following joint convergence (see \cite[Remark~5.10]{Chataignier2024} for details), for any $\beta \geq 0$,
\begin{equation}\label{eq:joint_convergence}
    (\cE_t^*(f), W_t(\beta)) \xrightarrow[t \to \infty]{} (\cE_\infty^*(f), W_\infty(\beta)), \quad \text{in distribution}.
\end{equation}

\subsection{Brownian estimates}

Let $(B_t)_{t \geq 0}$ be a standard Brownian motion on $\R$, $\alpha \in (0, 1/2)$, and define, $f_t(s) = s^\alpha \wedge (t - s)^\alpha$ for $t \geq s \geq 0$.
By \cite[Lemma~2.8]{KimLubetzkyZeitouni2023}, there exist $C, c > 0$ such that, for all $t \geq 0$, $x, y \in [1, \sqrt{t}]$,
\begin{equation}\label{eq:ballot_theorem}
    \frac{c x}{(t + 1)^{1/2}} \leq \P \left( \sup_{s \in [0, t]} (B_s + f_t(s)) \leq x \right) \leq \frac{C x}{(t + 1)^{1/2}},
\end{equation}
and
\begin{equation}\label{eq:ballot_theorem_plus}
    \frac{c x y}{(t + 1)^{3/2}} \leq \P \left( \sup_{s \in [0, t]} (B_s + f_t(s)) \leq x, B_t \in [x - y, x - y + 1] \right) \leq \frac{C x y}{(t + 1)^{3/2}}.
\end{equation}
Besides, the following bound will be useful.

\begin{lemma}\label{lem:brownian_estimate}
    Fix $\delta > 0$.
    There exists $C_\delta > 0$ such that, for all $x, y \geq 0$ and $t \geq 1$,
    \begin{equation}\label{eq:brownian_estimate}
        \P \left( \sup_{s \in [0, t]} B_s \leq x, B_t \in [x - y, x - y + \delta] \right) \leq \frac{C_\delta x y}{t^{3/2}} \left( \e^{-(x - y)^2/2t} + \e^{-(x - y + \delta)^2/2t} \right).
    \end{equation}
\end{lemma}

\begin{proof}
    By \cite[Lemma~2]{Bramson1978}, $\P \left( \sup_{s \in [0, t]} B_s \leq x \middle| B_t \right) \leq 2 x (x - B_t)^+/t$, where we write $a^+ = \max(a, 0)$.
    It follows that the left-hand side of~\eqref{eq:brownian_estimate} is bounded by
    \begin{equation*}
        \Expec{\frac{2 x (x - B_t)^+}{t} \mathds{1}_{\{B_t \in [x - y, x - y + \delta]\}}} \leq \frac{2 x y}{t} \P(B_t \in [x - y, x - y + \delta]).
    \end{equation*}
    Next, we can \eg use the bounds
    \begin{equation*}
        \P(B_t \in [x - y, x - y + \delta]) \leq \begin{cases}
            C \delta \e^{-(x - y)^2/2t}/\sqrt{t} & \text{if } x - y \geq 0, \\
            C \delta \e^{-(x - y + \delta)^2/2t}/\sqrt{t} & \text{if } x - y + \delta \leq 0, \\
            C/\sqrt{t} & \text{otherwise}.
        \end{cases}
    \end{equation*}
    Each of these terms is bounded by the right-hand side of~\eqref{eq:brownian_estimate} since $t \geq 1$.
\end{proof}

\section{Proof of Theorem~\ref{thm:asymptotics_of_the_overlap_distribution}}
\label{sec:typical}

Let $\beta \in [0, \sqrt{2})$ and $a \in (0, 1)$.
Recall that $\psi(\beta) = 1 + \beta^2/2$ and that the additive martingale $W_t(\beta)$ is defined in~\eqref{eq:definition_additive_martingale}.
In what follows, it will be convenient to use the following rewriting of the overlap distribution defined in~\eqref{eq:def_overlap_dist},
\begin{align}
    \nu_{\beta, t}([a, 1]) &= \frac{1}{W_t(\beta)^2} \sum_{u, v \in \cN(t), u \wedge v \geq at} \e^{\beta(X_u(t) + X_v(t)) - 2 \psi(\beta) t} \nonumber \\
    &= \frac{\e^{(\beta^2 - 1) at}}{W_t(\beta)^2} \sum_{u \in \cN(at)} \e^{2\beta X_u(at) - \psi(2\beta)at} W_{t - at}^{(u, at)}(\beta)^2, \label{eq:overlap_rewriting}
\end{align}
where we recall that
\begin{equation}\label{eq:definition_shifted_martingale}
    W_t^{(u, s)}(\beta) = W_t(\beta) \circ \Theta_{u, s} = \sum_{v \in \cN(s+t), v \geq u} \e^{\beta (X_v(s + t) - X_u(s)) - \psi(\beta)t}.
\end{equation}

\subsection{Upper temperature}\label{sct:first_regime}

\begin{proof}[Proof of Theorem~\ref{thm:asymptotics_of_the_overlap_distribution}.\ref{it:asymptotics_of_the_overlap_distribution_1}]
    Define
    \begin{equation*}
        F_t = \sum_{u \in \cN(at)} \e^{2\beta X_u(at) - \psi(2\beta)at} W_{t - at}^{(u, at)}(\beta)^2,
    \end{equation*}
    so that $\e^{(1 - \beta^2)at} \nu_{\beta, t}([a, 1]) = F_t W_t(\beta)^{-2}$, by~\eqref{eq:overlap_rewriting}.
    It suffices to show that
    \begin{equation*}
        F_t \xrightarrow[t \to \infty]{} F_\infty = W_\infty(2\beta) \Expec{W_\infty(\beta)^2}, \quad \text{\as}
    \end{equation*}
    
    \textbf{Replacement along a partition of time:}
    Let us fix the following partition of the time $[0, \infty)$,
    \begin{equation*}
        T = \bigcup_{n \geq 0} \left(\frac{1}{n+1} \Z\right) \cap [n, n+1].
    \end{equation*}
    First, we justify that, along this partition, we can replace $F_t$ with
    \begin{equation*}
        \bar{F}_t = \sum_{u \in \cN(at)} \e^{2\beta X_u(at) - \psi(2\beta)at} W_\infty^{(u, at)}(\beta)^2.
    \end{equation*}
    We use first branching property at time $at$, and then factorize the square and apply Cauchy--Schwarz inequality:
    \begin{align}
        \E \left[ |F_t - \bar{F}_t| \middle| \cF_{at} \right] 
        & \leq W_{at}(2\beta) \E|W_{t - at}(\beta)^2 - W_\infty(\beta)^2| \nonumber \\
        & \leq W_{at}(2\beta) 
        \left( \Expec{W_{t - at}(\beta)^2} + \Expec{W_\infty(\beta)^2} \right)^{1/2}
        \Expec{(W_{t - at}(\beta) - W_\infty(\beta))^2}^{1/2} \nonumber \\
        & \leq C W_{at}(2\beta) \e^{-(2\psi(\beta) - \psi(2\beta))s} \nonumber
    \end{align}
    using Lemma~\ref{lem:additive_martingale_convergence_in_Lp} in the last inequality.
    By~\eqref{eq:pc-cp}, it follows that $\E|F_t - \bar{F}_t| \leq \e^{-(1 - \beta^2)(t-at)}$.
    Hence, by the Borel--Cantelli lemma,
    \begin{equation} \label{eq:step_1}
        |F_t - \bar{F}_t| \xrightarrow[t \to \infty, t \in T]{} 0, \quad \text{\as}
    \end{equation}
    
    \textbf{Almost sure convergence along a partition of time:}
    Here, we show that $F_t$ converges \as to $F_\infty$ along the partition $T$.
    In view of the previous section, it suffices to show that the following process converges \as to $0$ along the partition $T$,
    \begin{equation*}
        \Delta_t = \bar{F}_t - W_{at}(2\beta) \Expec{W_\infty(\beta)^2} = \sum_{u \in \cN(at)} \e^{2\beta X_u(at) - \psi(2\beta)at} \left(W_\infty^{(u, at)}(\beta)^2 - \Expec{W_\infty(\beta)^2}\right).
    \end{equation*}
    Let us fix $p \in [1, 2]$ such that $p < 1/2\beta^2$.
    By Lemma~\ref{lem:von_Bahr},
    \begin{equation*}
        \E \left[ \left|\Delta_t\right|^p \middle| \cF_{at} \right] \leq W_{at}(p2\beta) \E\left|W_\infty(\beta)^2 - \Expec{W_\infty(\beta)^2}\right|^p \e^{-(p\psi(2\beta) - \psi(p2\beta))at}.
    \end{equation*}
    The variable $W_\infty(\beta)^2$ is in $L^p$ because $2p<2/\beta^2$, hence so is $W_\infty(\beta)^2 - \Expec{W_\infty(\beta)^2}$.
    Hence,
    \begin{equation*}
        \E\left|\Delta_t\right|^p \leq C \e^{-(p\psi(2\beta) - \psi(p2\beta))at}. 
    \end{equation*}
    By~\eqref{eq:pc-cp} and by the Borel--Cantelli lemma,
    \begin{equation} \label{eq:step_2}
        \Delta_t \xrightarrow[t \to \infty, t \in T]{} 0, \quad \text{\as}
    \end{equation}
    
    \textbf{Extension of the almost sure convergence to continuum:}
    Let us denote $(t_n)_{n \geq 0}$ the sequence of the elements of $T$ in increasing order.
    Here, we show that
    \begin{equation}\label{eq:extension_to_continuum}
        \sup_{t \in [t_n, t_{n+1}]} |F_t - F_{t_n}| \xrightarrow[n \to \infty]{} 0, \quad \text{\as}
    \end{equation}
    We have
    \begin{equation}\label{eq:extension_to_continuum_bound_Sn}
        \sup_{t \in [t_n, t_{n+1}]} |F_t - F_{t_n}| \leq \sum_{u \in \cN(at_n)} \e^{2\beta X_u(at_n) - \psi(2\beta)at_n} S_n^{(u)},
    \end{equation}
    where
    \begin{equation*}
        S_n^{(u)} = \sup_{t \in [t_n, t_{n+1}]} \left|\sum_{v \in \cN(at), v \geq u} \e^{2\beta(X_v(at) - X_u(at_n)) - \psi(2\beta)(at - at_n)} W_{t - at}^{(v, at)}(\beta)^2 - W_{t_n - at_n}^{(u, at_n)}(\beta)^2\right|.
    \end{equation*}
    Note that, by the branching property, for any $u \in \cN(at_n)$, we have $\E \left[ S_n^{(u)} \middle| \cF_{at_n} \right] = \Expec{S_n}$, where
    \begin{equation*}
        S_n = \sup_{s \in [0, t_{n+1} - t_n]} \left|\sum_{v \in \cN(as)} \e^{2\beta X_v(as) - \psi(2\beta)as} W_{s + t_n - a(s + t_n)}^{(v, as)}(\beta)^2 - W_{t_n - at_n}(\beta)^2\right|.
    \end{equation*}
    Therefore,
    \begin{equation*}
        \sum_{u \in \cN(at_n)} \e^{2\beta X_u(at_n) - \psi(2\beta)at_n} \E \left[ S_n^{(u)} \middle| \cF_{at_n} \right] = W_{at_n}(2\beta) \Expec{S_n}.
    \end{equation*}
    By Lemma~\ref{lem:Sn} below and since the additive martingale converges \as, the above process converges \as to $0$ as $n \to \infty$.
    To obtain~\eqref{eq:extension_to_continuum}, it is then sufficient to show that
    \begin{equation}\label{eq:Delta_n}
        V_n = \sum_{u \in \cN(at_n)} \e^{2\beta X_u(at_n) - \psi(2\beta)at_n} \left(S_n^{(u)} - \E \left[ S_n^{(u)} \middle| \cF_{at_n} \right] \right) \xrightarrow[n \to \infty]{} 0, \quad \text{\as}
    \end{equation}
    Let us fix $p \in (1, 2]$ such that $p < 1/2\beta^2$.
    By Lemma~\ref{lem:von_Bahr},
    \begin{equation*}
        \E \left[ V_n^p \middle| \cF_{at_n} \right] \leq W_{at_n}(p2\beta) \E|S_n - \Expec{S_n}|^p \e^{-(p\psi(2\beta) - \psi(p2\beta))at_n}.
    \end{equation*}
    By Lemma~\ref{lem:Sn} below, the process $S_n - \Expec{S_n}$ converges in $L^p$ to $0$.
    In particular,
    \begin{equation*}
        \Expec{V_n^p} \leq C \e^{-(p\psi(2\beta) - \psi(p2\beta))at_n}.
    \end{equation*}
    By~\eqref{eq:pc-cp} and by the Borel-Cantelli lemma, the \as convergence~\eqref{eq:Delta_n} holds, and therefore \eqref{eq:extension_to_continuum} as well.
    Together with \eqref{eq:step_1} and \eqref{eq:step_2}, this concludes the proof.
\end{proof}

\begin{lemma}\label{lem:Sn}
    For any $p \in [1, 1/\beta^2)$, the process $S_n$ converges in $L^p$ to $0$.
\end{lemma}

\begin{proof}
    We first prove that, on the event $\Omega'$ where $W_t(\beta)$ converges as $t \to \infty$, the process $S_n$ converges to $0$.
    Let us fix a realization $\omega \in \Omega'$ and denote by $\tau(\omega)$ the first time of branching of the BBM.
    There exists some rank $n_0(\omega)$ such that, for any $n \geq n_0(\omega)$, we have $t_{n+1} - t_n < \tau(\omega)$ and then
    \begin{align}
        S_n(\omega) &= \sup_{s \in [0, t_{n+1} - t_n]} \left|\e^{2\beta X_\varnothing(as)(\omega) - \psi(2\beta)as} W_{s + t_n - a(s + t_n)}^{(\varnothing, as)}(\beta)^2(\omega) - W_{t_n - at_n}(\beta)^2(\omega)\right| \nonumber \\
        &= \sup_{s \in [0, t_{n+1} - t_n]} \left|\e^{(2\psi(\beta) - \psi(2\beta))as} W_{s + t_n - at_n}(\beta)^2(\omega) - W_{t_n - at_n}(\beta)^2(\omega)\right|. \nonumber
    \end{align}
    Since $W_t(\beta)(\omega)$ converges as $t \to \infty$, the sequence $S_n(\omega)$ converges to $0$. 
    Since the event $\Omega'$ has probability $1$, we deduce that the process $S_n$ converges \as to $0$.
    
    In order to obtain a domination, we bound
    \begin{equation*}
        S_n \leq \sup_{s \in [0, 1]} \sum_{v \in \cN(as)} \e^{2\beta X_v(as) - \psi(2\beta)as} W_{s + t_n - a(s + t_n)}^{(v, as)}(\beta)^2 + W_{t_n - at_n}(\beta)^2.
    \end{equation*}
    Note that
    \begin{equation*}
        \sum_{v \in \cN(as)} \e^{2\beta X_v(as) - \psi(2\beta)as} W_{s + t_n - a(s + t_n)}^{(v, as)}(\beta)^2 \leq \e^{(2\psi(\beta) - \psi(2\beta))as} W_{s + t_n - at_n}(\beta)^2.
    \end{equation*}
    Hence,
    \begin{equation}\label{eq:Sn_domination}
        S_n \leq \left(\sup_{s \in [0, 1]} \e^{(2\psi(\beta) - \psi(2\beta))as} + 1\right) \sup_{s \in [0, 1]} W_{s + t_n - at_n}(\beta)^2.
    \end{equation}
    By Lemma~\ref{lem:additive_martingale_convergence_in_Lp}, the process $W_t(\beta)$ is bounded in $L^{2p}$ for any $p \in (1, 1/\beta^2)$.
    By Doob's maximal inequality, the domination in~\eqref{eq:Sn_domination} is in $L^p$.
    Thus, by dominated convergence, the process $S_n$ converges in $L^p$ to $0$.
\end{proof}

\subsection{Intermediate temperature}

\begin{proof}[Proof of Theorem~\ref{thm:asymptotics_of_the_overlap_distribution}.\ref{it:asymptotics_of_the_overlap_distribution_2}]
    Let us fix $\beta = \sqrt{2}/2$.
    By~\eqref{eq:overlap_rewriting}, we can rewrite
    \begin{equation*}
        \sqrt{at} \e^{at/2} \nu_{\beta, t}([a, 1]) = \sqrt{at} (G_t + R_t) W_t(\beta)^{-2},
    \end{equation*}
    where
    \begin{align}
        G_t & = \sum_{u \in \cN(at)} \e^{\sqrt{2} X_u(at) - 2at} W_\infty^{(u, at)}(\beta)^2, \nonumber \\
        R_t & = \sum_{u \in \cN(at)} \e^{\sqrt{2} X_u(at) - 2at} \left(W_{t - at}^{(u, at)}(\beta)^2 - W_\infty^{(u, at)}(\beta)^2\right). \nonumber
    \end{align}
    Since the additive martingale converges \as, it suffices to obtain the convergence in probability of $\sqrt{at}(G_t + R_t)$ toward $\sqrt{2/\pi} Z_\infty \Expec{W_\infty(\beta)^2}$.
    Let us justify that we can ignore the term $R_t$.
    By the triangle inequality and the branching property,
    \begin{equation*}
        \sqrt{at} \E \left[ |R_t| \middle| \cF_{at} \right] \leq \sqrt{at} W_{at}(\sqrt{2}) \E\left|W_{t - at}(\beta)^2 - W_\infty(\beta)^2\right|.
    \end{equation*}
    By~\eqref{eq:convergence_critical_additive_martingale} and Lemma~\ref{lem:additive_martingale_convergence_in_Lp}, the above process converges in probability to $0$ as $t \to \infty$, so does $\sqrt{at} R_t$.
    
    It remains to study $G_t$.
    Let us fix $\theta \in \R$.
    We use the conditional characteristic function
    \begin{equation*}
        \E \left[ \exp \left( \iu \theta \sqrt{at} G_t \right) \middle| \cF_{at} \right] = \prod_{u \in \cN(at)} \varphi(\lambda^{(u, at)}),
    \end{equation*}
    where
    \begin{equation*}
        \lambda^{(u, at)} = \theta \sqrt{at} \e^{\sqrt{2} X_u(at) - 2at}, \quad \varphi(\lambda) = \Expec{\e^{\iu \lambda W_\infty(\beta)^2}} = 1 + \iu \lambda \Expec{W_\infty(\beta)^2} + r_1(\lambda),
    \end{equation*}
    and $r_1(\lambda) = o(\lambda)$ as $\lambda \to 0$.
    In order to use the above expansion, we need to restrict to an event with high probability where $\max_{u \in \cN(at)} \lambda^{(u, at)}$ converges to $0$ as $t \to \infty$.
    We can choose the event $\Lambda = \Lambda_{\delta, at}$ defined in~\eqref{eq:bramson_estimate}, with an arbitrary $\delta \in (0,1/\sqrt{2})$.
    On this event, for $t$ large enough and for all $u \in \cN(at)$, we have $\varphi(\lambda^{(u, at)}) \in \C \setminus (-\infty, 0]$.
    This allows us to use the principal complex logarithm $\Log$ and to write
    \begin{align}
        \E \left[ \exp \left( \iu \theta \sqrt{at} G_t \right) \middle| \cF_{at} \right] \mathds{1}_\Lambda &= \exp \left( \sum_{u \in \cN(at)} \Log(1 + \iu \lambda^{(u, at)} \Expec{W_\infty(\beta)^2} + r_1(\lambda^{(u, at)})) \right) \mathds{1}_\Lambda \nonumber \\
        &= \exp \left( \iu \theta \sqrt{at} W_{at}(\sqrt{2}) \Expec{W_\infty(\beta)^2} + \sum_{u \in \cN(at)} r_2(\lambda^{(u, at)}) \right) \mathds{1}_\Lambda, \label{eq:conditional_characteristic_function_taylor_expansion}
    \end{align}
    where $r_2(\lambda) = o(\lambda)$ as $\lambda \to 0$.
    By choice of $\Lambda$, with probability $1$, we have
    \begin{equation}\label{eq:switching_sum_and_o}
        \sum_{u \in \cN(at)} r_2(\lambda^{(u, at)}) \mathds{1}_\Lambda = o \left( \sum_{u \in \cN(at)} \lambda^{(u, at)} \right) \mathds{1}_\Lambda = o \left( \sqrt{at} W_{at}(\sqrt{2}) \right) \mathds{1}_\Lambda.
    \end{equation}
    By~\eqref{eq:conditional_characteristic_function_taylor_expansion} and~\eqref{eq:switching_sum_and_o}, with probability $1$, we have
    \begin{equation*}
        \E \left[ \exp \left( \iu \theta \left( \sqrt{at} G_t - \sqrt{at} W_{at}(\sqrt{2}) \Expec{W_\infty(\beta)^2} \right) \right) \middle| \cF_{at} \right] \mathds{1}_\Lambda = \exp \left( o \left( \sqrt{at} W_{at}(\sqrt{2}) \right) \right) \mathds{1}_\Lambda.
    \end{equation*}
    By~\eqref{eq:convergence_critical_additive_martingale}, the above process converges in probability to $1$.
    This being true for any $\theta \in \R$, we deduce that $\sqrt{at} G_t - \sqrt{2/\pi} Z_\infty \Expec{W_\infty(\beta)^2}$ converges in probability to $0$ as $t \to \infty$, which concludes the proof.
\end{proof}

\subsection{Lower temperature}

Let $\beta \in (\sqrt{2}/2, \sqrt{2})$.
Here we follow the techniques from \cite[Theorem~2.5]{IksanovKoleskoMeiners2020}, where Iksanov, Kolesko and Meiners computed the fluctuations of the additive martingales of the BRW, using the convergence of the extremal process $\cE_t$.
The main differences are that we have to deal with this convergence jointly with the additive martingale and that we identify a stable distribution at the limit.
Recall that $m(t) = \sqrt{2}t - 3 \cdot 2^{-3/2} \log t$.
We rewrite
\begin{equation*}
    (at)^{3 \beta/\sqrt{2}} \e^{(\sqrt{2} - \beta)^2 at} \nu_{\beta,t}([a, 1]) = H_t W_t(\beta)^{-2},
\end{equation*}
where
\begin{equation*}
    H_t = \sum_{u \in \cN(at)} \e^{2 \beta (X_u(at) - m(at))} W_{t - at}^{(u, at)}(\beta)^2.
\end{equation*}
Let $\ell \in \R$, which has here to be thought as a large negative number.
We will use the following continuous approximation of $\mathds{1}_{[\ell, \infty)}$,
\begin{equation*}
    \chi_\ell^+(x) = \begin{cases}
        0 & \text{if } x \leq \ell, \\
        x - \ell &\text{if } \ell \leq x \leq \ell+1, \\
        1 & \text{if } x \geq \ell+1,
    \end{cases}
\end{equation*}
as well as $\chi_\ell^- = 1 - \chi_\ell^+$.
We decompose
\begin{align}
    H_t &= \sum_{u \in \cN(at)} \e^{2 \beta (X_u(at) - m(at))} \left(W_{t - at}^{(u, at)}(\beta)^2 - W_\infty^{(u, at)}(\beta)^2\right) \nonumber \\
    &\qquad + \sum_{u \in \cN(at)} \e^{2 \beta (X_u(at) - m(at))} W_\infty^{(u, at)}(\beta)^2 \chi_\ell^+(X_u(at) - m(at)) \nonumber \\
    &\qquad + \sum_{u \in \cN(at)} \e^{2 \beta (X_u(at) - m(at))} W_\infty^{(u, at)}(\beta)^2 \chi_\ell^-(X_u(at) - m(at)) \nonumber \\
    &= R_{t, \ell}^1 + H_{t, \ell} + R_{t, \ell}^2. \nonumber
\end{align}
Below, Lemma~\ref{lem:renormalized_subcritical_overlap_replacement} and Lemma~\ref{lem:renormalized_subcritical_overlap_contribution} state that $R_{t, \ell}^1$ and $R_{t, \ell}^2$ are negligible as $t \to \infty$ and $\ell \to -\infty$.
Lemma~\ref{lem:renormalized_subcritical_overlap_limit} gives a first characterization of the limit distribution.

\begin{lemma}\label{lem:renormalized_subcritical_overlap_replacement}
    For all $\ell \in \R$, $R_{t, \ell}^1$ converges in probability to $0$ as $t \to \infty$.
\end{lemma}

\begin{proof}
    We consider $p \in (0,1]$ such that $\sqrt{2}/2\beta < p < 1/\beta^2$ which is possible because $\beta \in (\sqrt{2}/2, \sqrt{2})$.
    Using the subadditivity of $x \mapsto x^p$ on $[0, \infty)$, we have
    \begin{equation*}
        \E \left[ |R_{t, \ell}^1|^p \middle| \cF_{at} \right]
        \leq \Expec{ \left|W_{t - at}(\beta)^2 - W_\infty(\beta)^2\right|^p } \sum_{u \in \cN(at)} \e^{p2\beta(X_u(at) - m(at))}.
    \end{equation*}
    Since $p <1/\beta^2$, the expectation on the right-hand side tends to 0 (see Lemma~\ref{lem:additive_martingale_convergence_in_Lp}).
    By~\eqref{eq:supercritical_W}, the sum on the right-hand side converges in distribution to a finite limit because $p 2\beta > \sqrt{2}$.
    Therefore, by Slutsky's theorem, the above quantity converges in distribution to $0$ as $t \to \infty$.
    Then, $R_{t, \ell}^1$ converges in probability to $0$ as $t \to \infty$.
\end{proof}

\begin{lemma}\label{lem:renormalized_subcritical_overlap_contribution}
    For any $\delta > 0$, $\lim_{\ell \to -\infty} \limsup_{t \to \infty} \P(R_{t,\ell}^2 > \delta) = 0$.
\end{lemma}

\begin{proof}
    Note that if a process $(X_{t, \ell})_{t \geq 0, \ell \in \R}$ takes its values in $[0, \infty)$, then the following two statements are equivalent
    \begin{enumerate}
        \item\label{step:renormalized_subcritical_overlap_2_2_1} for any $\delta > 0$, $\lim_{\ell \to -\infty} \limsup_{t \to \infty} \P(X_{t, \ell} \geq \delta) = 0$,
        \item\label{step:renormalized_subcritical_overlap_2_2_2} $\lim_{\ell \to -\infty} \limsup_{t \to \infty} \Expec{X_{t, \ell} \wedge 1} = 0$.
    \end{enumerate}
    We consider $p \in (0,1]$ such that $\sqrt{2}/2\beta < p < 1/\beta^2$.
    By Jensen's inequality,
    \begin{equation*}
        \Expec{(R_{t, \ell}^2)^p \wedge 1} \leq \Expec{\E \left[ (R_{t, \ell}^2)^p \middle| \cF_{at} \right] \wedge 1}.
    \end{equation*}
    But, using the subadditivity of $x \mapsto x^p$ on $[0, \infty)$,
    \begin{equation*}
        \E \left[ (R_{t, \ell}^2)^p \middle| \cF_{at} \right] \leq \Expec{W_\infty(\beta)^{2p}} \sum_{u \in \cN(at)} \e^{p2 \beta(X_u(at) - m(at))} \chi_\ell^-(X_u(at) - m(at))^p 
        = C \cE_{at}(\e^{p2 \beta x} \chi_\ell^-(x)^p),
    \end{equation*}
    using that $W_\infty(\beta) \in L^{2p}$ because $p < 1/\beta^2$ (see Lemma~\ref{lem:additive_martingale_convergence_in_Lp}).
    Since \ref{step:renormalized_subcritical_overlap_2_2_1} and \ref{step:renormalized_subcritical_overlap_2_2_2} are equivalent, it suffices to obtain that, for any $\delta > 0$,
    \begin{equation*}
        \lim_{\ell \to -\infty} \limsup_{t \to \infty} \P(\cE_{at}(\e^{p2 \beta x} \chi_\ell^-(x)^p) \geq \delta) = 0.
    \end{equation*}
    This is proved by Bonnefont in \cite[Proposition~A.2]{Bonnefont2022} (this requires $p2\beta > \sqrt{2}$, which is ensured by our choice of $p$).
\end{proof}

\begin{lemma}\label{lem:renormalized_subcritical_overlap_limit}
    We have
    \begin{equation*}
        (at)^{3 \beta/\sqrt{2}} \e^{(\sqrt{2} - \beta)^2 at} \nu_{\beta, t}([a, 1]) \xrightarrow[t \to \infty]{} X_{\mathrm{over}} = \frac{1}{W_\infty(\beta)^2} \sum_{k \geq 1} \e^{2 \beta \xi_k} W_k^2, \quad \text{in distribution},
    \end{equation*}
    where $W_k$, $k \geq 1$, are \iid copies of $W_\infty(\beta)$, independent of $W_\infty(\beta)$ and $\cE_\infty$.
\end{lemma}

\begin{proof}
    We first check that $X_{\mathrm{over}}$ is well-defined by showing
    \begin{equation} \label{eq:sum_finite}
        \sum_{k \geq 0} \e^{2 \beta \xi_k} W_k^2 < \infty, 
        \quad \text{almost surely}.
    \end{equation}
    To see this, fix $p \in (0,1]$ such that $\sqrt{2}/2\beta < p < 1/\beta^2$. 
    Then, by subadditivity,
    \begin{equation*}
        \E \left[ \left(\sum_{k \geq 1} \e^{2 \beta \xi_k} W_k^2\right)^p \middle| \cE_\infty \right] 
        = \sum_{k \geq 1} \e^{p2 \beta \xi_k} \Expec{W_\infty(\beta)^{2p}},
    \end{equation*}
    which is \as finite, by Lemma~\ref{lem:sum_gibbs_weights_ppp} and Lemma~\ref{lem:additive_martingale_convergence_in_Lp}, proving~\eqref{eq:sum_finite}.

    Now, in order to prove the lemma, by Lemma~\ref{lem:renormalized_subcritical_overlap_replacement}, Lemma~\ref{lem:renormalized_subcritical_overlap_contribution} and \cite[Theorem~3.2]{Billingsley1999}, it is sufficient to show that, for all $\ell \in \R$,
    \begin{equation}\label{eq:renormalized_subcritical_overlap_2_1}
        \frac{H_{t, \ell}}{W_t(\beta)^2} \xrightarrow[t \to \infty]{} \frac{1}{W_\infty(\beta)^2} \sum_{k \geq 1} \e^{2 \beta \xi_k} \chi_\ell^+(\xi_k) W_k^2, \quad \text{in distribution},
    \end{equation}
    and that
    \begin{equation}\label{eq:renormalized_subcritical_overlap_2_2}
        \frac{1}{W_\infty(\beta)^2} \sum_{k \geq 1} \e^{2 \beta \xi_k} \chi_\ell^+(\xi_k) W_k^2 \xrightarrow[\ell \to -\infty]{} X_{\mathrm{over}}, \quad \text{in distribution}.
    \end{equation}
    Note that~\eqref{eq:renormalized_subcritical_overlap_2_2} is a direct consequence of~\eqref{eq:sum_finite}.
    To prove~\eqref{eq:renormalized_subcritical_overlap_2_1}, we introduce the test function $f(x,y) = \e^{2\beta x} \chi_\ell^+(x) y^2$.
    By~\eqref{eq:joint_convergence}%
    \footnote{Note that $f$ is not bounded as required in~\eqref{eq:joint_convergence}. However, one can first apply~\eqref{eq:joint_convergence} to $f \wedge M$ for some $M> 0$, and then let $M \to \infty$, using that $\cE^*(f) < \infty$ \as by~\eqref{eq:sum_finite}.}, 
    $(\cE_{at}^*(f), W_{at}(\beta))$ converges in distribution to $(\cE_\infty^*(f), W_\infty(\beta))$.
    By Slutsky's theorem, the same is true for $(\cE_{at}^*(f), W_t(\beta))$ and hence we get~\eqref{eq:renormalized_subcritical_overlap_2_1}.
\end{proof}

We are now able to finish the proof of Theorem~\ref{thm:asymptotics_of_the_overlap_distribution}.

\begin{proof}[Proof of Theorem~\ref{thm:asymptotics_of_the_overlap_distribution}.\ref{it:asymptotics_of_the_overlap_distribution_3}] 
    By definition of $\cE_\infty$, we can rewrite
    \begin{equation}\label{eq:renormalized_subcritical_overlap_2_equal_in_distribution_1}
        X_{\mathrm{over}} = \frac{(C Z_\infty)^{2\beta/\sqrt{2}}}{W_\infty(\beta)^2} \sum_{i\geq 1} \e^{2\beta(p_i + X_i)},
    \end{equation}
    where $X_i = \frac{1}{2\beta} \log \sum_j \e^{2 \beta \Delta_{ij}} W_{ij}^2$ and where $W_{ij}$ are \iid copies of $W_\infty(\beta)$, independent of $\cF_\infty$, $\cP$, $(\cD_i)_{i \geq 1}$.
    Now, since the variables $X_i$ for $i \geq 1$ are \iid and independent of $\cP$, it follows from \cite[Proposition~8.7.a]{BolthausenSznitman2002} that
    \begin{equation}\label{eq:renormalized_subcritical_overlap_2_Bolthausen_Sznitman}
        \sum_{i \geq 1} \delta_{p_i + X_i} \overset{d}{=} \sum_{i \geq 1} \delta_{p_i + \frac{1}{\sqrt{2}} \log \Expec{\e^{\sqrt{2} X_1}}},
    \end{equation}
    under the condition that $\E[\e^{\sqrt{2} X_1}] < \infty$, which we check as follows: taking $p \in (\sqrt{2}/\beta, 2/\beta^2) \cap [1, 2]$, using first Jensen's inequality and then subadditivity of $x \mapsto x^{p/2}$ on $[0, \infty)$, we obtain
    \begin{equation*}
        \Expec{\e^{\sqrt{2} X_1}}^{\beta p /\sqrt{2}}
        \leq \Expec{\e^{\beta p X_1}}
        = \Expec{\left(\sum_{j \geq 1} \e^{2 \beta \Delta_{1j}} W_{1j}^2\right)^{p/2}} 
        \leq \Expec{\sum_{j \geq 1} \e^{p \beta \Delta_{1j}}}
        \Expec{W_\infty(\beta)^{p}}
        < \infty,
    \end{equation*}
    where the last quantity is finite by Lemma~\ref{lem:sum_gibbs_weights_ppp} and Lemma~\ref{lem:additive_martingale_convergence_in_Lp}.
    By~\eqref{eq:renormalized_subcritical_overlap_2_equal_in_distribution_1} and~\eqref{eq:renormalized_subcritical_overlap_2_Bolthausen_Sznitman},
    \begin{equation}\label{eq:renormalized_subcritical_overlap_2_equal_in_distribution_2}
        X_{\mathrm{over}} \overset{d}{=} \frac{\left(C Z_\infty \Expec{\e^{\sqrt{2} X_1}}\right)^{2\beta/\sqrt{2}}}{W_\infty(\beta)^2} \sum_{i\geq 1} \e^{2\beta p_i}.
    \end{equation}
    Finally, by Lemma~\ref{lem:sum_gibbs_weights_ppp}, the sum $S_{2\beta} = \sum_i \e^{2\beta p_i}$ is a non-degenerate $(\sqrt{2}/2\beta)$-stable random variable.
    This remark together with Lemma~\ref{lem:renormalized_subcritical_overlap_limit} and~\eqref{eq:renormalized_subcritical_overlap_2_equal_in_distribution_2} conclude the proof of Theorem~\ref{thm:asymptotics_of_the_overlap_distribution}.\ref{it:asymptotics_of_the_overlap_distribution_3}.
\end{proof}

\section{Proof of Theorem~\ref{thm:asymptotics_in_mean}}
\label{sec:mean}

In this section, the key tools will be the changes of measure~\eqref{eq:definition_of_Q_b} and~\eqref{eq:definition_of_Q_bt} at inverse temperature $2\beta$, as well as the associated spinal decomposition.
We denote by $\Pi$ the point process whose points $s_i$ for $i \geq 1$ are the ordered spinal branching times.
Recall that $\Pi$ is a homogeneous Poisson point process on $[0,\infty)$ with intensity 1 under $\P$ and $2$ under $\Q_{2\beta}$.
We also write $s_0 = 0$.
In what follows, it will be convenient to decompose the additive martingale along the spine,
\begin{equation} \label{eq:decompo_W_t_spine}
    W_t(\beta) = \sum_{i \geq 1} \e^{\beta X_\xi(s_i) - \psi(\beta)s_i} W_{t - s_i}^{(i)}(\beta) \mathds{1}_{s_i \leq t} + \e^{\beta X_\xi(t) - \psi(\beta)t},
\end{equation}
where, if $d_{u \wedge v}$ denotes the deathtime of the last common ancestor of $u$ and $v$,
\begin{equation*}
    W_s^{(i)}(\beta) = \sum_{u \in \cN(s_i + s), d_{u \wedge \xi(s_i + s)} = s_i} \e^{\beta (X_u(s_i + s) - X_u(s_i)) - \psi(\beta)s}.
\end{equation*}
Under both probability measures $\P$ and $\Q_{2\beta}$, given $\Pi$, the processes $(W_s^{(i)}(\beta))_{s \geq 0}$ for $i \geq 1$ are the additive martingales of independent BBMs without spine.
Therefore, they converge $\P$-\as and $\Q_{2\beta}$-\as to \iid copies of $W_\infty(\beta)$, that we denote by $W_\infty^{(i)}(\beta)$.
This allows us to define
\begin{equation}\label{eq:definition_W}
    W(\beta) = \sum_{i \geq 1} \e^{\beta X_\xi(s_i) - \psi(\beta)s_i} W_\infty^{(i)}(\beta),
\end{equation}
which is the \as limit of $W_t(\beta)$ under $\Q_{2\beta}$ as showed in the following lemma. Note that we distinguish $W(\beta)$ and $W_\infty(\beta)$ for now, but, once this lemma is proved, we can afterwards say that they are the same quantity, see Remark~\ref{rem:W(beta)_vs_Winfty(beta)}.
We also emphasize here that the convergence of $W_t(\beta)$ is classically studied under $\Q_\beta$ (see \eg~\cite{Kyprianou2004}), but we need it here under $\Q_{2\beta}$.

\begin{lemma}\label{lem:almost_sure_convergence_of_W_t}
    If $\beta \in [0, \sqrt{2\smash{/}3})$, the random variable $W(\beta)$ is $\Q_{2\beta}$-\as positive and finite.
    If $\beta \in [\sqrt{2\smash{/}3}, \sqrt{2})$, it is $\Q_{2\beta}$-\as infinite.
    In both cases, it is the $\Q_{2\beta}$-\as limit of the process $W_t(\beta)$ as $t \to \infty$.
\end{lemma}

\begin{proof}
    Let $\beta \in [0,\sqrt{2})$.
    Under $\Q_{2\beta}$,
    \begin{equation*}
        W(\beta) = \sum_{i \geq 1} \e^{\beta B_{s_i} + (3\beta^2/2 - 1) s_i} W_\infty^{(i)}(\beta),
    \end{equation*}
    where $B_s = X_\xi(s) - 2 \beta s$ is a standard Brownian motion.
    Note that $3\beta^2/2 - 1 \geq 0$ if and only if $\beta \geq \sqrt{2\smash{/}3}$.
    In particular, if $\beta \in [\sqrt{2\smash{/}3}, \sqrt{2})$, the (non-negative) terms of the series defining $W(\beta)$ do not converge to $0$ and then $W(\beta) = \infty$, $\Q_{2\beta}$-\as
    Besides, by Fatou's lemma,
    \begin{equation*}
        \liminf_{t \to \infty} W_t(\beta) \geq \sum_{i \geq 1} \liminf_{t \to \infty} \e^{\beta X_\xi(s_i) - \psi(\beta)s_i} W_{t - s_i}^{(i)}(\beta) \mathds{1}_{s_i \leq t},
    \end{equation*}
    and this lower bound is $\Q_{2\beta}$-\as equal to $W(\beta) = \infty$.

    Now let $\beta \in [0, \sqrt{2\smash{/}3})$.
    We have
    \begin{equation} \label{eq:W(beta)_finite}
        \E_{\Q_{2\beta}} \left[ W(\beta) \middle| \Pi, X_\xi \right] = \sum_{i \geq 1} \e^{\beta B_{s_i} - (1 - 3\beta^2/2) s_i},
    \end{equation}
    and, $\Q_{2\beta}$-\as, this series converges since
    \begin{equation*}
        \sqrt[n]{\e^{\beta B_{s_n} - (1 - 3\beta^2/2) s_n}} \xrightarrow[n \to \infty]{} \e^{-(1 - 3\beta^2/2)/2} < 1.
    \end{equation*}
    Thus, $W(\beta)$ is $\Q_{2\beta}$-\as finite.
    Furthermore, it is $\Q_{2\beta}$-\as positive since the terms of the series~\eqref{eq:definition_W} are $\Q_{2\beta}$-\as positive.
    
    It remains to control $W(\beta) - W_t(\beta)$.
    We rewrite it as
    \begin{align}
    T_1(t) + T_2(t) + T_3(t)
    & = \sum_{i \geq 1, s_i \leq t} \e^{\beta X_\xi(s_i) - \psi(\beta)s_i} (W_\infty^{(i)}(\beta) - W_{t - s_i}^{(i)}(\beta)) \nonumber \\
    & \qquad + \sum_{i \geq 1, s_i > t} \e^{\beta X_\xi(s_i) - \psi(\beta)s_i} W_\infty^{(i)}(\beta) \nonumber \\
    & \qquad - \e^{\beta X_\xi(t) - \psi(\beta)t}. \nonumber 
    \end{align}
    The term $T_3(t)$ converges $\Q_{2\beta}$-\as to $0$ since the exponent $\beta X_\xi(s) - \psi(\beta)s$ is a Brownian motion with drift $-1 + 3\beta^2/2 < 0$ (and variance $\beta^2$).
    The term $T_2(t)$ converges $\Q_{2\beta}$-\as to $0$ since it is the remainder of a series which converges $\Q_{2\beta}$-\as
    To control the term $T_1(t)$, we first show that it converges along the sequence of spinal branching times.
    By Lemma~\ref{lem:von_Bahr}, for any $p \in [1, 2]$ and $n \geq 1$,
    \begin{equation*} 
        \E_{\Q_{2\beta}} \left[ |T_1(s_n)|^p \middle | \Pi, X_\xi \right] \leq \sum_{i = 1}^n \e^{p \beta X_\xi(s_i) - p \psi(\beta) s_i} f(s_n - s_i),
    \end{equation*}
    where, for $s \geq 0$, $f(s) = \E|W_\infty(\beta) - W_s(\beta)|^p$.
    By Lemma~\ref{lem:additive_martingale_convergence_in_Lp}, we have $f(s) \leq C \e^{-(p - 1)(1 - p\beta^2/2)s}$. Hence, by~\eqref{eq:pc-cp}, the above quantity is bounded by
    \begin{equation*} 
        C \e^{-(p - 1)(1 - p\beta^2/2)s_n} \sum_{i \geq 1} \e^{p \beta X_\xi(s_i) - \psi(p \beta) s_i},
    \end{equation*}
    which is $\Q_{2\beta}$-\as summable in $n$ as soon as $p \in (1, 2/\beta^2)$ to ensure that the exponential prefactor vanishes, and $\beta^2 (2p-\frac{p^2}{2}) < 1$ to ensure that the sum is $\Q_{2\beta}$-\as finite by proceeding as in~\eqref{eq:W(beta)_finite} (such a choice of $p$ exists because $\beta < \sqrt{2\smash{/}3}$).
    By the Borel--Cantelli lemma, $T_1(s_n)$ converges $\Q_{2\beta}$-\as to $0$ as $n \to \infty$.
    Now let us control
    \begin{equation}\label{eq:T1_extension_to_continuum}
        \sup_{t \in (s_n, s_{n+1})} |T_1(t) - T_1(s_n)| \leq \sum_{i = 1}^n \e^{\beta X_\xi(s_i) - \psi(\beta)s_i} \underbrace{\sup_{t \in (s_n, s_{n+1})} \left| W_{t - s_i}^{(i)}(\beta) - W_{s_n - s_i}^{(i)}(\beta) \right|}_{= S_{n, i}}.
    \end{equation}
    Similarly to~\eqref{eq:extension_to_continuum}, we proceed by showing that
    \begin{equation}\label{eq:T1_extension_to_continuum_1}
        \sum_{i = 1}^n \e^{\beta X_\xi(s_i) - \psi(\beta)s_i} \E_{\Q_{2\beta}} \left[ S_{n, i} \middle| \Pi \right] \xrightarrow[n \to \infty]{} 0, \quad \Q_{2\beta}\text{-\as},
    \end{equation}
    and
    \begin{equation}\label{eq:T1_extension_to_continuum_2}
        \sum_{i = 1}^n \e^{\beta X_\xi(s_i) - \psi(\beta)s_i} \left( S_{n, i} - \E_{\Q_{2\beta}} \left[ S_{n, i} \middle| \Pi \right] \right) \xrightarrow[n \to \infty]{} 0, \quad \Q_{2\beta}\text{-\as}
    \end{equation}
    By applying the Cauchy-Schwarz inequality and Doob's maximal inequality, we obtain, for any $p \in (1,2]$,
    \begin{equation} \label{eq:bound_S_ni}
        \E_{\Q_{2\beta}} \left[ S_{n, i} \middle| \Pi \right] 
        \leq \E_{\Q_{2\beta}} \left[ (S_{n, i})^p \middle| \Pi \right]^{1/p}
        \leq \frac{p}{p-1} f(s_n - s_i, s_{n+1} - s_i)^{1/p},
    \end{equation}
    where, for $s \leq t$,
    \begin{equation*}
        f(s, t) = \E|W_t(\beta) - W_s(\beta)|^p \leq C \e^{-(p - 1)(1 - p\beta^2/2)s},
    \end{equation*}
    by Lemma~\ref{lem:additive_martingale_convergence_in_Lp}.
    In particular, since $p \in (1, 2/\beta^2)$, each term of the sum in~\eqref{eq:T1_extension_to_continuum_1} converges $\Q_{2\beta}$-\as to $0$ as $n \to \infty$.
    Then, by dominated convergence with domination
    \begin{equation*}
        \e^{\beta X_\xi(s_i) - \psi(\beta)s_i} \E_{\Q_{2\beta}} \left[ S_{n, i} \middle| \Pi \right] \leq C \e^{\beta X_\xi(s_i) - \psi(\beta)s_i}, 
    \end{equation*}
    we obtain~\eqref{eq:T1_extension_to_continuum_1}.
    As for the convergence~\eqref{eq:T1_extension_to_continuum_2}, it can be deduced from Lemma~\ref{lem:von_Bahr} and~\eqref{eq:bound_S_ni}, by proceeding as in the proof that $T_1(s_n)$ converges $\Q_{2\beta}$-\as to $0$ above.
    Finally, the bound in~\eqref{eq:T1_extension_to_continuum} converges $\Q_{2\beta}$-\as to $0$ as $n \to \infty$ and the term $T_1(t)$ converges $\Q_{2\beta}$-\as to $0$ as $t \to \infty$.
\end{proof}

\begin{remark} \label{rem:W(beta)_vs_Winfty(beta)}
    In the case $0 \leq \beta < \sqrt{2}/2$, a consequence of Lemma~\ref{lem:almost_sure_convergence_of_W_t} is that $W(\beta) = W_\infty(\beta)$ $\P$-\as and $\Q_{2\beta}$-\as since the measures $\P$ and $\Q_{2\beta}$ are equivalent.
    On the other hand, if $\beta \in [\sqrt{2}/2,\sqrt{2})$, then $\P$ and $\Q_{2\beta}$ are mutually singular so we can assume that $W_\infty(\beta) = W(\beta)$ $\Q_{2\beta}$-\as. 
    Moreover, one can check that $W(\beta) = W_\infty(\beta)$ $\P$-\as stays true for $\beta \in [\sqrt{2}/2,\sqrt{2})$ by following an argument similar to the proof of Lemma~\ref{lem:almost_sure_convergence_of_W_t} in the case $\beta \in (0,\sqrt{2\smash{/}3})$.
    This ensures that there is no notation conflict if we write $W_\infty(\beta)$ instead of $W(\beta)$ under $\P$ or $\Q_{2\beta}$, which we will do from now on.
\end{remark}

\subsection{Upper temperature}

\begin{proof}[Proof of Theorem~\ref{thm:asymptotics_in_mean}.\ref{it:asymptotics_in_mean_2}]
    Let us fix $\beta \in (0, \sqrt{2\smash{/}3})$. 
    Note that the claim concerning the convergence of $W_t(\beta)$ has already been proved in Lemma~\ref{lem:almost_sure_convergence_of_W_t} (see also Remark~\ref{rem:W(beta)_vs_Winfty(beta)}), so we focus on the convergence of the rescaled overlap distribution.
    We define
    \begin{equation}\label{eq:definition_F_t}
        F_t = \sum_{u \in \cN(at)} \e^{2\beta X_u(at) - \psi(2\beta)at} W_{t - at}^{(u, at)}(\beta)^2,
    \end{equation}
    so that $\e^{(1 - \beta^2)at} \nu_{\beta, t}([a, 1]) = F_t W_t(\beta)^{-2}$, by~\eqref{eq:overlap_rewriting}.
    We have
    \begin{align}
        \e^{(1 - \beta^2)at} \Expec{\nu_{\beta, t}([a, 1])} &= \E \left[ F_t W_t(\beta)^{-2} - F_t W_{at}(\beta)^{-2} \right] + \E \left[ F_t W_{at}(\beta)^{-2} \right] \nonumber \\
        &= \E \left[ F_t W_t(\beta)^{-2} - F_t W_{at}(\beta)^{-2} \right] + \E_{\Q_{2\beta}} \left[ W_{at}(\beta)^{-2} \right] \E \left[ W_{t-at}(\beta)^2 \right], \nonumber
    \end{align}
    using the branching property and the definition of $\Q_{2\beta}$ in the second term.
    Then, by Lemmas~\ref{lem:replacement_of_the_denominator} and~\ref{lem:convergence_in_mean_of_1/W_t^2} below, together with Lemma~\ref{lem:additive_martingale_convergence_in_Lp},
    \begin{equation*}
        \e^{(1 - \beta^2)at} \Expec{\nu_{\beta, t}([a, 1])} \xrightarrow[t \to \infty]{} \E_{\Q_{2\beta}} \left[ W_\infty^{-2} \right] \Expec{W_\infty(\beta)^2},
    \end{equation*}
    which is Theorem~\ref{thm:asymptotics_in_mean}.\ref{it:asymptotics_in_mean_2}.
\end{proof}

\begin{lemma}\label{lem:convergence_in_mean_of_1/W_t^2}
    For every $\beta \in (0, \sqrt{2\smash{/}3})$,
    \begin{equation*}
        \E_{\Q_{2\beta}}[W_t(\beta)^{-2}] \xrightarrow[t \to \infty]{} \E_{\Q_{2\beta}}[W_\infty(\beta)^{-2}] \in (0, \infty).
    \end{equation*}
\end{lemma}

\begin{remark} \label{rem:mean_of_ratios}
    The convergence in this lemma stays true for any $\beta \in [0, \sqrt{2})$, but for $\beta \notin (0, \sqrt{2\smash{/}3})$ the limit is not in $(0,\infty)$ anymore:
    \begin{itemize}
        \item If $\beta = 0$, the fixed point equation~\eqref{eq:rewriting_additive_martingale} with $t=\infty$ and $s=1$, together with \cite[Theorem~2.4]{Liu2001} implies that $\E_{\Q_0}[W_\infty(0)^{-2}] = \E [W_\infty(0)^{-1}] = \infty$.
        Then, by Fatou's lemma, $\E_{\Q_0}[W_t(0)^{-2}] \to \infty$ as $t \to \infty$.
        \item If $\beta \in [\sqrt{2\smash{/}3}, \sqrt{2})$, Lemma~\ref{lem:almost_sure_convergence_of_W_t} states that $W_t(\beta) \to \infty$, $\Q_{2\beta}$-\as, and Lemma~\ref{lem:bounded_in_Lp} below implies that $\E_{\Q_{2\beta}}[W_t(\beta)^{-2}] \to 0$ as $t \to \infty$.
    \end{itemize}
\end{remark}

\begin{lemma}\label{lem:replacement_of_the_denominator}
    Let $\beta \in (0, \sqrt{2\smash{/}3})$ and define $F_t$ as in~\eqref{eq:definition_F_t}.
    Then,
    \begin{equation*}
        \E \left[ F_t W_t(\beta)^{-2} - F_t W_{at}(\beta)^{-2}  \right] \xrightarrow[t \to \infty]{} 0.
    \end{equation*}
\end{lemma}

The remainder of this section is dedicated to the proofs of Lemma~\ref{lem:convergence_in_mean_of_1/W_t^2} and Lemma~\ref{lem:replacement_of_the_denominator}

\begin{lemma}\label{lem:negative_moments_linear_combination}
    Let $X_i$ for $i \geq 1$ be independent random variables and $a_i$, $i \geq 1$, be positive coefficients.
    Assume that there exist $b, K > 0$ such that, for each $i \geq 1$ and every $x > 0$,
    \begin{equation*}
        \P(X_i \leq x) \leq K x^b.
    \end{equation*}
    Then, for every $p \in [1, \infty)$ such that $p/b \notin \Z$, there exists $C = C(b,K,p) >0$ such that, letting $n = \lceil p/b \rceil$,
    \begin{equation*}
        \Expec{\left( \sum_{i = 1}^n a_i X_i \right)^{-p}} \leq C a_1^{-b} \cdots a_{n - 1}^{-b} a_n^{-p + b(n - 1)}.
    \end{equation*}
\end{lemma}

\begin{proof}
    Note that $b(n - 1) < p < bn$ because $p/b \notin \Z$.
    We can assume that $a_1 \geq a_2 \geq \cdots \geq a_n$ since the general case follows from
    \begin{equation*}
        a_i^{-b} a_n^{-p + b(n - 1)} \leq a_n^{-b} a_i^{-p + b(n - 1)}.
    \end{equation*}
    For each $j = 1, \ldots, n$,
    \begin{equation*}
        \P \left( \sum_{i = 1}^n a_i X_i \leq x \right) \leq \P(a_1 X_1 \leq x) \cdots \P(a_j X_j \leq x) \leq K^j a_1^{-b} \cdots a_j^{-b} x^{bj}.
    \end{equation*}
    In particular, we can use the following bound
    \begin{equation*}
        \P \left( \sum_{i = 1}^n a_i X_i \leq x \right) \leq \begin{cases}
            1 & \text{if } x \geq a_1, \\
            K^{n - 1} a_1^{-b} \cdots a_{n-1}^{-b} x^{b(n-1)} & \text{if } a_n \leq x < a_1, \\
            K^n a_1^{-b} \cdots a_n^{-b} x^{bn} & \text{if } x < a_n.
        \end{cases}
    \end{equation*}
    This yields
    \begin{align}
        \Expec{\left( \sum_{i = 1}^n a_i X_i \right)^{-p}} &= \int_0^\infty \P \left( \sum_{i = 1}^n a_i X_i \leq x^{-1/p} \right) \d{x} \nonumber \\
        &\leq \int_0^{a_1^{-p}} \d{x} + \int_{a_1^{-p}}^{a_n^{-p}} K^{n - 1} x^{-b(n-1)/p} a_1^{-b} \cdots a_{n-1}^{-b} \d{x} + \int_{a_n^{-p}}^\infty K^n x^{-bn/p} a_1^{-b} \cdots a_n^{-b} \d{x} \nonumber \\
        &\leq a_1^{-p} + C a_1^{-b} \cdots a_{n - 1}^{-b} a_n^{-p + b(n-1)}, \nonumber
    \end{align}
    using $b(n-1)/p<1$ in the second integral and $bn/p >1$ in the third one.
    To conclude, it suffices to note that $a_1^{-p} \leq a_1^{-b} \ldots a_{n-1}^{-b} a_n^{-p + b(n-1)}$.
\end{proof}

\begin{lemma}\label{lem:bound_for_negative_moments}
    Let $\beta \in [0, \sqrt{2})$ and $p \in [1, \infty)$.
    There exist $n_0\geq 1$ and $C > 0$ such that, for any $n \geq n_0$, any $I = \{i_1, \ldots, i_n\} \subset \{1, 2, \ldots\}$ with $i_1 < \dots < i_n$ and any $t \in [0, \infty]$, on the event where $s_{i_n} \leq t$, we have
    \begin{equation} \label{eq:bound_for_negative_moments}
        \E_{\Q_{2\beta}} \left[ \left( \sum_{i \in I} \e^{\beta X_\xi(s_i) - \psi(\beta)s_i} W_{t - s_i}^{(i)}(\beta) \right)^{-p} \middle| \Pi \right] \leq C \e^{(p^2\beta^2/2 + p(1 - 3\beta^2/2))s_{i_n}}.
    \end{equation}
\end{lemma}

\begin{proof}
    Let $b_{\max} = (\sqrt{\psi(\beta)^2 + 2\beta^2} - \psi(\beta))/\beta^2$ (recall it appears in the condition~\eqref{eq:b_values} of Lemma~\ref{lem:left_tail_additive_martingale}).
    We choose $n_0 > p/b_{\max}$.
    Now note that it is enough to prove~\eqref{eq:bound_for_negative_moments} for $n=n_0$: indeed, adding elements to $I$ can only decrease the left-hand side of~\eqref{eq:bound_for_negative_moments} and increase the right-hand side (note that $p^2\beta^2/2 + p(1 - 3\beta^2/2) \geq 0$).    
    So let $n = n_0$.
    Since $n > p/b_{\max}$, there exists $b \in (0,b_{\max})$ such that $n = \lceil p/b \rceil$ and $p/b \notin \Z$.
    Then, by Lemma~\ref{lem:negative_moments_linear_combination} and Lemma~\ref{lem:left_tail_additive_martingale},
    \begin{align}
        &\E_{\Q_{2\beta}} \left[ \left( \sum_{i \in I} \e^{\beta X_\xi(s_i) - \psi(\beta)s_i} W_{t - s_i}^{(i)}(\beta) \right)^{-p} \middle| \Pi, X_\xi \right] \nonumber \\
        &\quad \leq C \left( \prod_{i \in I \setminus \{i_n\}} \e^{\beta X_\xi(s_i) - \psi(\beta)s_i} \right)^{-b} \left( \e^{\beta X_\xi(s_{i_n}) - \psi(\beta)s_{i_n}} \right)^{-p + b(n-1)} \nonumber \\
        &\quad = C \prod_{j = 1}^{i_n} \left( \e^{\beta (X_\xi(s_j) - X_\xi(s_{j-1})) - \psi(\beta)(s_j - s_{j-1})} \right)^{-\alpha_j}, \nonumber
    \end{align}
    with $\alpha_j = p - b \#(I \cap \{1, \ldots, j-1\})$.
    Recall that, under $\Q_{2\beta}$, the spine trajectory $X_\xi$ is a Brownian motion with drift $2\beta$, independent of $\Pi$.
    Then,
    \begin{align}
        \E_{\Q_{2\beta}} \left[ \left( \sum_{i \in I} \e^{\beta X_\xi(s_i) - \psi(\beta)s_i} W_{t - s_i}^{(i)}(\beta) \right)^{-p} \middle| \Pi \right] &\leq C \prod_{j = 1}^{i_n} \e^{(\alpha_j^2 \beta^2/2 + \alpha_j (1 - 3\beta^2/2))(s_j - s_{j - 1})} \nonumber \\
        &\leq C \prod_{j = 1}^{i_n} \e^{(p^2 \beta^2/2 + p (1 - 3\beta^2/2))(s_j - s_{j - 1})}, \nonumber
    \end{align}
    using in the last inequality that $\alpha_j \in [0,p]$ and that the function $x \in [0,p] \mapsto x^2 \beta^2/2 + x (1 - 3\beta^2/2)$ reaches its maximum at $x=p$, because $\beta <\sqrt{2}$ and $p\geq 1$.
    This gives the desired inequality by telescoping.
\end{proof}

\begin{lemma}\label{lem:bounded_in_Lp}
    For every $\beta \in (0, \sqrt{2})$, the process $(W_t(\beta)^{-2})_{t \geq 0}$ is bounded in $L^p(\Q_{2\beta})$ for some $p > 1$.
\end{lemma}

\begin{proof}
    Let $\beta \in (0, \sqrt{2})$ and $p > 1$, which will be chosen close enough to 1 depending on $\beta$ later on.
    Let $n \geq 1$ be given by Lemma~\ref{lem:bound_for_negative_moments} applied to $\beta$ and $p$.
    Let $t \geq 0$ and set $A = A_{n, t} = \{s_n \leq t\}$.
    On the one hand, using~\eqref{eq:decompo_W_t_spine},
    \begin{align}
        \E_{\Q_{2\beta}} \left[ W_t(\beta)^{-2p} \mathds{1}_A \right] &\leq \E_{\Q_{2\beta}} \left[ \left( \sum_{i = 1}^n \e^{\beta X_\xi(s_i) - \psi(\beta)s_i} W_{t - s_i}^{(i)}(\beta) \right)^{-2p} \mathds{1}_A \right] \nonumber \\
        &\leq C \E_{\Q_{2\beta}} \left[ \e^{(2p^2\beta^2 + 2p(1 - 3\beta^2/2)) s_n} \right], \nonumber
    \end{align}
    by Lemma~\ref{lem:bound_for_negative_moments}.
    If we choose $p$ close enough to 1, we have $2p^2\beta^2 + 2p(1 - 3\beta^2/2) < 2$ and then the above bound is finite (recall $s_n$ is $\Gamma(n,2)$ distributed under $\Q_{2\beta}$).
    On the other hand, keeping only the contribution of the spine to $W_t(\beta)$, we have
    \begin{equation*}
        \E_{\Q_{2\beta}} \left[ W_t(\beta)^{-2p} \mathds{1}_{A^c} \right] \leq \E_{\Q_{2\beta}} \left[ \e^{-2p \beta X_\xi(t) + 2p \psi(\beta)t} \right] \Q_{2\beta}(s_n > t) 
        \leq \e^{(2 - 3\beta^2 + 2p\beta^2)pt} C t^{n-1} \e^{-2t}.
    \end{equation*}
    If we choose $p$ close enough to 1, we have $2 - 3\beta^2 + 2p\beta^2 < 0$ and the above quantity converges to $0$ as $t \to \infty$.
    Combining the two previous bounds shows the result.
\end{proof}

\begin{proof}[Proof of Lemma~\ref{lem:convergence_in_mean_of_1/W_t^2}]
    By Lemma~\ref{lem:almost_sure_convergence_of_W_t}, for every $\beta \in (0, \sqrt{2\smash{/}3})$, the random variable $W_\infty(\beta)$ is $\Q_{2\beta}$-\as positive and finite.
    In particular, $\E_{\Q_{2\beta}}[W_\infty(\beta)^{-2}] > 0$.
    Besides, $W_t(\beta)^{-2}$ converges $\Q_{2\beta}$-\as to 
    $W_\infty(\beta)^{-2}$ as $t \to \infty$ and is uniformly integrable for $\Q_{2\beta}$, by Lemma~\ref{lem:bounded_in_Lp}.
    This implies that
    \begin{equation*}
        \E_{\Q_{2\beta}}[W_t(\beta)^{-2}] \xrightarrow[t \to \infty]{} \E_{\Q_{2\beta}}[W_\infty(\beta)^{-2}] < \infty,
    \end{equation*}
    which concludes the proof.
\end{proof}

It remains to prove Lemma~\ref{lem:replacement_of_the_denominator}.
We will use the following concentration inequality.

\begin{lemma}\label{lem:concentration_for_ratio}
    Let $\beta \in [0, 1)$ and $p \in [1, 2]$. There exists $C>0$ such that, for any $0 \leq s \leq t \leq \infty$,
    \begin{equation*}
        \Expec{ \left| \frac{W_t(\beta) - W_s(\beta)}{W_s(\beta)} \right|^p} \leq C \e^{-(1 - \beta^2)ps/2}.
    \end{equation*}
\end{lemma}

\begin{proof}
    We have, using decomposition~\eqref{eq:rewriting_additive_martingale},
    \begin{align}
        (W_t(\beta) - W_s(\beta))^2 &= \sum_{u \in \cN(s)} \e^{2\beta X_u(s) - 2\psi(\beta)s} (W_{t - s}^{(u, s)}(\beta) - 1)^2 \nonumber \\
        &\quad + \sum_{u, v \in \cN(s), u \neq v} \e^{\beta(X_u(s) + X_v(s)) - 2\psi(\beta)s} (W_{t - s}^{(u, s)}(\beta) - 1)(W_{t - s}^{(v, s)}(\beta) - 1), \nonumber
    \end{align}
    where, conditionally on $\cF_s$, the variables $W_{t - s}^{(u,s)}(\beta)$ for $u \in \cN(s)$ are independent with mean $1$ and are bounded in $L^2(\P)$ because $\beta <1$ (see Lemma~\ref{lem:additive_martingale_convergence_in_Lp}).
    Then, by applying successively H\"{o}lder's inequality and taking a conditional expectation given $\cF_s$, we obtain
    \begin{equation*}
        \Expec{ \left| \frac{W_t(\beta) - W_s(\beta)}{W_s(\beta)} \right|^p} \leq \Expec{\frac{(W_t(\beta) - W_s(\beta))^2}{W_s(\beta)^2}}^{p/2} \leq C \Expec{\frac{\e^{-(1 - \beta^2)s} W_s(2\beta)}{W_s(\beta)^2}}^{p/2},
    \end{equation*}
    since $\psi(2\beta) - 2\psi(\beta) = -(1 - \beta^2)$ and since the additive martingale is bounded in $L^2$.
    The result follows from Lemma~\ref{lem:convergence_in_mean_of_1/W_t^2}, which implies that $\E[W_s(2\beta) W_s(\beta)^{-2}] = \E_{\Q_{2\beta}}[W_s(\beta)^{-2}]$ converges to a finite limit and hence is bounded.
\end{proof}

\begin{proof}[Proof of Lemma~\ref{lem:replacement_of_the_denominator}]
    Let $\beta \in (0, \sqrt{2\smash{/}3})$.
    In order to take advantage of Lemma~\ref{lem:bound_for_negative_moments}, we will work on the event $A = A_{n, a't} = \{|\xi(a't)| \geq n\}$, where $a' \in (0, a)$ is ``close enough'' to $a$ and $n \geq 1$ is ``large enough''.
    However, in view of applying Lemma~\ref{lem:gibbs_measure_weights}, it will be convenient to deal with a slightly smaller event that is measurable w.r.t.\@ $\cF_\infty$.
    Let us define $\bar{A} = \bar{A}_{n, a't} = \{\forall u \in \cN(a't), |u| \geq n\}$.
    
    First, let us show that
    \begin{equation}\label{eq:replacement_of_the_denominator_A^c}
        \E \left[ (F_t W_{at}(\beta)^{-2} - F_t W_t(\beta)^{-2}) \mathds{1}_{\bar{A}^c} \right] \xrightarrow[t \to \infty]{} 0.
    \end{equation}
    By conditioning on $\cF_{at}$, we obtain
    \begin{equation}\label{eq:replacement_of_the_denominator_A^c_1}
        \E[F_t W_{at}(\beta)^{-2} \mathds{1}_{\bar{A}^c}] = \E[W_{at}(2\beta) W_{at}(\beta)^{-2} \mathds{1}_{\bar{A}^c}] = \E_{\Q_{2\beta}}[W_{at}(\beta)^{-2} \mathds{1}_{\bar{A}^c}].
    \end{equation}
    Note that $\mathds{1}_{\bar{A}^c}$ converges \as to $0$ as $t \to\infty$.
    In addition, a direct consequence of Lemma~\ref{lem:almost_sure_convergence_of_W_t}, Lemma~\ref{lem:convergence_in_mean_of_1/W_t^2} and Scheffé's lemma is that $W_{at}(\beta)^{-2}$ converges in $L^1(\Q_{2\beta})$.
    In particular, this process is uniformly integrable, and so is $W_{at}(\beta)^{-2} \mathds{1}_{\bar{A}^c}$. Hence, the right-hand side of~\eqref{eq:replacement_of_the_denominator_A^c_1} converges to $0$ as $t \to \infty$.
    Besides,
    \begin{equation*}
        \E[F_t W_t(\beta)^{-2} \mathds{1}_{\bar{A}^c}] = \e^{(1 - \beta^2)at} \E \left[ \nu_{\beta, t}([a, 1]) \mathds{1}_{\bar{A}^c} \right] \leq \e^{(1 - \beta^2)at} \P(\bar{A}^c).
    \end{equation*}
    Applying the union bound, the change of measure~\eqref{eq:definition_of_Q_b_without_spine} and Lemma~\ref{lem:gibbs_measure_weights}, we obtain
    \begin{equation}\label{eq:probability_of_Ac}
        \P(\bar{A}^c) \leq \Expec{\sum_{u\in\cN(a't)} \mathds{1}_{\{|u| \leq n\}}} = \E_{\Q_0} \left[ \sum_{u \in \cN(a't)} \frac{ \mathds{1}_{\{|u| \leq n\}}}{W_{a't}(0)} \right] = e^{a't} \Q_0(|\xi(a't)| \leq n) \leq C t^n e^{-a' t},
    \end{equation}
    using in the last inequality that, under $\Q_0$, the spinal generation $|\xi(s)|$ has Poisson distribution with parameter $2s$.
    Then,
    \begin{equation}\label{eq:replacement_of_the_denominator_A^c_bound}
        \E[F_t W_t(\beta)^{-2} \mathds{1}_{\bar{A}^c}] \leq C t^n \e^{(1 - \beta^2)at - a't}.
    \end{equation}
    By choice of $a'$, this bound converges to $0$ as $t \to \infty$.
    Hence~\eqref{eq:replacement_of_the_denominator_A^c}.
    
    Now, let us show that
    \begin{equation}\label{eq:replacement_of_the_denominator_A}
        \E[(F_t W_{at}(\beta)^{-2} - F_t W_t(\beta)^{-2}) \mathds{1}_{\bar{A}}] \xrightarrow[t \to \infty]{} 0.
    \end{equation}
    Recalling that $\Q_{2\beta, at}|_{\cF_\infty}$ has Radon--Nikodym derivative $W_{at}(2\beta)$ w.r.t.\@ $\P|_{\cF_\infty}$, the left-hand side of~\eqref{eq:replacement_of_the_denominator_A} equals
    \begin{align}
        &\E_{\Q_{2\beta, at}} \left[ \frac{F_t}{W_{at}(2\beta)} (W_t(\beta) - W_{at}(\beta)) \left( \frac{1}{W_t W_{at}(\beta)^2} + \frac{1}{W_t(\beta)^2 W_{at}(\beta)} \right) \mathds{1}_{\bar{A}} \right] \nonumber \\
        &\quad = \E_{\Q_{2\beta, at}} \left[ W_{t - at}^{(\xi(at), at)}(\beta)^2 (W_t(\beta) - W_{at}(\beta)) \left( \frac{1}{W_t(\beta) W_{at}(\beta)^2} + \frac{1}{W_t(\beta)^2 W_{at}(\beta)} \right) \mathds{1}_{\bar{A}} \right], \nonumber
    \end{align}
    where we applied Lemma~\ref{lem:gibbs_measure_weights} to rewrite
    \begin{equation*}
        \frac{F_t}{W_{at}(2\beta)} = \sum_{u \in \cN(at)} \Q_{2\beta, at} \left( \xi(at) = u \middle| \cF_t \right) W_{t - at}^{(u, at)}(\beta)^2.
    \end{equation*}
    Let us abbreviate $W_s^{(\xi(at), at)}(\beta) = W_s^{(\xi, at)}(\beta)$.
    Decomposing the process $|W_t(\beta) - W_{at}(\beta)|$ along the spine as in~\eqref{eq:decompo_W_t_spine} and applying the triangle inequality, we bound it by
    \begin{align}
        T_1(t) + T_2(t) + T_3(t)
        & = \sum_{i = 1}^n \e^{\beta X_\xi(s_i) - \psi(\beta)s_i} |W_{t - s_i}^{(i)}(\beta) - W_{at - s_i}^{(i)}(\beta)| \nonumber \\
        &\quad + \left| \sum_{i > n, s_i \leq at} \e^{\beta X_\xi(s_i) - \psi(\beta)s_i} (W_{t - s_i}^{(i)}(\beta) - W_{at - s_i}^{(i)}(\beta)) \right| \nonumber \\
        &\quad + \e^{\beta X_\xi(at) - \psi(\beta)at} |W_{t - at}^{(\xi, at)}(\beta) - 1|. \nonumber
    \end{align}
    Besides,
    \begin{equation*}
        \frac{1}{W_t(\beta) W_{at}(\beta)^2} + \frac{1}{W_t(\beta)^2 W_{at}(\beta)} \leq \frac{1}{W_{at}(\beta)^3} + \frac{2}{W_t(\beta)^2 W_{at}(\beta)}.
    \end{equation*}
    Therefore, and since $\bar{A} \subset A$, to obtain~\eqref{eq:replacement_of_the_denominator_A}, it suffices to show that, for $i = 1,2,3$ and $\gamma = 1, a$,
    \begin{equation}\label{eq:replacement_of_the_denominator_A_reformulation}
        \E_{\Q_{2\beta, at}} \left[ W_{t - at}^{(\xi, at)}(\beta)^2 \frac{T_i(t)}{W_{\gamma t}(\beta)^2 W_{at}(\beta)} \mathds{1}_A \right] \xrightarrow[t \to \infty]{} 0.
    \end{equation}

    \textbf{Control of $T_1$.}
    Let us show~\eqref{eq:replacement_of_the_denominator_A_reformulation} for $i = 1$.
    Let $\gamma = 1 \text{ or } a$.
    We have
    \begin{align}
    &\E_{\Q_{2\beta, at}} \left[ W_{t - at}^{(\xi, at)}(\beta)^2 \frac{T_1(t)}{W_{\gamma t}(\beta)^2 W_{at}(\beta)} \mathds{1}_A \right] \nonumber \\
    & \quad \leq \E \left[ W_{t - at}(\beta)^2 \right]
    \sum_{i = 1}^n \E_{\Q_{2\beta}} \Biggl[ \frac{\e^{\beta X_\xi(s_i) - \psi(\beta)s_i} |W_{t - s_i}^{(i)}(\beta) - W_{at - s_i}^{(i)}(\beta)|}{W_{at}(\beta)} \nonumber \\
    & \hspace{5.5cm} \times \Biggl( \sum_{j = 1, j \neq i}^n \e^{\beta X_\xi(s_j) - \psi(\beta)s_j} W_{\gamma t - s_j}^{(j)}(\beta) \Biggr)^{-2} \mathds{1}_A \Biggr]. \nonumber
    \end{align}
    Let us denote by $V_1^{(i)}(t)$ the process appearing in the last expectation.
    Since the additive martingale is bounded in $L^2(\P)$, it suffices to get the convergence in mean of $V_1^{(i)}(t)$ to $0$ as $t \to \infty$, for $i = 1, \ldots, n$.
    Using that $W_{at}(\beta) \geq \e^{\beta X_\xi(s_i) - \psi(\beta)s_i} W_{at - s_i}^{(i)}(\beta)$ and then applying Lemma~\ref{lem:concentration_for_ratio}, we obtain
    \begin{align}
        \E_{\Q_{2 \beta}} \left[ V_1^{(i)}(t) \right] &\leq \E_{\Q_{2 \beta}} \left[ \frac{|W_{t - s_i}^{(i)}(\beta) - W_{at - s_i}^{(i)}(\beta)|}{W_{at - s_i}^{(i)}(\beta)} \left( \sum_{j = 1, j \neq i}^n \e^{\beta X_\xi(s_j) - \psi(\beta)s_j} W_{\gamma t - s_j}^{(j)}(\beta) \right)^{-2} \mathds{1}_A \right] \nonumber \\
        &\leq C \E_{\Q_{2 \beta}} \left[ \e^{-(1-\beta^2)(at-s_i)/2} \left( \sum_{j = 1, j \neq i}^n \e^{\beta X_\xi(s_j) - \psi(\beta)s_j} W_{\gamma t - s_j}^{(j)}(\beta) \right)^{-2} \mathds{1}_A \right]. \nonumber
    \end{align}
    We can assume $n \geq n_0+1$ where $n_0$ is given by Lemma~\ref{lem:bound_for_negative_moments} applied to $\beta$ and $p = 2$.
    This yields
    \begin{equation*}
        \E_{\Q_{2 \beta}} \left[ V_1^{(i)}(t) \right] \leq C \E_{\Q_{2 \beta}} \left[ \e^{-(1-\beta^2)(at-s_i)/2} \e^{(2 - \beta^2) s_n} \mathds{1}_A \right].
    \end{equation*}
    By definition of the event $A$, we have $s_i \leq a't < at$.
    The above expectation is then bounded by
    \begin{equation*}
        C \e^{-(1-\beta^2)(a-a')t/2} \E_{\Q_{2 \beta}} \left[ \e^{(2 - \beta^2) s_n} \right] \xrightarrow[t \to \infty]{} 0,
    \end{equation*}
    since $s_n$ is $\Gamma(n, 2)$ distributed under $\Q_{2\beta}$ and $2-\beta^2 < 2$.
    Hence~\eqref{eq:replacement_of_the_denominator_A_reformulation} for $i = 1$.

    \textbf{Control of $T_2$.}
    Let us show~\eqref{eq:replacement_of_the_denominator_A_reformulation} for $i = 2$.
    Let $\gamma = 1 \text{ or } a$.
    We have
    \begin{align}
    &\E_{\Q_{2\beta, at}} \left[ W_{t - at}^{(\xi, at)}(\beta)^2 \frac{T_2(t)}{W_{\gamma t}(\beta)^2 W_{at}(\beta)} \mathds{1}_A \right] \nonumber \\
    &\quad \leq \E \left[ W_{t - at}(\beta)^2 \right]\E_{\Q_{2\beta}} \Biggl[ \Biggl| \sum_{i > n, s_i \leq at} \frac{\e^{\beta X_\xi(s_i) - \psi(\beta)s_i} (W_{t - s_i}^{(i)}(\beta) - W_{at - s_i}^{(i)}(\beta))}{W_{at}(\beta)} \Biggr| \nonumber \\
    &\hspace{6.5cm} \times \Biggl( \sum_{j = 1}^n \e^{\beta X_\xi(s_j) - \psi(\beta)s_j} W_{\gamma t - s_j}^{(j)}(\beta) \Biggr)^{-2} \mathds{1}_A \Biggr]. \nonumber
    \end{align}
    Let us denote by $V_2(t)$ the process appearing in the last expectation and show that it converges in mean to $0$ as $t \to \infty$.
    To do so, we show that it converges $\Q_{2\beta}$-\as to $0$ and that it is bounded in $L^p(\Q_{2\beta})$ for some $p > 1$.
    
    By Lemma~\ref{lem:almost_sure_convergence_of_W_t}, the processes $W_{at}(\beta)$ and $W_{\gamma t - s_j}^{(j)}(\beta)$ for $j \geq 1$ converge $\Q_{2\beta}$-\as to some positive limits as $t \to \infty$.
    Then, the $\Q_{2\beta}$-\as convergence of $V_2(t)$ to $0$ will follow from
    \begin{equation}\label{eq:almost_sure_convergence_of_V_t}
        \sum_{i \geq 1} \e^{\beta X_\xi(s_i) - \psi(\beta)s_i} (W_{t - s_i}^{(i)}(\beta) - W_{at - s_i}^{(i)}(\beta)) \mathds{1}_{s_i \leq at} \xrightarrow[t \to \infty]{} 0, \quad \text{$\Q_{2\beta}$-\as}
    \end{equation}
    Each term of the above series converges $\Q_{2\beta}$-\as to $0$ and is dominated by $2 \e^{\beta X_\xi(s_i) - \psi(\beta)s_i} \sup_{r \geq 0} W_r^{(i)}(\beta)$.
    By Doob's maximal inequality, this domination is summable.
    Hence~\eqref{eq:almost_sure_convergence_of_V_t}.
    
    Now fix $p \in (1, 2)$ such that $\alpha_p = 2p^2\beta^2 + 2p(1 - 3\beta^2/2) < 2$.
    We can assume $n \geq n_0$ where $n_0$ is given by Lemma~\ref{lem:bound_for_negative_moments} applied to $\beta$ and $2p$.
    Using that $W_{at}(\beta) \geq \e^{\beta X_\xi(s_i) - \psi(\beta)s_i} W_{at - s_i}^{(i)}(\beta)$, applying Lemmas~\ref{lem:von_Bahr} together with Lemma~\ref{lem:bound_for_negative_moments}, and finally Lemma~\ref{lem:concentration_for_ratio}, we obtain
    \begin{align}
        \E_{\Q_{2 \beta}} \left[ V_2(t)^p \right] &\leq \E_{\Q_{2 \beta}} \left[ \left| \sum_{i > n, s_i \leq at} \frac{W_{t - s_i}^{(i)}(\beta) - W_{at - s_i}^{(i)}(\beta)}{W_{at - s_i}^{(i)}(\beta)} \right|^p \left( \sum_{j = 1}^n \e^{\beta X_\xi(s_j) - \psi(\beta)s_j} W_{\gamma t - s_j}^{(j)}(\beta) \right)^{-2p} \mathds{1}_A \right] \nonumber \\
        &\leq C \E_{\Q_{2 \beta}} \left[ \sum_{i > n, s_i \leq at} \E_{\Q_{2 \beta}} \left[ \left| \frac{W_{t - s_i}^{(i)}(\beta) - W_{at - s_i}^{(i)}(\beta)}{W_{at - s_i}^{(i)}(\beta)} \right|^p \middle| \Pi \right] \e^{\alpha_p s_n} \mathds{1}_A \right] \nonumber \\
        &\leq C \E_{\Q_{2 \beta}} \left[ \sum_{i > n, s_i \leq at} \e^{-(1 - \beta^2)p(at - s_i)/2} \e^{\alpha_p s_n} \mathds{1}_A \right]. \nonumber
    \end{align}
    We then decompose the sum over the spinal branching times located in $[0, a't]$ and those in $(a't, at]$.
    First,
    \begin{equation*}
        \E_{\Q_{2 \beta}} \left[ \sum_{i > n, s_i \leq a't} \e^{-(1 - \beta^2)p(at - s_i)/2} \e^{\alpha_p s_n} \mathds{1}_A \right] \leq \e^{-(1 - \beta^2)p(a - a')t/2} \E_{\Q_{2 \beta}} \left[ \#(\Pi \cap [0, a't]) \e^{\alpha_p s_n} \right].
    \end{equation*}
    Note that $\#(\Pi \cap [0, a't])$ has Poisson distribution with parameter $2a't$ and then its moment of order $k \geq 1$ is an $O(t^k)$.
    Thanks to H\"{o}lder's inequality, we see that the last expectation factor is an $O(t)$ since $\alpha_p < 2$.
    In particular, the above quantity is bounded.
    Now, by independence,
    \begin{align}
        \E_{\Q_{2 \beta}} \left[ \sum_{i > n, a't < s_i \leq at} \e^{-(1 - \beta^2)p(at - s_i)/2} \e^{\alpha_p s_n} \mathds{1}_A \right] &= \E_{\Q_{2 \beta}} \left[ \sum_{a't < s_i \leq at} \e^{-(1 - \beta^2)p(at - s_i)/2} \right] \E_{\Q_{2 \beta}} \left[ \e^{\alpha_p s_n} \mathds{1}_A \right] \nonumber \\
        &\leq \left( \int_{a't}^{at} \e^{-(1 - \beta^2)p(at - s)/2} 2\d{s} \right) \E_{\Q_{2 \beta}} \left[ \e^{\alpha_p s_n} \right], \nonumber
    \end{align}
    which is also bounded since $\alpha_p < 2$.
    Finally, the process $V_2(t)$ is bounded in $L^p(\Q_{2\beta})$ and then its convergence \as to $0$ implies its convergence in mean.
    Hence~\eqref{eq:replacement_of_the_denominator_A_reformulation} for $i = 2$.

    \textbf{Control of $T_3$.}
    Let us show~\eqref{eq:replacement_of_the_denominator_A_reformulation} for $i = 3$.
    Let $\gamma = 1 \text{ or } a$.
    We have
    \begin{align}
        &\E_{\Q_{2\beta, at}} \left[ W_{t - at}^{(\xi, at)}(\beta)^2 \frac{T_3(t)}{W_{\gamma t}(\beta)^2 W_{at}(\beta)} \mathds{1}_A \right] \nonumber \\
        &\quad \leq \E \left| W_{t - at}(\beta)^3 - W_{t - at}(\beta)^2 \right| \E_{\Q_{2\beta}} \left[ \frac{\e^{\beta X_\xi(at) - \psi(\beta) at}}{W_{at}(\beta)} \left( \sum_{i = 1}^n \e^{\beta X_\xi(s_i) - \psi(\beta)s_i} W_{\gamma t - s_i}^{(i)}(\beta) \right)^{-2} \mathds{1}_A \right]. \nonumber
    \end{align}
    Let us denote by $V_3(t)$ the process appearing in the last expectation.
    Since the additive martingale is bounded in $L^3(\P)$, it suffices to show that $V_3(t)$ converges in mean to $0$ as $t \to \infty$.
    But it converges $\Q_{2\beta}$-\as to $0$ since
    \begin{itemize}
        \item the processes $W_{at}(\beta)$ and $W_{\gamma t - s_i}^{(i)}(\beta)$ for $i \geq 1$ converge \as to some positive limits as $t \to \infty$, by Lemma~\ref{lem:almost_sure_convergence_of_W_t},
        \item $\beta X_\xi(s) - \psi(\beta)s = \beta B_s - (1 - 3\beta^2/2)s$, where $B_s$ is a standard Brownian motion, and $1 - 3\beta^2/2>0$.
    \end{itemize}
    In addition, the process $V_3(t)$ is bounded in $L^p(\Q_{2\beta})$.
    Indeed, noting that $\e^{\beta X_\xi(at) - \psi(\beta) at} \leq W_{at}(\beta)$,
    \begin{equation*}
        \E_{\Q_{2\beta}} \left[ V_3(t)^p \right] \leq \E_{\Q_{2\beta}} \left[ \left( \sum_{i = 1}^n \e^{\beta X_\xi(s_i) - \psi(\beta)s_i} W_{\gamma t - s_i}^{(i)}(\beta) \right)^{-2p} \right],
    \end{equation*}
    and we see that the right-hand side is bounded thanks to Lemma~\ref{lem:bound_for_negative_moments}.
    Finally, the process $V_3(t)$ converges in mean to $0$.
    Hence~\eqref{eq:replacement_of_the_denominator_A_reformulation} for $i = 3$.
    This concludes the proof.
\end{proof}

\subsection{Lower temperature}

\begin{proof}[Proof of Theorem~\ref{thm:asymptotics_in_mean}.\ref{it:asymptotics_in_mean_4}]
    We tackle separately the lower bound and the upper bound.
    
    \textbf{Lower bound:}
    Our goal is to find a constant $c > 0$ such that, for every $t \geq 1$,
    \begin{equation}\label{eq:lower_temperature_lower_bound_formulation1}
        \Expec{\nu_{\beta, t}([a, 1])} \geq c t^{-3/2} \e^{-(2 - \beta^2)^2 at/8\beta^2}.
    \end{equation}
    It will be useful for the proof of Theorem~\ref{thm:asymptotics_in_mean}.\ref{it:asymptotics_in_mean_3} to note that the arguments below work for $\beta = \sqrt{2\smash{/}3}$ as well, so we only assume $\beta \in [\sqrt{2\smash{/}3}, \sqrt{2})$.
    
    We can restrict ourselves to the event $\Gamma = \Gamma_{at} = \{\sup_{s \in [0, at]} (X_\xi(s) - v(\beta) s + f_{at}(s)) \leq 1\}$, where $v(\beta) = 1/\beta + \beta/2$, $f_t(s) = s^\alpha \wedge (t-s)^\alpha$ and $\alpha$ is an arbitrary number in $(0, 1/2)$.
    By~\eqref{eq:overlap_rewriting}, the definition of $\Q_{2\beta,at}$, and then Lemma~\ref{lem:gibbs_measure_weights},
    \begin{align}
        \Expec{\nu_{\beta, t}([a, 1])} &= \e^{(\beta^2 - 1) at} \E_{\Q_{2\beta, at}} \left[ \frac{1}{W_t(\beta)^2} \sum_{u \in \cN(at)} \frac{\e^{2\beta X_u(at) - \psi(2\beta)at}}{W_{at}(2\beta)} W_{t - at}^{(u, at)}(\beta)^2 \right] \nonumber \\
        &= \e^{(\beta^2 - 1) at} \E_{\Q_{2\beta, at}} \left[ W_{t - at}^{(\xi, at)}(\beta)^2 W_t(\beta)^{-2} \right] \geq \e^{(\beta^2 - 1) at} \E_{\Q_{2\beta, at}} \left[ W_{t - at}^{(\xi, at)}(\beta)^2 W_t(\beta)^{-2} \mathds{1}_\Gamma \right]. 
        \label{eq:overlap_rewriting_2}
    \end{align}
    Then, to prove~\eqref{eq:lower_temperature_lower_bound_formulation1}, it suffices to get
    \begin{equation}\label{eq:lower_temperature_lower_bound_formulation2}
        \E_{\Q_{2\beta, at}} \left[ W_{t - at}^{(\xi, at)}(\beta)^2 W_t(\beta)^{-2} \mathds{1}_\Gamma \right] \geq c t^{-3/2} \e^{-(2\beta - v(\beta))^2at/2}.
    \end{equation}
    Decomposing $W_t(\beta)$ along the spine as in \eqref{eq:decompo_W_t_spine}, we can rewrite the left-hand side of~\eqref{eq:lower_temperature_lower_bound_formulation2} as
    \begin{equation*}
        \E_{\Q_{2\beta, at}} \left[ W_{t - at}^{(\xi, at)}(\beta)^2 (U_t + V_t W_{t - at}^{(\xi, at)}(\beta))^{-2} \mathds{1}_\Gamma \right],
    \end{equation*}
    where
    \begin{equation} \label{eq:def_U_V}
        U_t = \sum_{i \geq 1, s_i \leq at} \e^{\beta X_\xi(s_i) - \psi(\beta)s_i} W_{t - s_i}^{(i)}(\beta) \quad \text{and} \quad V_t = \e^{\beta X_\xi(at) - \psi(\beta)at}.
    \end{equation}
    For all $U, V, W > 0$,
    \begin{equation*}
        W^2 (U + V W)^{-2} \geq W^2 (V + V W)^{-2} \wedge W^2 (U + U W)^{-2} \geq W^2 (1 + W)^{-2} (U + V)^{-2}.
    \end{equation*}
    Thus, the left-hand side of~\eqref{eq:lower_temperature_lower_bound_formulation2} admits the following lower bound,
    \begin{equation*}
        \E \left[ W_{t - at}(\beta)^2 (1 + W_{t - at}(\beta))^{-2} \right] \E_{\Q_{2\beta}} \left[ (U_t + V_t)^{-2} \mathds{1}_\Gamma \right] 
        \geq \E \left[ W_\infty(\beta)^2 (1 + W_\infty(\beta))^{-2} \right] \E_{\Q_{2\beta}} \left[ (U_t + V_t)^{-2} \mathds{1}_\Gamma \right],  \nonumber
    \end{equation*}
    by Fatou's lemma.
    Since $W_\infty(\beta) > 0$ \as, to prove~\eqref{eq:lower_temperature_lower_bound_formulation2}, it suffices to get
    \begin{equation}\label{eq:lower_temperature_lower_bound_formulation3}
        \E_{\Q_{2\beta}} \left[ (U_t + V_t)^{-2} \mathds{1}_\Gamma \right] \geq c t^{-3/2} \e^{-(2\beta - v(\beta))^2at/2}.
    \end{equation}
    On the event $\Gamma$, since $\beta v(\beta) = \psi(\beta)$,
    \begin{equation*}
        U_t \leq \e^{\beta} \sum_{i \geq 1, s_i \leq at} \e^{-\beta f_{at}(s_i)} W_{t - s_i}^{(i)}(\beta) \quad \text{and} \quad V_t \leq \e^{\beta}.
    \end{equation*}
    Then,
    \begin{equation}\label{eq:lower_temperature_lower_bound_markov_property}
        \E_{\Q_{2\beta}} \left[ (U_t + V_t)^{-2} \mathds{1}_\Gamma \right] \geq \e^{-2 \beta} \E_{\Q_{2\beta}} \left[ \left( \sum_{i \geq 1, s_i \leq at} \e^{-\beta f_{at}(s_i)} W_{t - s_i}^{(i)}(\beta) + 1 \right)^{-2} \right] \Q_{2\beta} \left( \Gamma \right).
    \end{equation}
    By Jensen's inequality and recalling the $s_i$'s are the points of a PPP($2 \d{s}$),
    \begin{equation}\label{eq:lower_temperature_lower_bound_jensen}
        \E_{\Q_{2\beta}} \left[ \left( \sum_{i \geq 1, s_i \leq at} \e^{-\beta f_{at}(s_i)} W_{t - s_i}^{(i)}(\beta) + 1 \right)^{-2} \right] \geq \left( \int_0^{at} \e^{-\beta f_{at}(s)} 2 \d{s} + 1 \right)^{-2} \geq c > 0.
    \end{equation}
    Besides, recalling that, under $\Q_{2\beta}$, $(X_\xi(s))_{s\geq 0}$ is a Brownian motion with drift $2\beta$, it follows from Girsanov theorem that
    \begin{equation}\label{eq:lower_temperature_lower_bound_before_ballot_theorem}
        \Q_{2\beta} \left( \Gamma \right) 
        = \E \left[ \e^{(2\beta - v(\beta))B_{at}-(2\beta - v(\beta))^2 at/2} \mathds{1}_{\{\sup_{s \in [0, at]} (B_s + f_{at}(s)) \leq 1\}} \right],
    \end{equation}
    where $(B_s)_{s\geq 0}$ denotes a standard Brownian motion.
    We then restrict ourselves to the event $\{ B_{at} \in [0,1] \}$ and use that $2\beta - v(\beta) \geq 0$ to get
    \begin{equation}
        \Q_{2\beta} \left( \Gamma \right) 
        \geq \e^{-(2\beta - v(\beta))^2 at/2} \P \left( \sup_{s \in [0, at]} (B_s + f_{at}(s)) \leq 1, B_{at} \in [0, 1] \right) 
        \geq c \e^{-(2\beta - v(\beta))^2 at/2} t^{-3/2}, \label{eq:lower_temperature_lower_bound_ballot_theorem}
    \end{equation}
    by~\eqref{eq:ballot_theorem_plus}.
    By~\eqref{eq:lower_temperature_lower_bound_markov_property}, \eqref{eq:lower_temperature_lower_bound_jensen} and~\eqref{eq:lower_temperature_lower_bound_ballot_theorem}, the inequality~\eqref{eq:lower_temperature_lower_bound_formulation3} holds.
    Hence, it proves~\eqref{eq:lower_temperature_lower_bound_formulation1} and the lower bound in Theorem~\ref{thm:asymptotics_in_mean}.\ref{it:asymptotics_in_mean_4}.
    
    \textbf{Upper bound:}
    Our goal is to show that
    \begin{equation}\label{eq:formulation1}
        \Expec{\nu_{\beta, t}([a, 1])} = O(t^{-3/2} \e^{-(2 - \beta^2)^2 at/8\beta^2}).
    \end{equation}
    We work with $\beta \in [\sqrt{2\smash{/}3},\sqrt{2})$, except in the last step where we specify that $\beta>\sqrt{2\smash{/}3}$, so that the beginning of the argument can be used to prove the upper bound in Theorem~\ref{thm:asymptotics_in_mean}.\ref{it:asymptotics_in_mean_3} as well.
    
    Consider an arbitrary integer $n \geq 1$ and define $\bar{A} = \bar{A}_{2n, at} = \{\forall u \in \cN(at), |u| \geq 2n\}$.
    As seen in~\eqref{eq:probability_of_Ac}, we have
    \begin{equation}\label{eq:lower_temperature_upper_bound_complementary_event}
        \Expec{\nu_{\beta, t}([a, 1]) \mathds{1}_{\bar{A}^c}} \leq \P(\bar{A}^c) \leq C t^{2n-1} \e^{-at} = o(t^{-3/2} \e^{-(2 - \beta^2)^2 at/8\beta^2}),
    \end{equation}
    using that $(2 - \beta^2)^2/8\beta^2 \leq 1/3 < 1$ for $\beta \in [\sqrt{2\smash{/}3}, \sqrt{2})$.
    Concerning the contribution of the event $\bar{A}$, we can rewrite, similarly to~\eqref{eq:overlap_rewriting_2},
    \begin{align}
        \Expec{\nu_{\beta, t}([a, 1]) \mathds{1}_{\bar{A}}} &= \e^{(\beta^2 - 1) at} \E_{\Q_{2\beta, at}} \left[ W_{t - at}^{(\xi, at)}(\beta)^2 W_t(\beta)^{-2} \mathds{1}_{\bar{A}} \right] \nonumber \\
        &\leq \e^{(\beta^2 - 1) at} \E_{\Q_{2\beta, at}} \left[ W_{t - at}^{(\xi, at)}(\beta)^2 W_t(\beta)^{-2} \mathds{1}_A \right], \label{eq:lower_temperature_upper_bound_change_of_measure}
    \end{align}
    where $A = A_{2n, at} = \{|\xi(at)| \geq 2n\}$.
    By~\eqref{eq:lower_temperature_upper_bound_complementary_event} and~\eqref{eq:lower_temperature_upper_bound_change_of_measure}, to prove~\eqref{eq:formulation1}, it suffices to get
    \begin{equation}\label{eq:formulation2}
        \E_{\Q_{2\beta, at}} \left[ W_{t - at}^{(\xi, at)}(\beta)^2 W_t(\beta)^{-2} \mathds{1}_A \right] = O(t^{-3/2} \e^{-(2\beta - v(\beta))^2 at/2}),
    \end{equation}
    where $v(\beta) = 1/\beta + \beta/2$.
    Then we decompose $W_t(\beta) = U_t + V_t W_{t - at}^{(\xi, at)}(\beta)$ with $U_t$ and $V_t$ defined in~\eqref{eq:def_U_V}. We also change measures, first using that $\Q_{2\beta, at}$ has density $\e^{2\beta X_\xi(at)-2\beta^2at}$ w.r.t.\@ $\Q_{0, at}$, and then applying Girsanov theorem under $\Q_{0, at}$ to trade the factor $\e^{v(\beta)X_\xi(at)-v(\beta)^2at/2}$ for a drift $v(\beta)$ added to the spine trajectory, so that $U_t$ and $V_t$ become (recall $\psi(\beta) = \beta v(\beta)$)
    \begin{equation*}
        U_t' = \sum_{i \geq 1, s_i \leq at} \e^{\beta X_\xi(s_i)} W_{t - s_i}^{(i)}(\beta) \quad \text{and} \quad V_t' = \e^{\beta X_\xi(at)}.
    \end{equation*}
    Therefore, we get that~\eqref{eq:formulation2} is equivalent to
    \begin{equation}\label{eq:formulation3}
        \E_{\Q_{0, at}} \left[ \e^{(2\beta - v(\beta)) X_\xi(at)} W_{t - at}^{(\xi, at)}(\beta)^2 (U_t' + V_t' W_{t - at}^{(\xi, at)}(\beta))^{-2} \mathds{1}_A \right] = O(t^{-3/2}).
    \end{equation}
    
    For all $U, V, W > 0$ and any $\alpha \in [0, 2]$,
    \begin{equation*}
        W^2 (U + V W)^{-2} \leq W^2 U^{-\alpha} (V W)^{-2 + \alpha} 
        = U^{-\alpha} V^{-2 + \alpha} W^\alpha.
    \end{equation*}
    Thus, the left-hand side of~\eqref{eq:formulation3} is bounded by
    \begin{align}
        &\E_{\Q_{0, at}} \left[ \e^{(2\beta - v(\beta))X_\xi(at)} (U_t')^{-\alpha} (V_t')^{-2 + \alpha} W_{t - at}^{(\xi, at)}(\beta)^\alpha \mathds{1}_A \right] \nonumber \\
        &= \E \left[ W_{t - at}(\beta)^\alpha \right] \E_{\Q_0} \left[ \e^{(\alpha \beta - v(\beta)) X_\xi(at)} (U_t')^{-\alpha} \mathds{1}_A \right]. \nonumber
    \end{align}
    Let $\alpha \in [0, 2] \cap [1, 2/\beta^2)$ so that the additive martingale is bounded in $L^\alpha(\P)$ by Lemma~\ref{lem:additive_martingale_convergence_in_Lp}.
    To prove~\eqref{eq:formulation3}, it suffices to get
    \begin{equation}\label{eq:formulation4}
        \E_{\Q_0} \left[ \e^{(\alpha \beta - v(\beta)) X_\xi(at)} (U_t')^{-\alpha} \mathds{1}_A \right] = O(t^{-3/2}).
    \end{equation}
    Now, decompose the process over the following events, indexed by integers $k, \ell \geq 0$,
    \begin{equation}\label{eq:definition_E_kl}
        E_{k, \ell} = E_{k, \ell}(X_\xi(s), 0 \leq s \leq at) = \left\{ \sup_{[0, at]} X_\xi \in [k, k+1], \sup_{[0, at]} X_\xi - X_\xi(at) \in [\ell, \ell+1] \right\}.
    \end{equation}
    Fix $k, \ell$ and denote $s_{k, n}$ the $n$-th point of $\Pi$ after time $\tau_k = \inf\{s \geq 0 : X_\xi(s) = k\}$.
    Let us first treat the case $s_{k, n} \leq at$.
    On the event $E_{k, \ell} \cap \{s_{k, n} \leq at\}$,
    \begin{align}
        &\E_{\Q_0} \left[ \e^{(\alpha \beta - v(\beta)) X_\xi(at)} (U_t')^{-\alpha} \middle| \Pi, X_\xi \right] \nonumber \\
        &\quad \leq C \e^{-v(\beta)k - (\alpha \beta - v(\beta))\ell} \E_{\Q_0} \left[ \left( \sum_{i \geq 1, \tau_k \leq s_i \leq s_{k, n}} \e^{\beta(X_\xi(s_i) - X_\xi(\tau_k))} W_{t - s_i}^{(i)}(\beta) \right)^{-\alpha} \middle| \Pi, X_\xi \right] \nonumber \\
        &\quad \leq C \e^{-v(\beta)k - (\alpha \beta - v(\beta))\ell} \e^{\alpha \beta \sup_{[\tau_k, s_{k, n}]} (X_\xi(\tau_k) - X_\xi)} \E_{\Q_0} \left[ \left( \sum_{i \geq 1, \tau_k \leq s_i \leq s_{k, n}} W_{t - s_i}^{(i)}(\beta) \right)^{-\alpha} \middle| \Pi, X_\xi \right].  \nonumber
    \end{align}
    By Lemma~\ref{lem:negative_moments_linear_combination} and Lemma~\ref{lem:left_tail_additive_martingale}, we can choose $n$ so that the last conditional expectation is bounded, which yields
    \begin{align}
        & \E_{\Q_0} \left[ \e^{(\alpha \beta - v(\beta)) X_\xi(at)} (U_t')^{-\alpha} \mathds{1}_A \mathds{1}_{E_{k, \ell} \cap \{s_{k, n} \leq at\}} \right] \nonumber \\
        &\quad \leq C \e^{-v(\beta)k - (\alpha \beta - v(\beta))\ell} \E_{\Q_0} \left[ \e^{\alpha \beta \sup_{[\tau_k, s_{k, n}]} (X_\xi(\tau_k) - X_\xi)} \mathds{1}_{E_{k, \ell} \cap \{s_{k, n} \leq at\}} \right]. \label{eq:lower_temperature_upper_bound_enough_branching_times}
    \end{align}
    By using time reversal invariance of Brownian motion and Poisson point process, we can treat the case $s_{k, n} > at$ similarly.
    Indeed, on the event $A$, if $s_{k, n} > at$, then there are at least $n$ spinal branching times between $0$ and $\tau_k$.
    Define $s_i' = at - s_i$, $\Pi'$ the corresponding point process and $X_\xi'(s) = X_\xi(at - s) - X_\xi(at)$. Denote $s_{\ell, n}'$ the $n$-th point of $\Pi'$ after time $\tau_\ell' = \inf\{s \geq 0 : X_\xi'(s) = \ell\}$ if it exists, and set $s_{\ell, n}' = +\infty$ otherwise.
    Note that, on $E_{k,\ell}$, we have $\tau_k \leq at-\tau_\ell'$, so on the event $E_{k,\ell} \cap A \cap \{ s_{k, n} > at \}$, there are at least $n$ point of $\Pi'$ after $\tau_\ell'$, so we get
    \begin{align}
        &\E_{\Q_0} \left[ \e^{(\alpha \beta - v(\beta)) X_\xi(at)} (U_t')^{-\alpha} \mathds{1}_A \mathds{1}_{E_{k, \ell} \cap \{s_{k, n} > at\}} \right] \nonumber \\
        &\leq \E_{\Q_0} \left[ \e^{-(\alpha \beta - v(\beta)) X_\xi'(at)} \left( \sum_{i \geq 1, \tau_\ell' \leq s_i' \leq s_{\ell, n}'} \e^{\beta(X_\xi'(s_i') - X_\xi'(at))} W_{t - (at - s_i')}^{(i)}(\beta) \right)^{-\alpha} \mathds{1}_{E_{\ell, k}(X_\xi'(s), 0 \leq s \leq at) \cap \{s_{\ell, n}' \leq at\}} \right] \nonumber \\
        &= \E_{\Q_0} \left[ \e^{-(\alpha \beta - v(\beta)) X_\xi(at)} \left( \sum_{i = 1}^n \e^{\beta(X_\xi(s_{\ell, i}) - X_\xi(at))} W_{t - (at - s_{\ell, i})}^{(i)}(\beta) \right)^{-\alpha} \mathds{1}_{E_{\ell, k}(X_\xi(s), 0 \leq s \leq at) \cap \{s_{\ell, n} \leq at\}} \right] \nonumber \\
        &\leq C \e^{-v(\beta)k -(\alpha \beta - v(\beta))\ell} \E_{\Q_0} \left[ \e^{\alpha \beta \sup_{[\tau_\ell, s_{\ell, n}]} (X_\xi(\tau_\ell) - X_\xi)} \mathds{1}_{E_{\ell, k} \cap \{s_{\ell, n} \leq at\}} \right], \label{eq:lower_temperature_upper_bound_not_enough_branching_times}
    \end{align}
    by Lemma~\ref{lem:negative_moments_linear_combination} and Lemma~\ref{lem:left_tail_additive_martingale} as before.
    Using $\alpha \leq 2$ if $\beta<1$ and $\alpha < 2/\beta^2$ otherwise, we have $\alpha \beta < 2$, which allows us to apply Lemma~\ref{lem:control_on_E_kl} below.
    This together with~\eqref{eq:lower_temperature_upper_bound_enough_branching_times} and~\eqref{eq:lower_temperature_upper_bound_not_enough_branching_times} lead to
    \begin{equation} \label{eq:lower_temperature_upper_bound_last_but_one_step}
        \E_{\Q_0} \left[ \e^{(\alpha \beta - v(\beta)) X_\xi(at)} (U_t')^{-\alpha} \mathds{1}_A \right] \leq C t^{-3/2} \sum_{k = 0}^\infty \sum_{\ell = 0}^\infty \e^{-v(\beta)k - (\alpha \beta - v(\beta))\ell} (k + \ell + 1) \e^{- c \ell^2/t}.
    \end{equation}
    We finally assume that $\beta \in (\sqrt{2\smash{/}3},\sqrt{2})$, so that $v(\beta)/\beta < 2 \wedge 2/\beta^2$ and we can thus choose $\alpha > v(\beta)/\beta$. This ensures that the summand in~\eqref{eq:lower_temperature_upper_bound_last_but_one_step} decreases exponentially fast in $\ell$, and be can simply bound $\e^{- c \ell^2/t} \leq 1$.
    Then, the double sum in~\eqref{eq:lower_temperature_upper_bound_last_but_one_step} is a $O(1)$, which proves~\eqref{eq:formulation4} and therefore~\eqref{eq:formulation1}.
    This concludes the proof.
\end{proof}

\begin{lemma}\label{lem:control_on_E_kl}
    For any $\gamma \in [0,2)$, there exist $C, c > 0$ such that, for all $t \geq 1$ and $k, \ell \geq 0$,
    \begin{equation}\label{eq:control_on_E_kl}
        \E_{\Q_0} \left[ \e^{\gamma \sup_{[\tau_k, s_{k, n}]} (X_\xi(\tau_k) - X_\xi)} \mathds{1}_{E_{k, \ell}(X_\xi(s), 0 \leq s \leq t)} \mathds{1}_{\{s_{k, n} \leq t\}} \right] \leq C (k + \ell + 1) t^{-3/2} \e^{-c k^2/t - c \ell^2/t}.
    \end{equation}
\end{lemma}

\begin{remark}
    The exponential factor in the bound~\eqref{eq:control_on_E_kl} has not been used to prove~\eqref{eq:formulation1}, because we bounded the factor $\e^{- c \ell^2/t}$ in~\eqref{eq:lower_temperature_upper_bound_last_but_one_step} by 1.
    However, it will play a key role in the proof of the upper bound for Theorem~\ref{thm:asymptotics_in_mean}.\ref{it:asymptotics_in_mean_3}, which is also based on~\eqref{eq:lower_temperature_upper_bound_last_but_one_step}.
\end{remark}

To prove Lemma~\ref{lem:control_on_E_kl}, we will need the following intermediate lemma.

\begin{lemma}\label{lem:control_on_E_kl_before_tau_k}
    Let $\sigma_n$ be a random variable with distribution $\Gamma(n, 2)$ and $(B_t)_{t \geq 0}$ be a standard Brownian motion independent of $\sigma_n$.
    For any $\gamma < 2$, there exist $C, c > 0$ such that, for any $t > 0$,
    \begin{equation}\label{eq:control_on_E_kl_before_tau_k}
        \E \left[ \e^{\gamma \sup_{[0, \sigma_n]} B} \mathds{1}_{\{\inf_{[0, t]} B \geq -1, B_t \in [\ell-1, \ell+1], \sigma_n \leq t\}} \right] \leq \begin{cases}
            C (t + 1)^{-3/2} & \text{if } \ell = 0 \text{ or } 1, \\
            C \ell t^{-3/2} \e^{-c \ell^2/t} & \text{if } \ell \geq 2.
        \end{cases}
    \end{equation}
\end{lemma}

\begin{proof}
    First consider $\ell = 0$ or $1$.
    The proof of the case $\ell \geq 2$ will be similar although more technical.
    We have
    \begin{align}
        &\E \left[ \e^{\gamma \sup_{[0, \sigma_n]} B} \mathds{1}_{\{\inf_{[0, t]} B \geq -1, B_t \in [\ell-1, \ell+1], \sigma_n \leq t\}} \right] \nonumber \\
        &\quad = C \int_0^t \E \left[ \e^{\gamma \sup_{[0, s]} B} \mathds{1}_{\{\inf_{[0, t]} B \geq -1, B_t \in [\ell-1, \ell+1]\}} \right] s^{n-1} \e^{-2s} \d{s}. \label{eq:control_on_E_kl_before_tau_k_integrated_formulation}
    \end{align}
    Bounding the indicator function by $1$ and using $\E[e^{\gamma \sup_{[0,s]} B}] \leq 2 \e^{\gamma^2s/2}$, we obtain
    \begin{align}
        \int_{t/2}^t \E \left[ \e^{\gamma \sup_{[0, s]} B} \mathds{1}_{\{\inf_{[0, t]} B \geq -1, B_t \in [\ell-1, \ell+1]\}} \right] s^{n-1} \e^{-2s} \d{s} & \leq C \int_{t/2}^t s^{n - 1} \e^{-(2 - \gamma^2/2)s} \d{s} \nonumber \\ 
        & \leq C t^n \e^{-(2 - \gamma^2/2)t/2}. \label{eq:control_on_E_kl_before_tau_k_large_s}
    \end{align}
    Concerning the integration over $s \in [0, t/2]$, we use Markov property at time $s$ and Lemma~\ref{lem:brownian_estimate} (recall $\ell = 0$ or $1$) to get 
    \begin{align}
        \E \left[ \e^{\gamma \sup_{[0, s]} B} \mathds{1}_{\{\inf_{[0, t]} B \geq -1, B_t \in [\ell-1, \ell+1]\}} \right] & \leq \E \left[ \e^{\gamma \sup_{[0, s]} B} \mathds{1}_{\{B_s \geq -1\}} \left( C (B_s + 2) (t - s)^{-3/2} \wedge 1 \right) \right] \nonumber \\ 
        & \leq C \E \left[ \e^{\gamma \sup_{[0, s]} B} \mathds{1}_{\{B_s \geq -1\}} (B_s + 2) (t - s + 1)^{-3/2} \right].
        \label{eq:control_on_E_kl_before_tau_k_ballot}
    \end{align}
    Note that $(t - s + 1)^{-3/2} \leq C(t + 1)^{-3/2}$ since $s \leq t/2$.
    Moreover, $(B_s, \sup_{[0, s]} B)$ has the following density (see \eg \cite[Equation~2]{Yor1997} or \cite[Formula~1.1.1.8]{BorodinSalminen2002}),
    \begin{equation}\label{eq:density_joint_law_brownian_motion_and_supremum}
        (x, y) \mapsto \sqrt{2/\pi} s^{-3/2} (2y - x) \e^{-(2y - x)^2/2s} \mathds{1}_{x \leq y, 0  \leq y} \d{x} \d{y}.
    \end{equation}
    We deduce that, up to a multiplicative constant, \eqref{eq:control_on_E_kl_before_tau_k_ballot} is bounded by
    \begin{align}
        &\frac{1}{(t + 1)^{3/2} s^{3/2}} \int_0^\infty (y + 1)^2 \e^{\gamma y} \int_{-1}^y \e^{-(2y - x)^2/2s} \d{x} \d{y} \nonumber \\
        &\quad \leq \frac{C}{(t + 1)^{3/2} s} \int_{-\infty}^\infty (y + 1)^2 \e^{\gamma y - y^2/2s} \d{y} = \frac{P(s) \e^{\gamma^2 s/2}}{(t + 1)^{3/2} \sqrt{s}}, \nonumber 
    \end{align}
    for some polynomial $P$.
    By integrating this against $s^{n-1} \e^{-2s} \mathds{1}_{s \in [0, t/2]} \d{s}$, we obtain
    \begin{equation}\label{eq:control_on_E_kl_before_tau_k_small_s}
        \int_0^{t/2} \E \left[ \e^{\gamma \sup_{[0, s]} B} \mathds{1}_{\{\inf_{[0, t]} B \geq -1, B_t \in [\ell-1, \ell+1]\}} \right] s^{n-1} \e^{-2s} \d{s} \leq C (t + 1)^{-3/2}.
    \end{equation}
    By~\eqref{eq:control_on_E_kl_before_tau_k_large_s} and~\eqref{eq:control_on_E_kl_before_tau_k_small_s}, the bound~\eqref{eq:control_on_E_kl_before_tau_k} holds for $\ell = 0$ or $1$.

    Now consider $\ell \geq 2$.
    By H\"{o}lder's inequality, for all $p, q > 1$ such that $1/p + 1/q = 1$,
    \begin{align}
        \E \left[ \e^{\gamma \sup_{[0, s]} B} \mathds{1}_{\{\inf_{[0, t]} B \geq -1, B_t \in [\ell-1, \ell+1]\}} \right] 
        &\leq \E \left[ \e^{p \gamma \sup_{[0, s]} B} \right]^{1/p} \P(B_t \geq \ell - 1)^{1/q} \nonumber \\
        & \leq C \e^{p \gamma^2 s/2} \e^{-(\ell - 1)^2/2qt}. \label{eq:control_on_E_kl_before_tau_k_holder}
    \end{align}
    Taking $p$ close enough to $1$ so that $p\gamma^2 < 4$ and integrating~\eqref{eq:control_on_E_kl_before_tau_k_holder} against $s^{n-1} \e^{-2s} \mathds{1}_{s \in [t/2, t]} \d{s}$, we obtain
    \begin{equation}\label{eq:control_on_E_kl_before_tau_k_large_s_2}
        \int_{t/2}^t \E \left[ \e^{\gamma \sup_{[0, s]} B} \mathds{1}_{\{\inf_{[0, t]} B \geq -1, B_t \in [\ell-1, \ell+1]\}} \right] s^{n-1} \e^{-2s} \d{s} \leq C \e^{-c t - c \ell^2/t}.
    \end{equation}
    The integration over $s \in [0, t/2]$ is more intricate.
    If $t \leq 2$, then~\eqref{eq:control_on_E_kl_before_tau_k_holder} implies that
    \begin{equation} \label{eq:control_on_E_kl_before_tau_k_small_s_small_t}
        \int_0^{t/2} \E \left[ \e^{\gamma \sup_{[0, s]} B} \mathds{1}_{\{\inf_{[0, t]} B \geq -1, B_t \in [\ell-1, \ell+1]\}} \right] s^{n-1} \e^{-2s} \d{s} 
        \leq C \e^{-c\ell^2/t} 
        \leq C \ell (t/2)^{-3/2} \e^{-c\ell^2/t}.
    \end{equation}
    Now assume that $t > 2$ so that $t - s > 1$.
    By using Markov property and Lemma~\ref{lem:brownian_estimate},
    \begin{align}
        &\E \left[ \e^{\gamma \sup_{[0, s]} B} \mathds{1}_{\{\inf_{[0, t]} B \geq -1, B_t \in [\ell-1, \ell+1]\}} \right] \nonumber \\
        &\quad \leq C \E \left[ \e^{\gamma \sup_{[0, s]} B} \mathds{1}_{\{B_s \geq -1\}} \frac{(B_s + 1) \ell}{(t - s)^{3/2}} \left( \e^{-(B_s - \ell - 1)^2/2(t-s)} + \e^{-(B_s - \ell + 1)^2/2(t-s)}\right) \right]. \nonumber 
    \end{align}
    Let $\ell' = \ell \pm 1$.
    Since $(B_s, \sup_{[0, s]} B)$ has density~\eqref{eq:density_joint_law_brownian_motion_and_supremum},
    \begin{align}
        &\E \left[ \e^{\gamma \sup_{[0, s]} B} \mathds{1}_{\{B_s \geq -1\}} (B_s + 1) \e^{-(B_s - \ell')^2/2(t-s)} \right] \nonumber \\
        &\quad \leq \frac{C}{s^{3/2}} \int_0^\infty (y + 1)^2 \e^{\gamma y} \int_{-1}^y \e^{-(x - \ell')^2/2(t-s) - (2y - x)^2/2s} \d{x} \d{y}. \label{eq:control_on_E_kl_before_tau_k_at_l'}
    \end{align}
    Up to a constant in $x$, the function $x \mapsto \e^{-(x - \ell')^2/2(t-s) - (2y - x)^2/2s}$ is the density at time $s$ of a Brownian bridge from $(0, 2y)$ to $(t, \ell')$.
    It achieves its maximum at $x_{\max} = 2y (t-s)/t + \ell' s/t$, which is larger than $y$ since $\ell' > 0$ and $s \leq t/2$.
    We deduce that~\eqref{eq:control_on_E_kl_before_tau_k_at_l'} is bounded by
    \begin{align}
        &\frac{C}{s} \int_0^\infty (y + 1)^2 \e^{\gamma y - (y - \ell')^2/2(t-s) - y^2/2s} \d{y} \nonumber \\
        &\quad = \frac{C}{s} \e^{\gamma^2 s (t-s)/2t + \gamma \ell' s/t - {\ell'}^2/2t} \int_0^\infty (y + 1)^2 \e^{-(y - y_{\max})^2 t/2s(t-s)} \d{y}, \nonumber 
    \end{align}
    where $y_{\max} = \ell' s/t + \gamma s(t-s)/t$.
    Subsequently,
    \begin{align}
        \E \left[ \e^{\gamma \sup_{[0, s]} B} \mathds{1}_{\{B_s \geq -1\}} (B_s + 1) \e^{-(B_s - \ell')^2/2(t-s)} \right] &\leq C \frac{(y_{\max} + 1)^2 + s(t-s)/t}{\sqrt{s}} \e^{\gamma^2 s (t-s)/2t + \gamma \ell' s/t - {\ell'}^2/2t} \nonumber \\
        &\leq \frac{P(s)}{\sqrt{s}} \e^{\gamma^2 s/2 + \gamma \ell' s/t - c{\ell'}^2/t}, \nonumber 
    \end{align}
    for some polynomial $P$, where we have used that $(\ell' s/t)^2 \e^{-{\ell'}^2/2t} \leq s ({\ell'}^2/t) \e^{-{\ell'}^2/2t} \leq C s \e^{-c {\ell'}^2/t}$.
    Hence,
    \begin{align}
        &\int_0^{t/2} \E \left[ \e^{\gamma \sup_{[0, s]} B} \mathds{1}_{\{B_s \geq -1\}} \frac{(B_s + 1) \ell}{(t - s)^{3/2}} \e^{-(B_s - \ell')^2/2(t-s)} \right] s^{n - 1} \e^{-2s} \d{s} \nonumber \\
        &\quad \leq \frac{C \ell}{t^{3/2}} \e^{-c {\ell'}^2/t} \int_0^{t/2} \frac{P(s) s^{n - 1}}{\sqrt{s}} \e^{(-2 + \gamma^2/2 + \gamma \ell'/t)s} \d{s}, \label{eq:bound_for_integral_of_I}
    \end{align}
    Let $\varepsilon \in (0, 1)$.
    \begin{enumerate}
        \item If $t \geq (1 + \varepsilon) \frac{2 \gamma}{4 - \gamma^2} \ell'$, then
        \begin{equation*}
            -2 + \frac{\gamma^2}{2} + \frac{\gamma \ell'}{t} \leq\frac{-4 + \gamma^2}{2} + \frac{4 - \gamma^2}{2 (1 + \varepsilon)} < 0.
        \end{equation*}
        In this case, \eqref{eq:bound_for_integral_of_I} is bounded by $C \ell t^{-3/2} \e^{-c {\ell'}^2/t}$.
        \item If $t < (1 + \varepsilon) \frac{2 \gamma}{4 - \gamma^2} \ell'$, then we can directly use H\"{o}lder's inequality~\eqref{eq:control_on_E_kl_before_tau_k_holder} to bound
        \begin{equation*}
            \int_0^{t/2} \E \left[ \e^{\gamma \sup_{[0, s]} B} \mathds{1}_{\{\inf_{[0, t]} B \geq -1, B_t \in [\ell-1, \ell+1]\}} \right] s^{n-1} \e^{-2s} \d{s} 
            \leq C \e^{-c\ell^2/t} \leq C \e^{-ct} \e^{-c\ell^2/t},
        \end{equation*}
        up to a change of the constant $c$.
    \end{enumerate}
    In both cases, there exist $C, c > 0$ such that, for all $t > 2$,
    \begin{equation} \label{eq:control_on_E_kl_before_tau_k_small_s_large_t}
        \int_0^{t/2} \E \left[ \e^{\gamma \sup_{[0, s]} B} \mathds{1}_{\{\inf_{[0, t]} B \geq -1, B_t \in [\ell-1, \ell+1]\}} \right] s^{n-1} \e^{-2s} \d{s} \leq C \ell t^{-3/2} \e^{-c\ell^2/t}.
    \end{equation}
    By~\eqref{eq:control_on_E_kl_before_tau_k_large_s_2}, \eqref{eq:control_on_E_kl_before_tau_k_small_s_small_t} and~\eqref{eq:control_on_E_kl_before_tau_k_small_s_large_t}, the bound~\eqref{eq:control_on_E_kl_before_tau_k} holds for $\ell \geq 2$.
\end{proof}

\begin{proof}[Proof of Lemma~\ref{lem:control_on_E_kl}]
    If $k = 0$, then $\tau_k = 0$ and~\eqref{eq:control_on_E_kl} is a direct consequence of Lemma~\ref{lem:control_on_E_kl_before_tau_k}, because $X_\xi$ under $\Q_0$ has the same distribution as $-B$.
    Assume $k \geq 1$.
    In order to use Markov property, define $\sigma_n = s_{k, n} - \tau_k$ and
    \begin{equation*}
        F_{\ell}(s) = \e^{\gamma \sup_{[s, s + \sigma_n]} (X_\xi(s) - X_\xi)} \mathds{1}_{\{\inf_{[s, t]} (X_\xi(s) - X_\xi) \geq -1, X_\xi(s) - X_\xi(t) \in [\ell - 1, \ell + 1], \sigma_n \leq t - s\}}.
    \end{equation*}
    The left-hand side of~\eqref{eq:control_on_E_kl} is bounded by
    \begin{equation}\label{eq:control_on_E_kl_markov_property}
        \E_{\Q_0} \left[ F_{\ell}(\tau_k) \mathds{1}_{\{\sup_{[0, \tau_k]} X_\xi \leq k + 1, \tau_k \leq t\}} \right] = \int_0^t \E_{\Q_0} \left[ F_{\ell}(s) \right] \Q_0 \left( \sup_{[0, \tau_k]} X_\xi \leq k + 1, \tau_k \in \d{s} \right).
    \end{equation}
    If $\ell = 0$ or $1$, then by Lemma~\ref{lem:control_on_E_kl_before_tau_k} and \cite[Formula~1.2.1.4.(1)]{BorodinSalminen2002}, up to a multiplicative constant, \eqref{eq:control_on_E_kl_markov_property} is bounded by
    \begin{equation*}
        \int_0^t \frac{k \e^{-k^2/2s}}{(t - s + 1)^{3/2} s^{3/2}} \d{s} \leq \frac{2^{3/2} k}{t^{3/2}} \left( \int_0^{t/2} \frac{\e^{-k^2/2s}}{s^{3/2}} \d{s} + \int_{t/2}^t \frac{\e^{-k^2/2s}}{(t - s + 1)^{3/2}} \d{s} \right),
    \end{equation*}
    which is bounded by $C k t^{-3/2} \e^{-k^2/2t}$, using for the first term that
    \begin{equation}\label{eq:control_on_E_kl_bounds_for_integrals}
    \int_0^{t} \frac{\e^{-k^2/2s}}{s^{3/2}} \d{s} 
    \leq \e^{-k^2/4t} \int_0^\infty \frac{\e^{-k^2/4s}}{s^{3/2}} \d{s} 
    = \frac{C \e^{-k^2/4t}}{k}.
    \end{equation}
    If $ \ell \geq 2$, then by Lemma~\ref{lem:control_on_E_kl_before_tau_k} and \cite[Formula~1.2.1.4.(1)]{BorodinSalminen2002}, up to a multiplicative constant, \eqref{eq:control_on_E_kl_markov_property} is bounded by
    \begin{equation*}
        \int_0^t \frac{\ell \e^{-c \ell^2/(t-s)} k \e^{-k^2/2s}}{(t-s)^{3/2} s^{3/2}} \d{s} \leq \frac{2^{3/2} k \ell}{t^{3/2}} \left( \e^{-c \ell^2/t} \int_0^{t/2} \frac{\e^{-k^2/2s}}{s^{3/2}} \d{s} + \e^{-k^2/2t} \int_{t/2}^t \frac{\e^{-c \ell^2/(t-s)}}{(t - s)^{3/2}} \d{s} \right).
    \end{equation*}
    We can use~\eqref{eq:control_on_E_kl_bounds_for_integrals} to bound this by $C (k + \ell + 1) t^{-3/2} \e^{-c k^2/t - c \ell^2/t}$, which concludes.
\end{proof}

\subsection{Intermediate temperature}

\begin{proof}[Proof of Theorem~\ref{thm:asymptotics_in_mean}.\ref{it:asymptotics_in_mean_3}]
    Let $\beta = \sqrt{2\smash{/}3}$. The ideas are closed to the one used in the case $\beta \in (\sqrt{2\smash{/}3},\sqrt{2})$ in the previous subsection, only the end of the argument has to be adapted.
    
    \textbf{Lower bound:}
    Let $\alpha$ be an arbitrary number in $(0, 1/2)$.
    The arguments used to prove~\eqref{eq:lower_temperature_lower_bound_formulation1} also hold for $\beta = \sqrt{2\smash{/}3}$, but the fact of restricting ourselves to the event $\{ B_{at} \in [0,1] \}$ in the last step~\eqref{eq:lower_temperature_lower_bound_ballot_theorem} of the proof is now too crude: there the factor $\e^{(2\beta - v(\beta))B_{at}}$ prevents $B_{at}$ from being large in the negatives, whereas here this factor does not play a role because $v(\beta)=2\beta$.
    Instead, we stop the reasoning at~\eqref{eq:lower_temperature_lower_bound_before_ballot_theorem} to get 
    \begin{equation*}
        \Expec{\nu_{\beta, t}([a, 1])} \geq c \e^{-at/3} \P \left( \sup_{s \in [0, at]} (B_s + s^\alpha \wedge (at - s)^\alpha) \leq 1 \right).
    \end{equation*}
    By~\eqref{eq:ballot_theorem}, this quantity is larger than $c \e^{-at/3} t^{-1/2}$, which concludes the proof of the lower bound.
    Note that, here, we could have replaced $s^\alpha \wedge (at - s)^\alpha$ with $s^\alpha$ since the resulting probabilities are of the same order.
    
    \textbf{Upper bound:}
    Our goal is to show that $\Expec{\nu_{\beta, t}([a, 1])} = O(t^{-1/2} \e^{-at/3})$.
    We follow the proof of~\eqref{eq:formulation1} up to~\eqref{eq:lower_temperature_upper_bound_last_but_one_step} (recall it has been done in the case $\beta \in [\sqrt{2\smash{/}3},\sqrt{2})$) and choosing $\alpha = 2$ which satisfies the condition $\alpha \in [0, 2] \cap [1, 2/\beta^2)$.
    Noting that $2\beta = v(\beta)$ when $\beta = \sqrt{2\smash{/}3}$, \eqref{eq:lower_temperature_upper_bound_last_but_one_step} becomes
    \begin{equation*}
        \E_{\Q_0} \left[ (U_t')^{-2} \mathds{1}_A \right] \leq C t^{-3/2} \sum_{k = 0}^\infty \sum_{\ell = 0}^\infty \e^{-2 \beta k} (k + \ell + 1) \e^{-c k^2/t - c \ell^2/t} 
        = O(t^{-1/2}).
    \end{equation*}
    This implies $\Expec{\nu_{\beta, t}([a, 1])} = O(t^{-1/2} \e^{-at/3})$ in the same way as~\eqref{eq:formulation4} implies~\eqref{eq:formulation1}.
    This concludes the proof.
\end{proof}

\subsection{Infinite temperature}

\begin{proof}[Proof of Theorem~\ref{thm:asymptotics_in_mean}.\ref{it:asymptotics_in_mean_1}]
    We consider here the case $\beta = 0$. 
    Using the first expression in~\eqref{eq:overlap_rewriting} with $\beta=0$, and then that $\#\cN(t)$ is geometrically distributed with parameter $\e^{-t}$ (see the first example in Section III.5 of \cite{AthreyaNey1972}), we can write
    \begin{equation*}
    \nu_{0, t}([a, 1]) 
    = \frac{1}{(\#\cN(t))^2} \sum_{u\in\cN(at)} (\# \{ v \in \cN(t) : v \geq u\})^2
    = \frac{\sum_{i=1}^M N_i^2}{\left( \sum_{i=1}^M N_i \right)^2},
    \end{equation*}
    where $M$ is geometrically distributed with parameter $q = \e^{-at}$ and $(N_i)_{i\geq 1}$ is a sequence of independent geometric r.v.\@ with parameter $p = \e^{-(1-a)t}$, independent of $M$.
    Using the representation $x^{-2} = \int_0^\infty u \e^{-ux} \d{u}$ for $x > 0$ together with Fubini's theorem, we get
    \begin{equation*}
    \E \left[ \nu_{0, t}([a, 1]) \right]
    = \int_0^\infty \E \left[ \sum_{i=1}^M N_i^2 \exp \left( - u \sum_{i=1}^M N_i\right) \right] u \d{u}.
    \end{equation*}
    Noting that 
    \begin{equation*}
    \E \left[ \sum_{i=1}^M N_i^2 \exp \left( - u \sum_{i=1}^M N_i\right) \middle| M \right]
    = M \cdot \E \left[ N_1^2 \e^{-uN_1} \right] \cdot \E \left[ \e^{-uN_1} \right]^{M-1}.
    \end{equation*}
    and $\E[M s^{M-1}] = q(1-(1-q)s)^{-2}$ for any $s \in [0,1/(1-q))$, we get
    \begin{equation} \label{eq:beta0_exact_expression}
    \E \left[ \nu_{0, t}([a, 1]) \right]
    = \int_0^\infty
    \frac{q u \E[ N_1^2 \e^{-uN_1}]}{(1-(1-q)\E[\e^{-uN_1}])^2} \d{u}
    = pq \int_0^\infty 
    \frac{u \e^{-u} (1+(1-p)\e^{-u})}{(1-(1-p)\e^{-u})(1-(1-pq)\e^{-u})^2} \d{u},
    \end{equation}
    using $\E[\e^{-uN_1}] = p\e^{-u}/(1-(1-p)\e^{-u})$ and $\E[N^2 \e^{-uN_1}] = p\e^{-u}(1+(1-p)\e^{-u})/(1-(1-p)\e^{-u})^3$.
    We now split the integral on the right-hand side of~\eqref{eq:beta0_exact_expression} in three pieces, the main part being the one from $pq$ to $p$.
    Recall $q = \e^{-at}$ and $p = \e^{-(1-a)t}$.
    
    For the part $u \in [0,pq]$, we use $1-(1-r)\e^{-u} \geq r \e^{-1}$ with $r =p$ and $r=pq$ to get
    \begin{equation} \label{eq:beta0_part1}
    \int_0^{pq}
    \frac{u \e^{-u} (1+(1-p)\e^{-u})}{(1-(1-p)\e^{-u})(1-(1-pq)\e^{-u})^2} \d{u}
    \leq \frac{C}{p (pq)^2} \int_0^{pq} u \d{u}
    \leq \frac{C}{p}.
    \end{equation}
    For the part $u \in [p,\infty)$, we use $1-(1-r)\e^{-u} \geq 1-\e^{-u} \geq (1-\e^{-1})(u\wedge1)$ with $r =p$ and $r=pq$ to get
    \begin{equation} \label{eq:beta0_part2}
    \int_p^\infty
    \frac{u \e^{-u} (1+(1-p)\e^{-u})}{(1-(1-p)\e^{-u})(1-(1-pq)\e^{-u})^2} \d{u}
    \leq \int_p^\infty \frac{Cu \e^{-u}}{(u\wedge1)^3} \d{u}
    \leq \int_p^1 \frac{C}{u^2} \d{u}
    + \int_1^\infty Cu \e^{-u}\d{u}
    \leq \frac{C}{p}.
    \end{equation}
    Finally, we consider the part $u \in [pq,p]$. We have the following asymptotic behavior, as $t\to\infty$, uniformly in $u \in [pq,p]$,
    \begin{align} 
    \frac{u \e^{-u} (1+(1-p)\e^{-u})}{(1-(1-p)\e^{-u})(1-(1-pq)\e^{-u})^2} 
    &= \frac{u (1+O(u)) (2+O(p))}{(p+O(u+p^2))(u+O(pq+u^2))^2} \nonumber \\
    &= \frac{2}{pu} 
    \left( 1 + O \left( p + \frac{u}{p} + \frac{pq}{u} \right) \right), \nonumber
    \end{align}
    which yields 
    \begin{equation} \label{eq:beta0_part3}
    \int_{pq}^p
    \frac{u \e^{-u} (1+(1-p)\e^{-u})}{(1-(1-p)\e^{-u})(1-(1-pq)\e^{-u})^2} \d{u}
    = \frac{2}{p} \left( \log \frac{1}{q} + O \left( p \log \frac{1}{q} + 1 \right) \right).
    \end{equation}
    Combining~\eqref{eq:beta0_part1}, \eqref{eq:beta0_part2}, \eqref{eq:beta0_part3} together with~\eqref{eq:beta0_exact_expression}, we get 
    $\E[\nu_{0, t}([a, 1])] \sim 2q \log \frac{1}{q}$, which concludes the proof, recalling that $q = \e^{-at}$.
\end{proof}

\section*{Acknowledgement}

We warmly thank Pascal Maillard and Bastien Mallein for helpful discussions throughout the project.
We also acknowledge support from the CNRS 80Prime project GEx-MBB, the ANR project MBAP-P (ANR-24-CE40-1833), the MINT fellowship program, and the Réseau Thématique Mathématiques et Physique.
Finally, we are grateful to the anonymous referee for their helpful comments and suggested corrections, which have improved the redaction of this paper.

\bibliographystyle{abbrv}
\bibliography{biblio}

@article {AidekonBerestyckiBrunetShi2013,
	author = {A\"{i}d\'{e}kon, E. and Berestycki, J. and Brunet, \'{E}. and Shi, Z.},
	title = {Branching {B}rownian motion seen from its tip},
	journal = {Probab. Theory Relat. Fields},
	volume = {157},
	year = {2013},
	number = {1-2},
	pages = {405--451},
}

@article {AidekonShi2014,
	author = {A\"{i}d\'{e}kon, E. and Shi, Z.},
	title = {The {S}eneta-{H}eyde scaling for the branching random walk},
	journal = {Ann. Probab.},
	volume = {42},
	year = {2014},
	number = {3},
	pages = {959--993},
}

@article {ArguinBovierKistler2013,
    AUTHOR = {Arguin, Louis-Pierre and Bovier, Anton and Kistler, Nicola},
     TITLE = {The extremal process of branching {B}rownian motion},
   JOURNAL = {Probab. Theory Related Fields},
  FJOURNAL = {Probability Theory and Related Fields},
    VOLUME = {157},
      YEAR = {2013},
    NUMBER = {3-4},
     PAGES = {535--574},
      ISSN = {0178-8051},
   MRCLASS = {60J80 (60G70 60J65)},
  MRNUMBER = {3129797},
MRREVIEWER = {Anthony G. Pakes},
       DOI = {10.1007/s00440-012-0464-x},
}

@book{AthreyaNey1972,
    AUTHOR = {Athreya, Krishna B. and Ney, Peter E.},
     TITLE = {Branching processes},
    SERIES = {Die Grundlehren der mathematischen Wissenschaften, Band 196},
 PUBLISHER = {Springer-Verlag, New York-Heidelberg},
      YEAR = {1972},
     PAGES = {xi+287},
   MRCLASS = {60J80},
  MRNUMBER = {373040},
MRREVIEWER = {C. C. Heyde},
}

@article {BarralRhodesVargas2018,
    AUTHOR = {Barral, J. and Rhodes, R. and Vargas, V.},
     TITLE = {Limiting laws of supercritical branching random walks},
   JOURNAL = {C. R. Math. Acad. Sci. Paris},
    VOLUME = {350},
      YEAR = {2012},
    NUMBER = {9-10},
     PAGES = {535--538},
}

@book {Bertoin1996,
    AUTHOR = {Bertoin, J.},
     TITLE = {L\'{e}vy processes},
    SERIES = {Cambridge Tracts in Mathematics},
    VOLUME = {121},
 PUBLISHER = {Cambridge University Press, Cambridge},
      YEAR = {1996},
     PAGES = {x+265},
}

@article {Biggins1992,
    AUTHOR = {Biggins, J. D.},
     TITLE = {Uniform convergence of martingales in the branching random walk},
   JOURNAL = {Ann. Probab.},
    VOLUME = {20},
      YEAR = {1992},
    NUMBER = {1},
     PAGES = {137--151},
}

@book {Billingsley1999,
    AUTHOR = {Billingsley, P.},
     TITLE = {Convergence of probability measures},
    SERIES = {Wiley Series in Probability and Statistics: Probability and Statistics},
   EDITION = {2},
 PUBLISHER = {John Wiley \& Sons, Inc., New York},
      YEAR = {1999},
     PAGES = {x+277},
}

@book {BolthausenSznitman2002,
  author = {Bolthausen, E. and Sznitman, A.-S.},
  year = {2002},
  pages = {},
  title = {Ten lectures on random media},
  volume = {32},
  publisher = {Birkh\"{a}user Basel},
}

@article {Bonnefont2022,
    AUTHOR = {Bonnefont, B.},
     TITLE = {The overlap distribution at two temperatures for the branching {B}rownian motion},
   JOURNAL = {Electron. J. Probab.},
    VOLUME = {27},
      YEAR = {2022},
     PAGES = {Paper No. 116, 21},
}

@book {BorodinSalminen2002,
    AUTHOR = {Borodin, A. N. and Salminen, P.},
     TITLE = {Handbook of {B}rownian motion---facts and formulae},
    SERIES = {Probability and its Applications},
   EDITION = {Second},
 PUBLISHER = {Birkh\"{a}user Verlag, Basel},
      YEAR = {2002},
}

@book{Bovier2016,
place={Cambridge},
series={Cambridge Studies in Advanced Mathematics},
title={Gaussian processes on trees: from spin glasses to branching {B}rownian motion}, 
publisher={Cambridge University Press}, author={Bovier, A.},
year={2016},
collection={Cambridge Studies in Advanced Mathematics},
}

@article {BovierKurkova2004,
    AUTHOR = {Bovier, A. and Kurkova, I.},
     TITLE = {Derrida's generalized random energy models. {II}. {M}odels with continuous hierarchies},
   JOURNAL = {Ann. Inst. H. Poincar\'{e} Probab. Statist.},
    VOLUME = {40},
      YEAR = {2004},
    NUMBER = {4},
     PAGES = {481--495},
}

@article {Bramson1978,
    AUTHOR = {Bramson, M. D.},
     TITLE = {Maximal displacement of branching {B}rownian motion},
   JOURNAL = {Comm. Pure Appl. Math.},
    VOLUME = {31},
      YEAR = {1978},
    NUMBER = {5},
     PAGES = {531--581},
}

@article {Bramson1983,
    AUTHOR = {Bramson, M. D.},
     TITLE = {Convergence of solutions of the {K}olmogorov equation to travelling waves},
   JOURNAL = {Mem. Amer. Math. Soc.},
    VOLUME = {44},
      YEAR = {1983},
    NUMBER = {285},
     PAGES = {iv+190},
}

@article {BuraczewskiIksanovMallein2021,
    AUTHOR = {Buraczewski, Dariusz and Iksanov, Alexander and Mallein, Bastien},
     TITLE = {On the derivative martingale in a branching random walk},
   JOURNAL = {Ann. Probab.},
  FJOURNAL = {The Annals of Probability},
    VOLUME = {49},
      YEAR = {2021},
    NUMBER = {3},
     PAGES = {1164--1204},
      ISSN = {0091-1798},
   MRCLASS = {60G50 (60F05 60G42 60J80)},
  MRNUMBER = {4255141},
MRREVIEWER = {Vitali Wachtel},
       DOI = {10.1214/20-aop1474},
}

@misc{Chataignier2024,
      title={Additive martingales of the branching {B}rownian motion}, 
      author={Chataignier, L.},
      year={2024},
      note={Master thesis. {\tt arXiv:2407.20227}},
}

@article{Chauvin1991,
  author = {Chauvin, B.},
  title = {Product martingales and stopping lines for branching {B}rownian motion},
  volume = {19},
  journal = {Ann. Probab.},
  number = {3},
  publisher = {Institute of Mathematical Statistics},
  pages = {1195--1205},
  year = {1991},
}

@article {ChauvinRouault1988,
  title={{KPP} equation and supercritical branching {B}rownian motion in the subcritical speed area. Application to spatial trees},
  author={Chauvin, B. and Rouault, A.},
  journal={Probab. Theory Relat. Fields},
  year={1988},
  volume={80},
  pages={299--314},
}

@incollection {ChauvinRouault1997,
    AUTHOR = {Chauvin, B. and Rouault, A.},
     TITLE = {Boltzmann-{G}ibbs weights in the branching random walk},
 BOOKTITLE = {Classical and modern branching processes ({M}inneapolis, {MN}, 1994)},
    SERIES = {IMA Vol. Math. Appl.},
    VOLUME = {84},
     PAGES = {41--50},
 PUBLISHER = {Springer, New York},
      YEAR = {1997},
   MRCLASS = {60J80 (60F10 60K35 82C41)},
  MRNUMBER = {1601693},
MRREVIEWER = {Rinaldo Schinazi},
       DOI = {10.1007/978-1-4612-1862-3\_3},
}

@article {ChenMadauleMallein2019,
    AUTHOR = {Chen, X. and Madaule, T. and Mallein, B.},
     TITLE = {On the trajectory of an individual chosen according to supercritical {G}ibbs measure in the branching random walk},
   JOURNAL = {Stoch. Process. Their Appl.},
    VOLUME = {129},
      YEAR = {2019},
    NUMBER = {10},
     PAGES = {3821--3858},
}

@article {CortinesHartungLouidor2019,
    AUTHOR = {Cortines, A. and Hartung, L. and Louidor, O.},
     TITLE = {The structure of extreme level sets in branching {B}rownian motion},
   JOURNAL = {Ann. Probab.},
    VOLUME = {47},
      YEAR = {2019},
    NUMBER = {4},
     PAGES = {2257--2302},
}

@article{DerridaGardner1986,
  title = {Solution of the generalised random energy model},
  author = {Derrida, B. and Gardner, E.},
  journal = {J. Phys. C: Solid State Phys.},
  volume = {19},
  issue = {2},
  pages = {2253--2274},
  year = {1986},
}

@article{DerridaMottishaw2016,
year = {2016},
publisher = {{IOP} Publishing},
volume = {115},
number = {4},
pages = {40005},
author = {Derrida, B. and Mottishaw, P.},
title = {On the genealogy of branching random walks and of directed polymers},
journal = {{EPL} (Europhysics Letters)},
}

@article {DerridaMottishaw2018,
    AUTHOR = {Derrida, B. and Mottishaw, P.},
     TITLE = {Finite size corrections to the {P}arisi overlap function in the {GREM}},
   JOURNAL = {J. Stat. Phys.},
    VOLUME = {172},
      YEAR = {2018},
    NUMBER = {2},
     PAGES = {592--610},
}

@article{DerridaSpohn1988,
    AUTHOR = {Derrida, B. and Spohn, H.},
     TITLE = {Polymers on disordered trees, spin glasses, and traveling waves},
   JOURNAL = {J. Stat. Phys.},
    VOLUME = {51},
      YEAR = {1988},
    NUMBER = {5-6},
     PAGES = {817--840},
}

@incollection {HardyHarris2006,
    AUTHOR = {Hardy, Robert and Harris, Simon C.},
     TITLE = {A spine approach to branching diffusions with applications to {$\mathscr{L}^p$}-convergence of martingales},
 BOOKTITLE = {S\'{e}minaire de {P}robabilit\'{e}s {XLII}},
    SERIES = {Lecture Notes in Math.},
    VOLUME = {1979},
     PAGES = {281--330},
 PUBLISHER = {Springer, Berlin},
      YEAR = {2009},
   MRCLASS = {60J80 (60F25 60G44 60J60)},
  MRNUMBER = {2599214},
MRREVIEWER = {Jos\'{e} Villa-Morales},
       DOI = {10.1007/978-3-642-01763-6\_11},
}

@article {HarrisRoberts2017,
author = {Harris, S. C. and Roberts, M. I.},
title = {The many-to-few lemma and multiple spines},
volume = {53},
journal = {Ann. Inst. H. Poincar\'{e} Probab. Statist.},
number = {1},
pages = {226--242},
year = {2017},
}

@article {HartungKlimovsky2015,
    AUTHOR = {Hartung, Lisa and Klimovsky, Anton},
     TITLE = {The glassy phase of the complex branching {B}rownian motion energy model},
   JOURNAL = {Electron. Commun. Probab.},
  FJOURNAL = {Electronic Communications in Probability},
    VOLUME = {20},
      YEAR = {2015},
     PAGES = {no. 78, 15},
   MRCLASS = {60J80 (60F05 60G70 60K35 82B44)},
  MRNUMBER = {3417450},
       DOI = {10.1214/ECP.v20-4360},
}

@article {HartungKlimovsky2018,
    AUTHOR = {Hartung, Lisa and Klimovsky, Anton},
     TITLE = {The phase diagram of the complex branching {B}rownian motion energy model},
   JOURNAL = {Electron. J. Probab.},
  FJOURNAL = {Electronic Journal of Probability},
    VOLUME = {23},
      YEAR = {2018},
     PAGES = {Paper No. 127, 27},
   MRCLASS = {60J80 (60F05 60J65 60K35 82B26 82B44)},
  MRNUMBER = {3896864},
       DOI = {10.1214/18-EJP245},
}

@article {HouRenSong2024,
    AUTHOR = {Hou, Haojie and Ren, Yan-Xia and Song, Renming},
     TITLE = {1-stable fluctuation of the derivative martingale of branching random walk},
   JOURNAL = {Stochastic Process. Appl.},
  FJOURNAL = {Stochastic Processes and their Applications},
    VOLUME = {172},
      YEAR = {2024},
     PAGES = {Paper No. 104338, 32},
      ISSN = {0304-4149},
   MRCLASS = {60J80 (60F05 60G42)},
  MRNUMBER = {4718367},
MRREVIEWER = {Krzysztof Joachim Bartoszek},
       DOI = {10.1016/j.spa.2024.104338},
}

@article{IancuMuellerMunier2005,
title = {Universal behavior of {QCD} amplitudes at high energy from general tools of statistical physics},
journal = {Phys. Lett. B},
volume = {606},
number = {3},
pages = {342-350},
year = {2005},
author = {Iancu, E. and Mueller, A. H. and Munier, S.},
}

@article {IksanovKabluchko2016,
    AUTHOR = {Iksanov, Alexander and Kabluchko, Zakhar},
     TITLE = {A central limit theorem and a law of the iterated logarithm for the {B}iggins martingale of the supercritical branching random walk},
   JOURNAL = {J. Appl. Probab.},
  FJOURNAL = {Journal of Applied Probability},
    VOLUME = {53},
      YEAR = {2016},
    NUMBER = {4},
     PAGES = {1178--1192},
      ISSN = {0021-9002},
   MRCLASS = {60G42 (60J80)},
  MRNUMBER = {3581250},
MRREVIEWER = {Ger\'{o}nimo Uribe Bravo},
       DOI = {10.1017/jpr.2016.73},
}

@article {IksanovKoleskoMeiners2020,
    AUTHOR = {Iksanov, A. and Kolesko, K. and Meiners, M.},
     TITLE = {Fluctuations of {B}iggins' martingales at complex parameters},
   JOURNAL = {Ann. Inst. H. Poincar\'{e} Probab. Statist.},
    VOLUME = {56},
      YEAR = {2020},
    NUMBER = {4},
     PAGES = {2445--2479},
}

@article {KahanePeyriere1976,
    AUTHOR = {Kahane, J.-P. and Peyri\`ere, J.},
     TITLE = {Sur certaines martingales de {B}enoit {M}andelbrot},
   JOURNAL = {Adv. Math.},
    VOLUME = {22},
      YEAR = {1976},
    NUMBER = {2},
     PAGES = {131--145},
}

@article{KimLubetzkyZeitouni2023,
    AUTHOR = {Kim, Yujin H. and Lubetzky, Eyal and Zeitouni, Ofer},
     TITLE = {The maximum of branching {B}rownian motion in {$\mathbb{R}^d$}},
   JOURNAL = {Ann. Appl. Probab.},
  FJOURNAL = {The Annals of Applied Probability},
    VOLUME = {33},
      YEAR = {2023},
    NUMBER = {2},
     PAGES = {1315--1368},
      ISSN = {1050-5164},
   MRCLASS = {60J80 (60J65 60J70)},
  MRNUMBER = {4564433},
MRREVIEWER = {Bastien Mallein},
       DOI = {10.1214/22-aap1848},
}

@article {Kyprianou2004,
    AUTHOR = {Kyprianou, A. E.},
     TITLE = {Travelling wave solutions to the {K}-{P}-{P} equation:
              alternatives to {S}imon {H}arris' probabilistic analysis},
   JOURNAL = {Ann. Inst. H. Poincar\'{e} Probab. Statist.},
    VOLUME = {40},
      YEAR = {2004},
    NUMBER = {1},
     PAGES = {53--72},
}

@article {LalleySellke1987,
    AUTHOR = {Lalley, S. P. and Sellke, T.},
     TITLE = {A conditional limit theorem for the frontier of a branching {B}rownian motion},
   JOURNAL = {Ann. Probab.},
    VOLUME = {15},
      YEAR = {1987},
    NUMBER = {3},
     PAGES = {1052--1061},
}

@article {Liu2001,
    AUTHOR = {Liu, Quansheng},
     TITLE = {Asymptotic properties and absolute continuity of laws stable by random weighted mean},
   JOURNAL = {Stochastic Process. Appl.},
  FJOURNAL = {Stochastic Processes and their Applications},
    VOLUME = {95},
      YEAR = {2001},
    NUMBER = {1},
     PAGES = {83--107},
      ISSN = {0304-4149},
   MRCLASS = {60J80 (60E10 60F05 60G30 60G42 60G57 60K35)},
  MRNUMBER = {1847093},
MRREVIEWER = {M. P. Quine},
       DOI = {10.1016/S0304-4149(01)00092-8},
}

@article{Liu2002,
    AUTHOR = {Liu, Quansheng},
     TITLE = {An extension of a functional equation of {P}oincar\'{e} and {M}andelbrot},
   JOURNAL = {Asian J. Math.},
  FJOURNAL = {Asian Journal of Mathematics},
    VOLUME = {6},
      YEAR = {2002},
    NUMBER = {1},
     PAGES = {145--168},
      ISSN = {1093-6106},
   MRCLASS = {60E05},
  MRNUMBER = {1902651},
       DOI = {10.4310/AJM.2002.v6.n1.a8},
}

@incollection {Lyons1997,
    AUTHOR = {Lyons, Russell},
     TITLE = {A simple path to {B}iggins' martingale convergence for
              branching random walk},
 BOOKTITLE = {Classical and modern branching processes ({M}inneapolis, {MN},
              1994)},
    SERIES = {IMA Vol. Math. Appl.},
    VOLUME = {84},
     PAGES = {217--221},
 PUBLISHER = {Springer, New York},
      YEAR = {1997},
      ISBN = {0-387-94872-4},
   MRCLASS = {60J80},
  MRNUMBER = {1601749},
       DOI = {10.1007/978-1-4612-1862-3\_17},
}

@article {Madaule2016,
    AUTHOR = {Madaule, T.},
     TITLE = {First order transition for the branching random walk at the critical parameter},
   JOURNAL = {Stoch. Process. Their Appl.},
    VOLUME = {126},
      YEAR = {2016},
    NUMBER = {2},
     PAGES = {470--502},
}

@article {MadauleRhodesVargas2015,
    AUTHOR = {Madaule, Thomas and Rhodes, R\'{e}mi and Vargas, Vincent},
     TITLE = {The glassy phase of complex branching {B}rownian motion},
   JOURNAL = {Comm. Math. Phys.},
  FJOURNAL = {Communications in Mathematical Physics},
    VOLUME = {334},
      YEAR = {2015},
    NUMBER = {3},
     PAGES = {1157--1187},
      ISSN = {0010-3616},
   MRCLASS = {60G57 (60J80)},
  MRNUMBER = {3312433},
MRREVIEWER = {Yueyun Hu},
       DOI = {10.1007/s00220-014-2257-9},
}

@article {MaillardPain2019,
    AUTHOR = {Maillard, P. and Pain, M.},
    TITLE = {1-stable fluctuations in branching {B}rownian motion at critical temperature {I}: the derivative martingale},
    JOURNAL = {Ann. Probab.},
    VOLUME = {47},
    YEAR = {2019},
    NUMBER = {5},
    PAGES = {2953--3002},
}

@misc {MaillardPain2021,
      title={1-stable fluctuations in branching {B}rownian motion at critical temperature {II}: general functionals}, 
      author={Maillard, P. and Pain, M.},
      year={2021},
      note={{\tt arXiv:2103.10412}}
}

@article {Mallein2018,
    AUTHOR = {Mallein, B.},
     TITLE = {Genealogy of the extremal process of the branching random walk},
   JOURNAL = {ALEA Lat. Am. J. Probab. Math. Stat.},
    VOLUME = {15},
      YEAR = {2018},
    NUMBER = {2},
     PAGES = {1065--1087},
}

@article {McKean1975,
    AUTHOR = {McKean, H. P.},
     TITLE = {Application of {B}rownian motion to the equation of {K}olmogorov-{P}etrovskii-{P}iskunov},
   JOURNAL = {Comm. Pure Appl. Math.},
    VOLUME = {28},
      YEAR = {1975},
    NUMBER = {3},
     PAGES = {323--331},
}

@article {MezardParisiSourlasToulouseVirasoro1984,
    AUTHOR = {M\'{e}zard, M. and Parisi, G. and Sourlas, N. and Toulouse, G. and Virasoro, M.},
     TITLE = {Replica symmetry breaking and the nature of the spin glass phase},
   JOURNAL = {J. Phys.},
    VOLUME = {45},
      YEAR = {1984},
    NUMBER = {5},
     PAGES = {843--854},
}

@article {MuellerMunier2018a,
    AUTHOR = {Mueller, A. H. and Munier, S.},
     TITLE = {Diffractive electron-nucleus scattering and ancestry in branching random walks},
   JOURNAL = {Phys. Rev. Lett.},
    VOLUME = {121},
      YEAR = {2018},
     PAGES = {082001},
}

@article{MuellerMunier2018b,
  title = {Rapidity gap distribution in diffractive deep-inelastic scattering and parton genealogy},
  author = {Mueller, A. H. and Munier, S.},
  journal = {Phys. Rev. D},
  volume = {98},
  issue = {3},
  pages = {034021},
  numpages = {14},
  year = {2018},
  publisher = {American Physical Society},
}

@article{Munier2009,
title = {Quantum chromodynamics at high energy and statistical physics},
journal = {Phys. Rep.},
volume = {473},
number = {1},
pages = {1-49},
year = {2009},
author = {Munier, S.},
}

@incollection {Neveu1988,
    AUTHOR = {Neveu, J.},
     TITLE = {Multiplicative martingales for spatial branching processes},
 BOOKTITLE = {Seminar on {S}tochastic {P}rocesses, 1987 ({P}rinceton, {NJ},
              1987)},
    SERIES = {Progr. Probab. Statist.},
    VOLUME = {15},
     PAGES = {223--242},
 PUBLISHER = {Birkh\"{a}user Boston, Boston, MA},
      YEAR = {1988},
}

@article {Pain2018,
    AUTHOR = {Pain, M.},
     TITLE = {The near-critical {G}ibbs measure of the branching random walk},
   JOURNAL = {Ann. Inst. H. Poincar\'{e} Probab. Statist.},
    VOLUME = {54},
      YEAR = {2018},
    NUMBER = {3},
     PAGES = {1622--1666},
}

@article {Parisi1983,
  title = {Order parameter for spin-glasses},
  author = {Parisi, G.},
  journal = {Phys. Rev. Lett.},
  volume = {50},
  issue = {24},
  pages = {1946--1948},
  numpages = {0},
  year = {1983},
  publisher = {American Physical Society},
}

@book {Sato1999,
    AUTHOR = {Sato, K.},
     TITLE = {L\'{e}vy processes and infinitely divisible distributions},
    SERIES = {Cambridge Studies in Advanced Mathematics},
    VOLUME = {68},
      NOTE = {Translated from the 1990 Japanese original,
              Revised by the author},
 PUBLISHER = {Cambridge University Press, Cambridge},
      YEAR = {1999},
     PAGES = {xii+486},
}

@article{VonBahrEsseen1965,
author = {von Bahr, B. and Esseen, C.-G.},
title = {Inequalities for the $r$th absolute moment of a sum of random variables, $1 \leqq r \leqq 2$},
volume = {36},
journal = {Ann. Math. Stat.},
number = {1},
publisher = {Institute of Mathematical Statistics},
pages = {299--303},
year = {1965},
}

@incollection {Yor1997,
    AUTHOR = {Yor, Marc},
     TITLE = {Some remarks about the joint law of {B}rownian motion and its
              supremum},
 BOOKTITLE = {S\'eminaire de {P}robabilit\'es, {XXXI}},
    SERIES = {Lecture Notes in Math.},
    VOLUME = {1655},
     PAGES = {306--314},
 PUBLISHER = {Springer, Berlin},
      YEAR = {1997},
      ISBN = {3-540-62634-4},
   MRCLASS = {60J65},
  MRNUMBER = {1478739},
MRREVIEWER = {Paul\ Embrechts},
       DOI = {10.1007/BFb0119315},
}

\end{document}